\documentclass[reqno, 11pt]{amsart}
\usepackage{amssymb, amsmath,cite}
\usepackage{graphics}
\usepackage{epsfig}
\usepackage{amsxtra}
\newtheorem{thm}{Theorem}[section]

\newtheorem{lemma}[thm]{Lemma}

\newtheorem{prop}[thm]{Proposition}

\theoremstyle{definition}
\newtheorem{rmk}{Remark}[section]

\numberwithin{equation}{section}
\renewcommand{\Re}{\operatorname{Re}}
\renewcommand{\Im}{\operatorname{Im}}

\usepackage[letterpaper,hmargin=1in,vmargin=1.2in]{geometry}
\usepackage{amsmath, amssymb, amsthm, verbatim, url}
\usepackage{graphicx}
\usepackage{enumerate, pinlabel}
\usepackage{amsmath,amssymb,amsthm,mathrsfs,graphicx,url}
\usepackage[usenames,dvipsnames]{color}
\usepackage[colorlinks=true,linkcolor=Black,citecolor=Black,urlcolor=Black]{hyperref}

\def \Real {{\mathbb R}}
\def \Integers {{\mathbb Z}}
\def \Complex {{\mathbb C}}
\def \Natural {{\mathbb N}}
\def \Sphere {{\mathbb S}}
\def \supp {\operatorname{supp}}
\def \mch {\mathcal{H}}

\def \mchs {\mch \oplus \mch}
\def \lp {\lambda'}
\def \NY {N_Y}
\def \cY {c_Y}

\def\bbbone{{\mathchoice {1\mskip-4mu {\rm{l}}} {1\mskip-4mu {\rm{l}}}
{ 1\mskip-4.5mu {\rm{l}}} { 1\mskip-5mu {\rm{l}}}}}

\def \mcd {\mathcal{D}}
\def \restrict {\!\!\upharpoonright}

\def \tS {S}
\def \gef {\Phi}
\def \zhat {\hat{Z}}
\def \zhath {\hat{Z}_h}

\def \xinf{X_\infty}

\def \Cs {\Complex_{\operatorname{slit}}}
\def \mcx {\mathcal{X}}
\def \defeq {:=}
\def \tN {\tilde{N}}
\def \psiso {\psi_{sp,1}}

\def \tchi {\tilde{\chi}}

\DeclareMathAlphabet{\mathpzc}{OT1}{pzc}{m}{it}
\def \pr {\mathpzc{p}}
\def \spec {\operatorname{spec}}
\def \Pre {\mathcal{P}}
\def \PV {\operatorname{PV}}
\def \ev {E}

\title[Wave  asymptotics]{Wave  asymptotics for waveguides and manifolds with infinite cylindrical ends}
\author{T.J. Christiansen and K. Datchev}
\address{Department of Mathematics, University of Missouri, Columbia, MO 65211 USA}
\email{christiansent@missouri.edu}
\address{Department of Mathematics, Purdue University, West Lafayette, IN 47907 USA}
\email{kdatchev@purdue.edu}
\begin{document}
\begin{abstract}
We describe wave decay rates associated to embedded resonances and spectral thresholds for 
waveguides and  manifolds with infinite cylindrical ends. We show that if the cut-off resolvent is polynomially 
bounded at high energies, as is the case in certain favorable geometries, then there is an 
associated asymptotic expansion, up to a $O(t^{-k_0})$ remainder, of solutions of the wave equation 
on compact sets as $t \to \infty$. In the most general such case we have $k_0=1$, and under an 
additional assumption on the infinite ends we have $k_0 = \infty$.  If we localize the solutions to the wave equation in frequency as well as in space, then our results hold for quite general waveguides and manifolds with infinite cylindrical ends.

To treat problems with and without boundary in a unified way, we introduce a black
box framework analogous to the Euclidean one of Sj\"ostrand and Zworski. We study the resolvent, generalized eigenfunctions, spectral measure, and spectral thresholds in this framework, providing a new approach to some
mostly 
 well-known results in the scattering theory of manifolds with cylindrical ends.

\end{abstract}
\maketitle

\section{Introduction}
Wave decay rates on a  manifold of infinite volume can be related to the
 geometry of the manifold via the 
behavior of the resolvent $(-\Delta - z)^{-1}$ in the vicinity of the spectrum.
A particularly important and long-studied class of problems
is that of compactly supported perturbations
of Euclidean space,  an example of which is the classical obstacle 
scattering problem. In that 
case the dominant contributions to wave decay rates come from the 
resolvent behavior near the only threshold in the spectrum, $z=0$, 
 and as $\Re z \to + \infty$.  
We think of the former as being 
related to the geometric infinity--here the spatial dimension is especially important.
The 
latter typically
 reflects the dynamics of the compact, and possibly empty,
 set of trapped geodesics.  In 
particular, in this setting we can separate the contributions of the 
geometric infinity and the 
trapped set. Similar results hold in many situations where infinity is 
``large" in a suitable sense, 
such as on asymptotically Euclidean, conic, and hyperbolic manifolds.

In this paper we consider manifolds which   are isometric to a 
 cylinder
 $(0, \infty) \times Y$ (with $Y$ compact) outside of a compact
set.  
 Thus we cannot separate the geometric infinity and the trapped geodesics, 
since 
the latter occur outside of arbitrarily large compact sets.
Also in contrast with the case of obstacle scattering is the relatively complicated nature of the spectrum of the Laplacian.
The continuous spectrum has infinitely
many thresholds, given by the eigenvalues of the Laplacian on $Y$, where the multiplicity increases.
In addition, there may be 
 up to infinitely many embedded resonances and eigenvalues.

A motivation for the study of such manifolds comes from waveguides and quantum
dots connected to leads. These appear in certain models of
electron motion in semiconductors and of propagation of electromagnetic and sound waves.  We give just a few pointers to the physics and applied math literature here  \cite{lcm,rai,rbbh,EK,bgw}.

Our main results concern manifolds with infinite cylindrical ends for which 
the resolvent
is well-behaved at high energies,
 which is the case in some 
favorable geometric situations discussed below.  In this case we can compute 
asymptotics of the wave equation 
in terms of the features of the spectrum discussed above. Roughly speaking,
a resonance at zero or  
any eigenvalues  
contribute non-decaying terms; any other embedded resonances, 
which can occur only at thresholds,
contribute terms decaying like $t^{-1/2-k}$, with
$k\in \Natural_0$; and non-resonant thresholds contribute terms decaying 
like $t^{-3/2-k}$. 

More specifically, 
in this paper we study asymptotic expansions as $t \to \infty$ of
 solutions to the wave equation 
\begin{equation}\label{e:cylcauchy}
(\partial_t^2 - \Delta ) u(t) = 0, \quad
u(0) = f_1, \quad \partial_t u(0) = f_2,
\end{equation}
where $\Delta\le 0$ is the Laplacian on a suitable Riemannian manifold $(X,g)$ with infinite cylindrical ends, 
and $f_1$ and $f_2$ are suitable initial conditions.   

 Our main results allow us to replace $-\Delta$ with a  more general 
self-adjoint operator $H$ but require an assumption on the 
high energy behavior of the cut-off resolvent of $-\Delta$ or $H$.  The companion paper 
\cite{ch-da} gives a technique of constructing manifolds 
with infinite cylindrical ends so that such estimates hold for the resolvent of the Laplacian, or in fact for the resolvent of many Schr\"odinger
 operators; see
Sections \ref{ex:non1} and \ref{ex:3fun1} of this paper for 
examples.  The paper \cite{ch-da3} studies the Dirichlet Laplacian on 
domains in $\Real^d$ which are ``star-shaped'' in a certain sense,
and shows that such estimates hold for the high-energy
resolvent; see Section \ref{ss:waveguides} for examples of domains satisfying
these conditions.
 However, if we apply a spectral cut-off to our solution
$u$ of (\ref{e:cylcauchy}), the hypothesis on the 
high energy resolvent is unnecessary--see Proposition \ref{p:speccutoffexp}.

Our starting point is the following elementary result for the wave equation on a Riemannian product. We will see below that many aspects of this result carry over to a range of more complicated geometries.

\subsection{The wave equation on a half cylinder}\label{s:hprod}
Let $(X,g) = ((0,\infty)_r \times  Y_y, dr^2 + g_Y)$, where $(Y,g_Y)$ is a compact Riemannian manifold without boundary. Let $\Delta_Y\le 0$ be the Laplacian on $(Y,g_Y)$, and let $\{\phi_j\}_{j=0}^\infty$ be a complete orthonormal set of eigenfunctions of $\Delta_Y$, with $-\Delta_Y \phi_j = \sigma_j^2 \phi_j$,  $0 = \sigma_0 \le \sigma_1 \le \cdots$. 

Let $f_1,\; f_2\in C_c^\infty(X)$. 
We shall consider solutions $u_D$ and $u_N$, satisfying
Dirichlet and Neumann boundary conditions
respectively,  to (\ref{e:cylcauchy}) on $X$.
Then we can solve \eqref{e:cylcauchy} by separating variables,  writing
\begin{equation}\label{eq:sepvarh}
 f_\ell = f_\ell(r,y) =  \sum_{j =0}^\infty f_{\ell,j}(r) \phi_j(y), \qquad 
u_B(t) = u_B(t,r,y) = \sum_{j =0}^\infty u_{j,B}(t,r) \phi_j(y),
\end{equation}
where $\ell \in \{1,\ 2\}$ and $B$ denotes the boundary condition
``D'' or ``N.''  We can perhaps most easily solve these initial value
problems by extending $f_1,\;f_2$ to be odd (Dirichlet) or even (Neumann) 
functions on $\Real\times Y$, and solving the wave equation
on the full cylinder.  In doing so, we see that we get different
expressions for the terms with $\sigma_j=0$ and those with
$\sigma_j\not =0$.  Suppose $f_l(r,y)=0$ if $r>R_1$, $l=1,\;2$.
Then if $\sigma_j=0$
\begin{equation}\label{e:u0asyh}
u_{j,N}(t,r) = \int_0^\infty f_{2,j}(r')dr',\qquad u_{j,D}(t)=0   \; \textrm{ if } \sigma_j=0,\; \text{for $0<r<R$ and $t>R+R_1$ }
\end{equation}
by d'Alembert's formula. 
On the other hand, if $\sigma_j>0$ for  
 any nonnegative integer $k_0$ we have for $t$ sufficiently large
\begin{equation}\label{e:ujasyh}
 u_{j,B}(t,r) =  \sum_{k=0}^{k_0}  \left[p_{j,k,B}(r)\cos(\sigma_j t + \tfrac \pi 4)  + q_{j,k,B}(r)\sin(\sigma_j t + \tfrac \pi 4)\right]t^{-k- \frac 12}  + O\left(t^{-k_0 - \frac 32}\right),  \quad \textrm{ if } \sigma_j>0,
\end{equation}
by the method of stationary phase (see also \cite{horklein} for asymptotics as $r \to \infty$).
Here  for each $j$ $p_{j,k,B}$ and $q_{j,k,B}$ are polynomials in $r$ of degree at most $2k$, 
and the remainders are uniform in $j$ and as $r$ varies in a compact set. Moreover, for the Neumann boundary condition
\begin{equation}\label{e:pj0qj0h}
 p_{j,0,N} =2 \sqrt{\frac{\sigma_j}{2\pi}}\int_0^\infty f_{1,j}(r')dr', \qquad  q_{j,0,N} = \frac{2}{\sqrt{2\pi\sigma_j}}\int_0^\infty f_{2,j}(r')dr'.
\end{equation}
With the Dirichlet boundary condition, $p_{j,0,D}=0=q_{j,0,D}$, and
\begin{equation}\label{e:pj1qj1h}
 p_{j,1,D}(r)= 2r \sqrt{\frac{\sigma_j}{2\pi}}\int_0^\infty r'f_{2,j}(r')dr', \qquad q_{j,1,D}(r) = 2\sigma_j r \sqrt{\frac{\sigma_j}{2\pi}} \int_0^\infty r'f_{1,j}(r')dr'.
\end{equation}

To interpret the above result in terms of the spectrum, we can similarly write the resolvent of $-\Delta_B$ in terms of shifted resolvents of 
$-(\partial_r^2)_B$.  If $f(r,y)=\sum_{j=0}^\infty f_j(r)\phi_j(y)$, then
\[
 (-\Delta_B - z)^{-1} f =
\sum_{j=0}^\infty \left(\left(-(\partial_r^2)_B + \sigma_j^2 - z\right)^{-1}f_j\right)\phi_j, \qquad z \in \mathbb C \setminus [0,\infty).
\]
From this we see that the spectrum of $-\Delta_B$ is $[0,\infty)$, is purely absolutely continuous, and has thresholds (points at which multiplicity jumps) at the eigenvalues of $-\Delta_Y$.  For the Neumann Laplacian on
the half cylinder, at 
each threshold the spectrum contains an embedded resonance of
multiplicity equal to the  multiplicity of the corresponding eigenvalue of $-\Delta_Y$ (and there are no other resonances embedded in the spectrum).  The
Dirichlet Laplacian on the half-cylinder has no embedded resonances.

We now see that the coefficient of the constant term in the expansion \eqref{e:u0asyh}, given by $u_{j,B}(r)$ with $\sigma_j=0$
as in \eqref{e:pj0qj0h}, is determined by 
the resonant states at zero and a (distributional) pairing with 
the initial data.
The number of states is equal to the number of connected components of $Y$.
The terms of order $t^{-1/2}$ in \eqref{e:ujasyh} have coefficients given in \eqref{e:pj0qj0h} in
terms of resonant states
at nonzero eigenvalues of $-\Delta_Y$ 
and (distributional) pairings of these with the initial data.  
 These terms are $zero$ in the absence of threshold
resonances, as is the case for the Dirichlet half-cylinder. 
From the example of the Dirichlet half-cylinder, we see that in
the absence of eigenvalues or 
resonances embedded in the continuous spectrum we may
have a wave decay like $O(t^{-3/2})$, but we cannot in general expect faster
decay.

In the remainder of the paper we adapt the asymptotics above, in a somewhat weaker form, to Schr\"odinger operators on more general manifolds with cylindrical ends. One difficulty is that such operators can have much nastier behavior of the resolvent near the continuous spectrum, including the possible presence of infinitely many embedded eigenvalues, accumulating at infinity, \cite{Ch-Zw95, Pa95}.  
In additional to the possible existence of resonances embedded in the continuous spectrum, there may be additional ``complex'' resonances,
resonances which lie on the Riemann surface $\zhat$ to 
which the resolvent continues; see Section \ref{ss:r}.
 In general settings
we cannot rule out the possibility that at high energy these rapidly approach
the continuous spectrum.
Below we mostly 
restrict our attention to some particular cases in which the resolvent is better behaved.  In these cases, the non-real resonances do not approach the 
continuous spectrum too rapidly, and they do not appear in the expansions.

\subsection{Two term asymptotics for
 mildly trapping manifolds 
or waveguides
with cylindrical ends}\label{s:two} Our first extension of the results of 
Section \ref{s:hprod} is to manifolds with infinite cylindrical ends for which we have polynomial bounds on the cut-off
resolvent.  Rather than state the theorem in  full generality here,  for 
now we let $(X,g)$ 
be one of the examples in Sections \ref{ex:non1} and \ref{ex:3fun1}.
Alternatively, we allow the waveguides 
$X\subset \Real^d$ to be one of the domains 
in Section \ref{ss:waveguides}, in which 
case we consider the case of Dirichlet boundary
conditions.  In each case, $Y$ is the cross section of the infinite cylindrical
ends, $\Delta_Y\leq 0$ is the Laplacian on $Y$ (with Dirichlet boundary
conditions for the domains of Section
\ref{ss:waveguides}), and $0\leq \sigma_0^2\leq \sigma_1^2\leq \cdots$ are the eigenvalues of $-\Delta_Y$, repeated with 
multiplicity.

By the results of \cite[Theorem 1.1]{ch-da} for the 
examples of Sections \ref{ex:non1} and \ref{ex:3fun1}  and \cite{ch-da3}
for the examples of Section \ref{ss:waveguides},
  $-\Delta$ has 
a finite, possibly empty, set of eigenvalues which we denote $\{\ev_\ell\}$
(repeated with multiplicity), with corresponding orthonormal 
$L^2$ eigenfunctions $\{\eta_\ell\}$: $-\Delta \eta_\ell =\ev_\ell\eta_\ell$.
  The generalized eigenfunctions
  $\{\gef_j\}$ are defined in (\ref{eq:gef});
each $\gef_j$  depends on a parameter $\lambda$ and a variable $x\in X$,
so that $\gef_j(\lambda)=\gef_j(\lambda, x)$.
The $\gef_j(\lambda)$  are (certain) elements of the null space of $-\Delta-\lambda^2$
which are not in $L^2(X)$.
 If $\gef_j(\sigma_j)\not \equiv 0$ then $\sigma_j$ corresponds to a threshold 
resonance and $\gef_j(\sigma_j)$ is a resonant state; 
$\gef_j(\sigma_j)\in C^\infty(X)$.  See Section
\ref{ss:gef} for a  more detailed
discussion of the generalized eigenfunctions. 

Our first result is a two term expansion, an example of which is the 
following theorem.  In the statement of the theorem, 
for $f_j\in C_c^\infty(X)$ and $v\in C^\infty(X)$, $\langle f_j,v\rangle =\int_X f_j \overline{v}$.
\begin{thm}\label{thm:t1intro} Let  $(X,g)$ be as in the examples in 
Sections \ref{ex:non1} or \ref{ex:3fun1}, or let $X\subset \Real^2$ be a domain
as in Section \ref{ss:waveguides}.  Let
 $ \ f_1, \ f_2 \in C_c^\infty(X)$ be given
and $u(t)$ solve \eqref{e:cylcauchy}, and in addition satisfy Dirichlet
boundary conditions if $X$ is as in Section 
\ref{ss:waveguides}. Then we can write
\begin{equation}\label{e:udectwo}
 u(t) = u_e(t) + u_{thr}(t) + u_r(t), 
\end{equation}
where
\begin{equation}\label{eq:ue1}
 u_e(t)=u_e(t,x) = \sum_{\ev_\ell\in \spec_{p}(-\Delta)} \eta_\ell(x) \left( \cos((\ev_\ell)^{1/2} t) \langle f_1,\eta_\ell\rangle_ + 
 \frac{\sin( (\ev_\ell)^{1/2} t)}{(\ev_\ell)^{1/2} } \langle f_2,\eta_\ell\rangle \right)
\end{equation}
and
\begin{multline}\label{eq:uthr1}
 u_{thr}(t)=u_{thr}(t,x) = \frac{1}{4} \sum_{\sigma_j=0} \gef_j(0,x)\langle f_2,\gef_j(0)\rangle \\ + 
\frac{1}{2\sqrt{t}}\sum_{\sigma_j>0} \sqrt{\frac{ \sigma_j}{2\pi}}\cos(\sigma_j t+\pi/4) \gef_j
(\sigma_j,x) \langle f_1,\gef_j(\sigma_j) \rangle \\
+
\frac{1}{2\sqrt{t}}\sum _{\sigma_j>0}  \frac{1}{\sqrt{2\pi \sigma_j}} \sin(\sigma_j t +\pi/4) \gef_j(\sigma_j,x) \langle f_2,\gef_j(\sigma_j)\rangle
\end{multline}
where $x\in X$.
Moreover, for any $\chi\in C_c^\infty(X)$, there  is a constant $C$ so that the remainder, $u_r(t)$, satisfies
\begin{equation}\label{eq:ur1}
 \|\chi u_r(t) \|_{L^2(X)} \le C t^{-1},\; \text{for $t$ sufficiently large}.
\end{equation}
\end{thm} 
For the manifolds and domains considered here, as a consequence
of \cite[Theorem 1.1]{ch-da} and \cite[Theorem 3]{ch-da3}, each of these sums 
in (\ref{eq:ue1}) and (\ref{eq:uthr1}) is in fact a finite sum.  
Since we have assumed the initial 
data $f_1, \;f_2\in C_c^\infty(X)$, we could replace
the bound (\ref{eq:ur1}) by  
$\|\chi u_r(t) \|_{H^m(X)}\leq Ct^{-1}$ for any $m\in \Natural$,
 with a new constant $C$ depending on $m$.

We
compare the expansion of $u$ in Theorem \ref{thm:t1intro} with that of the 
solution to the wave equation on the Neumann half cylinder given by
 (\ref{eq:sepvarh}-\ref{e:pj0qj0h}).  From the expression for $u_{thr}$ in
(\ref{eq:uthr1}), if $\sigma_j=0$, then $\langle \gef_j(0), f_2\rangle $
corresponds to $2 \int_0^\infty f_{2,j}(r)dr$ from (\ref{e:u0asyh})
and (\ref{eq:sepvarh}).  If $\sigma_j>0$,
then $\gef_j (\sigma_j) \langle f_1, \gef_j(\sigma_j)\rangle $
corresponds to $ 2\sqrt{\frac{2\pi}{\sigma_j}}p_{j,0,N}(r) \phi_j$
from (\ref{e:ujasyh}).  In contrast, for the Dirichlet half-cylinder
there are no embedded resonances.  Hence, for the Dirichlet half-cylinder
$\gef_j(\sigma_j)=0$ for each $j.$

Theorem \ref{thm:wavedecay}  is 
 a more general version of Theorem \ref{thm:t1intro}.   In Theorem \ref{thm:wavedecay} we 
can allow any manifold with infinite cylindrical ends for which we have a polynomial bound on the cut-off resolvent of the Laplacian at
high energies. 
That this condition holds for the manifolds in Sections \ref{ex:non1} and 
\ref{ex:3fun1} is shown in \cite{ch-da}, where such estimates are shown for the resolvent of $-\Delta+V$, for a large class of potentials $V$.  Theorems 1 and 2 
of \cite{ch-da3} give the needed resolvent estimates for the Dirichlet
Laplacian for the domains of Section \ref{ss:waveguides}.

  In fact, our wave
expansion, Theorem \ref{thm:wavedecay},
 holds for more general compactly supported perturbations of the Laplacian on a manifold or domain
with infinite cylindrical ends, as long 
as appropriate high-energy resolvent bounds hold.
 In Section \ref{ss:blackbox} we adapt the Euclidean
black box formalism of \cite{sj-zw}
to the cylindrical end setting. This allows  us to treat in a unified way problems with and 
without a boundary extending to infinity. This formalism
also allows us to handle the Klein-Gordon 
equation in addition to the wave equation.
In Section  \ref{ss:blackbox} we give a self-contained
presentation of 
some properties of the resolvent, generalized eigenfunctions, and 
spectral measure in this very general setting.   This
extends  some results known
for manifolds with infinite cylindrical ends. 



\subsubsection{Examples with minimal trapping}\label{ex:non1}
Let $r$ be the radial coordinate in $\mathbb R^d$ for some $d\ge 2$, and let
\[
  X = \mathbb R^d, \qquad g_0 = dr^2 + F(r) dS,
\]
 where $dS$ is the usual metric on the $(d-1)$-dimensional unit sphere, $F(r) = r^2$ near $r=0$, and $F'$ is compactly supported on some interval $[0,R]$ and positive on $(0,R)$; see Figure \ref{f:cigar}. 
\begin{figure}[h]
\hspace{3cm}
\includegraphics{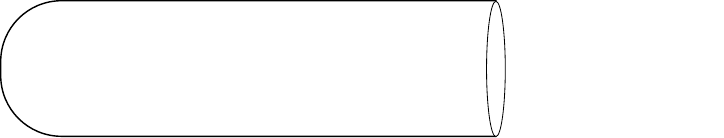}
 \caption{A cigar-shaped warped product.}\label{f:cigar}
\end{figure}

Then all $g_0$-geodesics obey, for $r(t)\not =0$,
\[\ddot r(t) := \frac{d^2}{dt^2} r(t) = 2 |\eta|^2 F'(r(t))F(r(t))^{-2}\ge0,\]
 where $r(t)$ is the $r$ coordinate of the geodesic at time $t$ and $\eta$ is the angular momentum, constant on each geodesic. Consequently, the only trapped geodesics (that is, the only maximally extended geodesics with $\sup_{t \in \mathbb R} r(t) < +\infty$) are the ones with $\dot r(t) \equiv F'(r(t)) \equiv 0$, that is the ones in the cylindrical end that have no radial momentum. This is the smallest amount of trapping a manifold with a cylindrical end can have.

Let $g$ be any metric on $ X$ such that $g-g_0$ is supported in  
$\{(r,y)\mid r <R \}$, and such that $g$ and $g_0$ have the same trapped geodesics. 
For example we may take $g = g_0 + c g_1$, where $g_1$ is any symmetric two-tensor with support in $\{(r,y)\mid r < R\}$, and $c \in \mathbb R$ is chosen sufficiently small depending on $g_1$. 
Alternatively, we may take $g = dr^2 + g_S(r)$, where $g_S(r)$ is a smooth family of metrics on the sphere such that $g_S(r) = r^2dS$ near $r=0$ and $g_S(r) = F(r)dS$ near $r \ge R$, and such that $\partial_r g_S(r) > 0$ on $(0,R)$. This way we can construct examples where $g-g_0$ is not small. 

For this set of examples, it is a consequence of the results of \cite{ch-da} that
there are at most finitely many eigenvalues of the Laplacian on $X$, and at most
finitely many threshold resonances.

\subsubsection{Examples based on convex cocompact manifolds}\label{ex:3fun1}

Let $(X,g_H)$ be a convex cocompact hyperbolic surface, such as the symmetric hyperbolic `pair of pants' surface with three funnels depicted in Figure \ref{f:3fun}.

\begin{figure}[h]
\labellist
\small
\pinlabel $r$ [l] at 280 -3
\pinlabel $\cosh^2\!r$ [l] at 225 70
\pinlabel $F(r)$ at 275 36
\endlabellist
 \includegraphics[width=5cm]{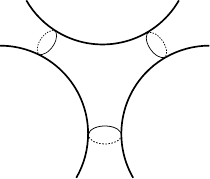} \hspace{2cm} \includegraphics[width=7cm]{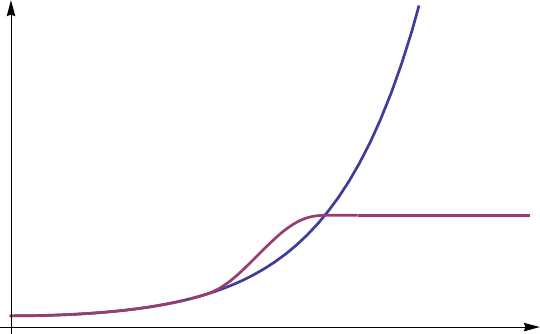}
 \caption{A hyperbolic surface $(X, g_H)$ with three funnels, and a modification of the metric which changes the funnel ends to cylindrical ends.}\label{f:3fun}
\end{figure}

In particular, there is a compact set $N \subset X$ (the convex core of $X$) such that 
\[
 X\setminus N = (0,\infty)_r\times Y_y, \qquad g_H|_{X \setminus N} = dr^2 + \cosh^2 \!r\, dy^2,
\]
where $Y$ is a disjoint union of $k\ge1$ geodesic circles, not 
necessarily all of the same length.

We 
construct a metric on $X$ which
gives it the structure of a manifold with infinite cylindrical ends
by modifying the metric on the funnel ends.  Take $g$ such that 
\begin{equation}\label{eq:ppants}
g|_N = g_H|_N, \qquad g|_{X \setminus N} = dr^2 + F(r)dy^2,
\end{equation}
 where $F(r) = \cosh^2\!r$ near $r=0$, and  $F'$ is compactly supported and positive on the interior of the 
convex hull of  its support.

It is also possible to construct examples with dimension $d\geq 3$ -- see \cite[Section 2.2]{ch-da}.


Again, for this set of examples, it is a consequence of the results of \cite{ch-da} that
there are at most finitely many eigenvalues of the Laplacian on $X$, and at most
finitely many threshold resonances.

\subsubsection{Star-shaped planar waveguides, with the Dirichlet Laplacian}
\label{ss:waveguides}
We next turn to examples of domains $X\subset \Real^2$ 
for which Theorem \ref{thm:t1intro} holds
for solutions $u$ of the wave equation
 which in addition satisfy 
Dirichlet boundary conditions on $\partial X$.  Though here we restrict
ourselves to $d=2$ for simplicity, the
paper \cite{ch-da3} has conditions on domains in $\Real^d$ with $d\geq 3$
which ensure the theorem holds, as well as alternative sufficient 
conditions for planar 
domains.

\begin{figure}
[h]
\labellist
\pinlabel $X$ [l] at 280 190
\pinlabel $X$ [l] at 900 190
\pinlabel $X$ [l] at 1480 190
\endlabellist
\includegraphics[width=3cm]{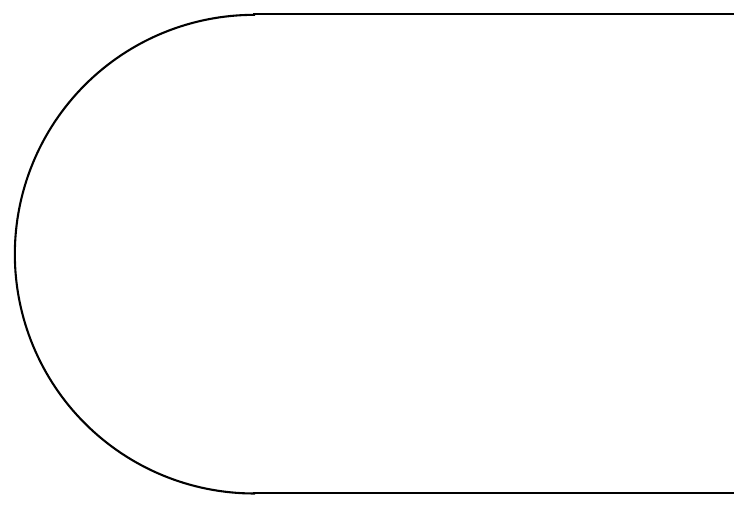}
\hspace{1cm}
\includegraphics[width=4cm]{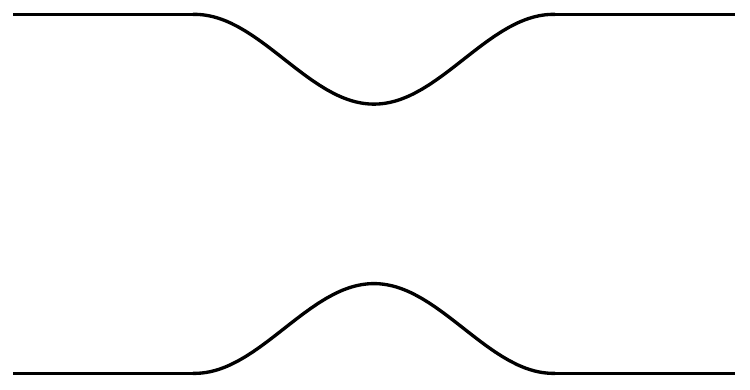}
\hspace{1cm}
\includegraphics[width=3.4cm]{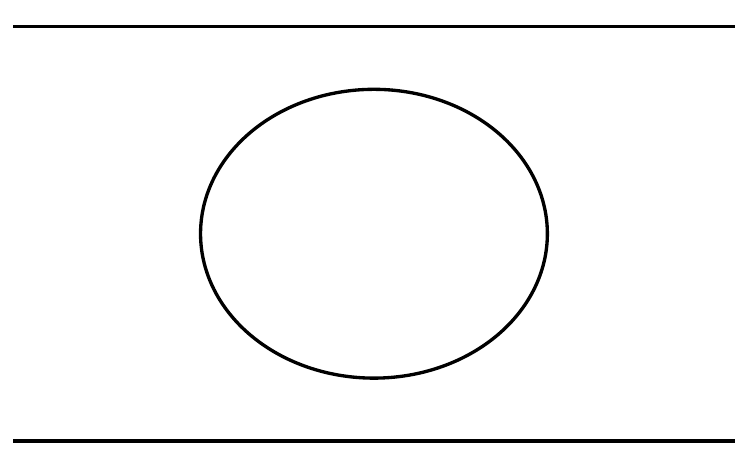}
 \caption{A domain with one end, and two domains with two isometric ends.}\label{f:cig}
\label{f:isometricends}
\end{figure}




Let
 $(s,y)$ be Cartesian coordinates on $\Real^2$ and $X\subset \Real^2$
be a domain. Three examples of domains we allow are shown in Figure \ref{f:isometricends}. They are defined as follows:
\begin{enumerate}
\item Let $X$ be the union of the open disk of radius $1$ centered at $(1,0)$ with the strip $(1,\infty)\times(-1,1)$.
\item Let $X$ be the set of $(s,y)$ such that $ f_1(s) - 1 \le y \le 1 - f_2(s)$, where $f_1$ and $f_2$ are nontrivial compactly supported functions in $C^{1,1}(\mathbb R)$, which take values in $[0,1)$, and which are monotonic away from $s=0$.
\item Let $X = (\mathbb R \times (-1,1)) \setminus K$, where $K$ is a compact, convex subset of $\mathbb R \times (-1,1)$ which has $C^{1,1}$ boundary and is symmetric about the $y$ axis.
\end{enumerate}  

To state more general assumptions, let
$\nu=(\nu_s,\nu_y)$ be the outward pointing unit normal to $\partial X$.
We assume 
 $s\nu_s\leq 0$ on $\partial X$, and call 
such domains ``star-shaped'' with respect to $s=\pm \infty$.
  Given  an open interval
$I$ and positive 
constant $C_I$, define
$\Gamma_F=\Gamma_F(I,C_I)\subset \partial X\cap (I\times \Real)  $ via
\begin{equation}\label{e:flare}
\Gamma_F=\{(s,y)\in \partial X\cap(I\times \Real)\mid\; 
s \nu_s \le -C_I \} .
\end{equation}
Now we assume that there is an open interval $I$ and a positive 
constant $C_I$ so that 
 the intersection of $X$ with $I \times \mathbb R$ consists of bounded open sets 
$X_1, \dots, X_K$ with mutually disjoint closures such that for each 
$k = 1, \dots ,K$, 
\[(\partial X \cap \partial X_k) \setminus \Gamma_F \subset I \times \{a_k\},\] for some real $a_k$.


Moreover, we assume $X$ is a domain with infinite cylindrical ends.
Specializing to the case here,
by this we mean that there is $R_0>0$ such that $X\cap ([-R_0,R_0] \times \mathbb R)$ is bounded and 
\begin{equation}\label{eq:cylend}
X \cap ((-\infty,-R_0] \times \mathbb R) = (-\infty,-R_0] \times Y_-, \qquad X \cap ([R_0,\infty) \times \mathbb R) = [R_0, \infty) \times Y_+,
\end{equation}
where $Y_-$ and $Y_+$ are (not necessarily connected) bounded open sets in $\mathbb R$.  We allow the possibility that one, but not both, of 
$Y_{\pm}$ is the empty set.  Then $Y=Y_+\sqcup Y_-$.

Finally we assume that 
each point $p \in \partial X$ has a neighborhood $U_p$ so that either 
$U_p\cap X$ is convex or $U_p\cap \partial X$ is $C^{1,1}$.
 Figures 
\ref{f:isometricends} and \ref{f:nonisometricends} include examples of 
domains satisfying all of these conditions.


\begin{figure}[h]
\labellist
\pinlabel $X$ [l] at 300 150
\pinlabel $X$ [l] at 760 156
\endlabellist
\includegraphics[width=4cm]{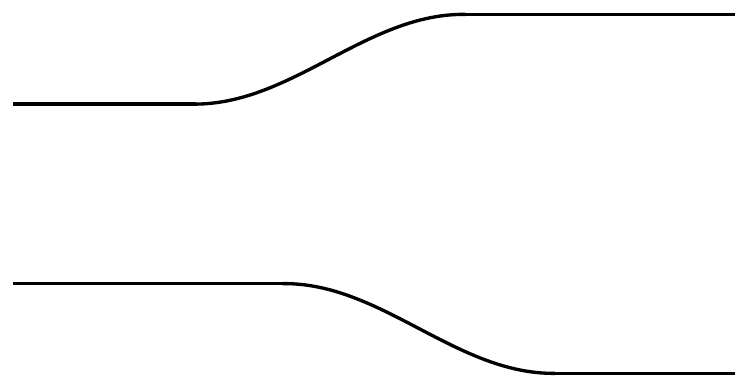} 
\hspace{1.5cm}
\includegraphics[width=3.4cm]{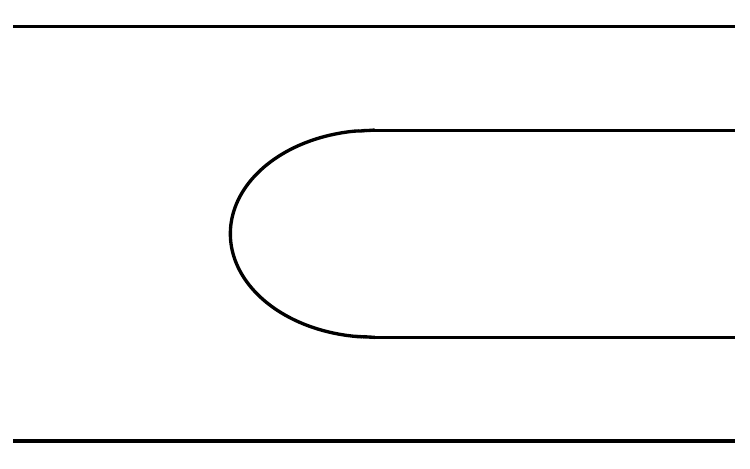} 
 \caption{Some domains with multiple nonisometric ends.}
\label{f:nonisometricends}
\end{figure}

Resolvent estimates for domains of this type are studied in \cite{ch-da3}. 
Results of that paper ensure that for the class of domains in 
$\Real^2$ described here, there are no embedded eigenvalues or
resonances.  Hence the sums in (\ref{eq:ue1}) and (\ref{eq:uthr1}) are 
actually $0$, and $\| \chi u(t)\|=O(t^{-1})$.

\subsection{Complete expansions under an additional spacing condition on the thresholds}
\label{s:compexp} In this subsection we suppose that $(X,g)$ is as in Section \ref{s:two} but with an additional assumption on the eigenvalues of $-\Delta_Y$.   This assumption holds for
all the examples in Section \ref{ex:non1}, but  only for
some of the other examples.  

To state it, let $\{\nu_l\}_{l=1}^\infty$ be the sequence of square roots of \textit{distinct positive} eigenvalues of $-\Delta_Y$ in increasing order, so 
that $0  < \nu_1< \nu_2<\cdots$. The assumption is that there are positive constants $c_Y$ and $N_Y$, such that
\begin{equation}\label{eq:distinct}
 \nu_{l+1} - \nu_l \ge c_Y \nu_l^{-N_Y},
\end{equation}
 for all $l\in \Natural $ with $\nu_l\geq 1$. Note that this assumption allows the eigenvalues of $-\Delta_Y$ to have high multiplicities, but forbids \textit{distinct} eigenvalues from clustering too closely together.  
Since for the examples in Section \ref{ex:non1} $\nu_l= \sqrt{l(l+d-2)}$,
it is easy to see \eqref{eq:distinct} holds.  On the other hand, consider
the examples illustrated using Figure \ref{f:3fun} in Section \ref{ex:3fun1}--these have
$d=2$ and 
three connected ends.  In this case,
$Y$ consists of the disjoint union of three circles. 
Noting that the function $F(r)$ of (\ref{eq:ppants}) is the same for 
each end, if the radii of these three circles
are rational multiples of each other, then (\ref{eq:distinct}) holds.
In fact, if the radii are algebraic 
multiples of each other, this holds by  Liouville's Theorem  on approximation
of algebraic numbers, e.g. \cite[Proposition 7.5.15]{DKS}.   If, however, for some pair of the circles 
the radius of one circle is a transcendental multiple of the other, 
then this condition may not hold.  In general, this condition holds for
some of the examples of Section \ref{ex:3fun1}, but not for all.  Turning
to the domains in $\Real^2$
 pictured in Section \ref{ss:waveguides}, we can see 
that (\ref{eq:distinct}) holds for all the examples in Figure \ref{f:isometricends}, where each domain either has a single
connected end, or two isometric connected ends.
For the examples in Figure \ref{f:nonisometricends}, whether or not 
(\ref{eq:distinct}) holds is again subtle, depending on the ratios of 
the lengths of the cross-sections of the ends.

With the assumption (\ref{eq:distinct}) and a bound on the cut-off resolvent
at high energy we can bound derivatives of the cut-off resolvent
at high energy, see Lemmas \ref{l:usecauchyestimates1} and 
\ref{l:usecauchyestimates2}.  This allows us to refine our expansions.

\begin{thm}\label{thm:intro2}
Let  $(X,g)$ be as in the examples in 
Sections \ref{ex:non1} or \ref{ex:3fun1}.
Alternatively, let $X\subset \Real^d$ be a domain as in 
section \ref{ss:waveguides}, and consider
the Dirichlet boundary conditions.
In either case, assume (\ref{eq:distinct}) holds.  Let
 $ \ f_1, \ f_2 \in C_c^\infty(X)$ be given, and let
$u(t)$ solve \eqref{e:cylcauchy} and satisfy Dirichlet boundary
conditions in the case of a domain $X\subset \Real^2$.
Then for each $k_0 \in \mathbb N$ we can write
\[
 u(t) = u_e(t) + u_{thr,k_0}(t) + u_{r,k_0}(t),
\]
where $u_e$ is still given by \eqref{eq:ue1} and
\begin{equation*}
 u_{thr,k_0} (t)= \frac{1}{4} \sum_{\sigma_j=0} \Phi_j(0)\langle f_2,\Phi_j(0)\rangle+
 \sum_{k=0}^{k_0-1}t^{-1/2-k} 
\sum_{l=1}^\infty (e^{it \nu_l}b_{l,k,+}+ e^{-it \nu_l}b_{l,k,-})  
\end{equation*}
for some  $b_{l,k,\pm}\in \langle r \rangle ^{1/2+2k+\epsilon}L^2(X)$.
For any $\chi \in C_c^\infty(X)$ there is a constant $C$ so that
$$\sum_{l=1}^\infty \|\chi  b_{l,k,\pm}\|_{L^2(X)}<C,\; k=0,1,2,...,k_0-1$$
and
\[
 \|\chi u_{r,k_0}(t)\|_{L^2(X)} \le C t^{-k_0}\; \text{for $t$ sufficiently large}.
\]  Moreover, the $b_{l,k,\pm}$
 are determined by the value $\nu_l$, the initial data
$f_1$, $f_2$, and suitable 
derivatives  of elements of the set 
$\{ \gef_{j'}(\lambda)\}_{0\leq \sigma_{j'}\leq \nu_l}$ evaluated at $\pm \nu_l$.
\end{thm}
For further details about how the $b_{l,k,\pm}$ are determined, see
Theorem \ref{thm:waveexp2} and  Lemma \ref{l:lcont} and its proof.
We note, however, that for general initial data and $k>0$
we do not expect the 
sums $\sum_{l=1}^\infty b_{l,k,\pm}$ to be finite sums; compare, for 
example the Dirichlet half-cylinder
case  (\ref{e:ujasyh}) and (\ref{e:pj1qj1h}).
Comparing Theorem \ref{thm:t1intro} gives, for $l\in \Natural$
$$b_{l,0,\pm}=\frac{\sqrt{\nu_l}}{4\sqrt{2\pi}}\sum_{\sigma_j=\nu_l}\gef_j(\sigma_j)
\left( 
e^{\pm i\pi/4} \langle f_1,\gef_j(\sigma_j)\rangle +\frac{1}{i\nu_l}e^{\pm i \pi/4}\langle f_2,\gef_j(\sigma_j)\rangle\right).$$

As in Theorem \ref{thm:t1intro}, because of the smoothness of the initial data we can instead bound $\|\chi u_{r,k_0}(t)\|_{H^m(X)} \le C t^{-k_0}$
with the constant depending on $m$ as well as $\chi$ and the initial data, and the series $\sum_{l=1}^\infty \|\chi  b_{l,k,\pm}\|_{H^m(X)}$
converges for each value of $m$.

Theorems \ref{thm:t1intro} and 
\ref{thm:intro2}, and their more general versions, 
Theorems \ref{thm:wavedecay} and  
\ref{thm:waveexp2},  require high energy bounds on the norm of the cut-off resolvent.  The bounds are
generally proved under some conditions on the trapping, as we have discussed 
earlier in the introduction.  However, for a general manifold with
infinite cylindrical ends, without a bound on
the high energy behavior of the resolvent 
or any restrictions on the trapping, we can find 
an asymptotic expansion of $\chi \psi_{sp}(-\Delta )u(t)$, provided that $\psi_{sp}\in C_c^\infty(\Real)$, see Proposition \ref{p:speccutoffexp}.  Here, as before, 
$u(t)$ is the solution of (\ref{e:cylcauchy}), and we must assume the
initial data have support in a fixed compact set.

The literature of the study of local energy decay under the assumptions of no trapping
or mild trapping is quite large, and we mention only a few papers.
The study of local energy decay for nontrapping perturbations of the Laplacian
on Euclidean space was initiated by Morawetz in \cite{Mo61} and continued in, for example,
\cite{ LaMoPh62,MoRaSt, vainberg}.  The question of wave expansions or 
wave decay on noncompact manifolds with various kinds of ends and different
trapping assumptions is a very active area of research;
see \cite{zw16} and references therein for some more recent results.  The 
most closely related results of which we are aware are 
for  solutions to the wave equation on a planar waveguide without forcing
\cite{Lyf76} and  
with forcing \cite{HeWe}, where
the expansion is found to order $o(1)$.  For energy decay of
solutions to a dissipative 
wave equation on a waveguide, see \cite{MaRo}.

Our results build on studies of the spectral theory of manifolds with
cylindrical ends, in particular we mention \cite{goldstein, Lyf76, gui89,
tapsit, ace, Pa95}.  More recent papers include \cite{Chr02,IsKula,MuSt, RiTi13} and 
references therein.
We give more precise references as they are used.


\subsection{Notation} In this section we collect, for reference, some notation introduced in Section \ref{ss:bb}: 

\begin{itemize}
\item $H$ is an operator (generalizing the Laplacian on a manifold with cylindrical ends) on a Hilbert space $\mch$.
\item $H_Y$ is a
nonnegative operator (generalizing the Laplacian on the cross sections of the cylindrical ends)  on a Hilbert space $\mch_Y$.
\item
$\{ \sigma_j^2\}_{j \geq 0}$ are the eigenvalues of $H_Y$, repeated with multiplicity, with $0\leq \sigma_0\leq \sigma_1\leq \sigma_2\leq \cdots$.  
This set may be finite or infinite.
\item
$\{ \phi_j\}$ are a complete orthonormal set of eigenfunctions of $H_Y$
 with $H_Y\phi_j=\sigma_j^2 \phi_j$.
\item
$\{ \nu_l^2\}_{l\geq 1}$ are the {\em distinct}, {\em positive} eigenvalues of $H_Y$  with $0<\nu_1<\nu_2< \cdots$.  
\end{itemize}

And here is some notation introduced elsewhere in the paper:

\begin{itemize}
\item $f_1,f_2$ are initial data; see (\ref{e:cylcauchy}) and (\ref{eq:usolution}).
\item $\tau_j(\lambda) = (\lambda^2 - \sigma_j^2)^{1/2}$, with $\Im \tau_j>0$ when $\Im \lambda>0$; see \eqref{e:taudef}.
\item The resolvent $R(\lambda) = (H - \lambda^2)^{-1}\colon \mch \to \mch$ is originally defined when $\Im \lambda >0$, but the cutoff resolvent $\chi R(\lambda)\chi$ continues meromorphically to the Riemann surface $\zhat$, which is a cover of $\mathbb C$ with covering map $\pr$; see Section~\ref{ss:r}.
\item $\Pre$ is the orthogonal projection onto the span of the eigenfunctions of $H$, if any; see Section~\ref{ss:r}.
\item  $\gef_j(\lambda)$ are generalized eigenfunctions (also sometimes called scattering solutions, or distorted plane waves) of $H$: they obey $H\Phi_j(\lambda)= \lambda^2 \Phi_j(\lambda)$ and also satisfy further conditions; see (\ref{eq:gef}).
\item
$\{ \eta_\ell\}$ are orthonormal eigenfunctions of $H$ with eigenvalue 
$\ev_\ell$: $H \eta_\ell = \ev_\ell \eta_\ell$; see Section~\ref{ss:eigenfunctions}.
\end{itemize}

Throughout the paper, $C$ denotes a positive constant whose value may change
from line to line.

\section{Black box scattering on cylinders}
\label{ss:blackbox}

The results we 
shall prove about local wave expansions
 are valid for a large class of ``black box" perturbations of 
the Laplacian on a manifold with infinite cylindrical ends, with or without boundary, provided that there are appropriate 
high-energy bounds on the cut-off resolvent. 
In fact, the natural setting for our work is more general still, requiring only that `infinity is one-dimensional' in a sense we make precise below. This more abstract setting also allows low regularity (as in some of the examples from \cite{ch-da3}) and other settings, like quantum graphs.

 We adapt the idea of \cite{sj-zw} (see also \cite[Chapter 4]{dy-zw})
of a  compactly supported black box perturbation of the Laplacian on Euclidean space to give a definition of a compactly supported black box 
 perturbation of the Laplacian on a manifold with
 infinite cylindrical ends, or more generally of an operator on $L^2(\mathbb R_+;\mathcal H_Y)$, where $\mathcal H_Y$ is any Hilbert space. 
We collect a number of results related to the resolvent, spectrum, generalized
eigenfunctions and spectral measure of such operators.  
We give a mostly self-contained presentation of these results, though
many can be found in \cite{gui89,tapsit,Pa95,ace} in less general settings.

\subsection{Black box operators} 
\label{ss:bb}
In this section we give the main definitions and assumptions we will use throughout the paper, and illustrate them with examples, focusing first on the simpler case of complete manifolds without boundary, and then on the more complicated one where there is a boundary. We conclude with some further examples.

\subsubsection{Manifolds without boundary}\label{sss:mannob}

Our basic example of a black box operator is the Laplacian on a manifold with infinite cylindrical ends and no boundary, as in Sections \ref{ex:non1} and \ref{ex:3fun1}. Namely, let $(X,g)$ be a complete Riemannian manifold without boundary, having an open subset $X_\infty$ (the infinite cylindrical ends) with the following properties: $X \setminus X_\infty$ is bounded, and
\[(X_\infty,g\restrict_{X_\infty}) = ((0,\infty)_r \times Y, dr^2 + g_Y),\]
where $(Y, g_Y)$ is a compact manifold without boundary. Then let 
\begin{equation}\label{e:basex}
\begin{split}
&\mch = L^2(X), \qquad \mch_0 = L^2(X\setminus X_\infty), \qquad \mch_Y = L^2(Y),\\
& H = -\Delta \text{ be the Laplacian on }X, \qquad \mcd = H^2(X),\\
& H_Y = -\Delta_Y  \text{ be the Laplacian on }Y, \qquad \mcd_Y = H^2(Y).
\end{split}
\end{equation}
In what follows we will use the notation on the left hand sides of \eqref{e:basex} in a more general way, but the reader should keep in mind this more concrete setting.

\subsubsection{Definitions, notation, and assumptions}\label{s:bbdefnot}
 Let $\mch$ be a complex Hilbert space with orthogonal decomposition 
\[
 \mch= \mch_0\oplus L^2(\mathbb R_+; \mch_Y),
\]
 where $\mch_0$ and $\mch_Y$ are separable Hilbert spaces.  
 We use $r$ as the coordinate for $\mathbb R_+ = (0,\infty)$. Let $H$ be a self-adjoint operator on $\mch$ with dense domain $\mathcal D$. Let $H_Y$ be a  self-adjoint operator on $\mathcal H_Y$ with dense domain $\mathcal D_Y$.
 
 Let ${\bbbone}_0$ and  ${\bbbone}_\infty$
 denote the orthogonal projections ${\bbbone}_0:\mch\rightarrow \mch_0$ and
 $\bbbone_\infty:\mch\rightarrow  L^2(\mathbb R_+; \mch_Y)$.
 
If $g \in C([0,\infty))$  and 
$f\in \mch$, then 
by $g f$ we mean 
\begin{equation}\label{eq:chiconvention} \chi f=g(0) \bbbone_0 f+ g \bbbone_\infty  f.
\end{equation}
Hence $f\in g \mch$ means $f = g h$ for some $h \in \mch$.
We will most often use this for $g = \chi \in C_c^\infty([0,\infty))$ and $g(r)=\langle r \rangle^{p}=(1+r^2)^{p/2}$.


Assume that $H$ is lower semi-bounded and that $\bbbone_{0}(H+i)^{-1}$ is compact. Assume that $H_Y \ge 0$ and that $H_Y$ has no essential spectrum.

Let $\{\phi_0, \, \phi_1, \dots\}$ be an orthonormal basis of eigenfunctions of $H_Y$, such that $H_Y \phi_j = \sigma_j^2 \phi_j$, with $0 \le \sigma_0 \le \sigma_1 \le \cdots$. Beginning in \S\ref{s:tta}, we will also use $0<\nu_1<\nu_2<\cdots$ to denote the positive, distinct elements in the sequence $0 \le \sigma_0 \le \sigma_1 \le \cdots$. Thus, the $\sigma$'s are the square roots of \textit{all} the eigenvalues of $H_Y$ \textit{with} multiplicity, and the $\nu$'s are the square roots of the \textit{positive} eigenvalues of $H_Y$ \textit{without} multiplicity.

If $ \dim \mch_Y < \infty$, then these sequences terminate.  This occurs 
 if and only if the operator $H_Y$ is bounded.

We identify elements $u \in L^2(\mathbb R; \mch_Y)$ with sequences $(u_j) \in \ell^2(\mathbb N_0; L^2(\mathbb R_+))$ using the above basis:
\[
u(r) = \sum_{j=0}^\infty u_j(r) \phi_j,
\]
where, if  $ \dim \mch_Y < \infty$, then once again the sequences must terminate.
Let
\begin{equation}\label{e:dinfdef}
\mcd_\infty = \left\{(u_j) \in \ell^2(\mathbb N_0; L^2(\mathbb R_+))  \mid (-u_j'' + \sigma_j^2u_j) \in \ell^2(\mathbb N_0; L^2(\mathbb R_+))  \right\}.
\end{equation}
Assume that $\bbbone_\infty \mcd \subset \mcd_\infty$,
and  ${\bbbone }_\infty H = -\partial_r^2 + H_Y$ 
in the sense that if $f\in\mcd$, then 
$\bbbone_\infty H f=(- \partial_r^2 + H_Y)\bbbone_\infty f$.  
Assume that if 
$f\in  \mcd_\infty$ and $f\restrict_{r\leq \epsilon}=0$ for some $\epsilon>0$,
then $f\in \mcd$.



\subsubsection{Examples with boundary.}

We can modify the basic examples of Section \ref{sss:mannob} to allow $X$ and $Y$ to be open manifolds, if we impose appropriate boundary conditions.

The simplest such example is the half cylinder of Section \ref{s:hprod}. In addition to the Dirichlet and Neumann boundary conditions considered there, we could just as well take more general self-adjoint boundary conditions.

We now turn to our primary example  with $\partial
Y \not =\emptyset$, where $X$ is isometric to a waveguide. Here by a waveguide we mean  a 
domain $\Omega\subset \Real^d$  which 
has a subset $\Omega_\infty\subset \Omega$
so that the closure of
$\Omega \setminus \Omega_\infty$ is compact and $\Omega_\infty$ 
is a ``cylindrical end.''  That is, $\Omega_\infty=\bigcup_{l=1}^L\Omega_l$,
and, for $l=1,\dots,L$, 
$\Omega_l$ can, under a rigid motion of $\Real^d$, be identified
with $(0,\infty)\times U_l$, with $U_l\subset \Real^{d-1}$ a 
bounded domain.  In our earlier notation, $Y=\sqcup_{l=1}^LU_l$.
The metrics on $\Omega$ and $Y$ come from restriction of the usual 
Euclidean metric. Then put $\mch = L^2(\Omega)$, 
$\mch_0 = L^2(\Omega\setminus \Omega_\infty)$, and $\mch_Y = L^2(Y)$.

To consider the Dirichlet Laplacian on $\Omega$, we set $H=-\Delta$, where $\Delta$ is the  Euclidean Laplacian here, and   $\mcd=\{f\in H^1_0(\Omega): \Delta f \in L^2(\Omega)\}$; see e.g. \cite[Section~8.2]{Tay2}. Similarly $H_Y = -\Delta_Y$ and 
$\mcd_Y=\{g\in H^1_0(Y): \Delta_Y g \in L^2(Y)\}$. Then 
\[
\mcd_\infty = \left\{f \in H^1\left( \Omega_\infty\right) \mid \Delta f \in L^2\left(\Omega_\infty \right) \text{ and } f \restrict_{ (0,\infty)\times \partial U_l} = 0\right\},\]
and the desired inclusions involving $\mcd$ and $\mcd_\infty$
follow.  That 
$\bbbone_{\mch_0}(H+i)^{-1}$ is compact follows from e.g. 
\cite[Theorem~6.9]{Borth}, and the desired properties of $H_Y$ can be obtained by e.g. taking  an open cube $Q$ with 
$Y\subset Q$ and using domain monotonicity of eigenvalues as in e.g.  \cite[Theorem~6.20]{Borth}.

Notice that for the treatment of Dirichlet boundary conditions on $\Omega$ we did not require any regularity on the boundary $\partial \Omega$. With appropriate regularity assumptions,
the Laplacian on a domain or manifold 
with Neumann (or certain more general) boundary
conditions can also be put into the above black box framework; see e.g. 
\cite[Section~8.2]{Tay2} and \cite[Section~6.4]{Borth}.

\subsubsection{Other examples} 
We have already mentioned our primary examples of 
black box operators motivated by the Laplacian on a
manifold with infinite cylindrical ends, but we 
 briefly outline some further examples. 

We begin with an example in which the space $\mch_Y$ is finite-dimensional.
Consider the operator $-\partial_r^2+V$ on 
$L^2(\Real)$, where $V$ is a ``steplike''  potential, 
as studied in, for example \cite{CoKa,Chr06,DA-Se}, and others.  For us,
this means $V\in L^\infty(\Real;\Real)$, and there is an $R>0$ and 
constants $V_\pm$ so that $V(r)=V_{\pm}$ 
when $\pm r>R$.  Here $\mch_Y=\Complex^2=\mcd_Y$ with the 
usual inner product, and if $f=(f_+,f_-)\in \mch_Y$, then $H_Y f=(V_+ f_+,V_-f_-)$.
Thus we must make the assumption that $V_\pm \geq 0$ in 
order that the condition that $H_Y\geq 0$ be satisfied.  The
hypotheses we make on $V$ are 
 more restrictive than those of \cite{CoKa}.  In a similar fashion,
matrix-valued Schr\"odinger operators on $\Real$ with steplike potentials
can be put in this framework, provided the matrix-valued potential satisfies
 analogous
conditions.

Another example with $\mch_Y$ finite-dimensional is provided by a 
Schr\"odinger operator on a quantum graph $X$ which is the union of a finite
part and a finite number of half-lines. 
 Let $\mathcal E$ be the set of edges,
and let $\mathcal{E}=\mathcal{E}_K\cup\mathcal{E}_\infty$,
where each edge $e\in \mathcal{E}_K$  has finite
length $l(e)$, and each edge in $\mathcal{E}_\infty$ can
be identified with $(0,\infty)$.  We assume each
set $\mathcal E_K$ and $\mathcal E_\infty$ is finite.
Then
$\mch = \bigoplus_{e \in \mathcal E_K} L^2(0, \ell(e))\bigoplus_{e\in \mathcal E_\infty}L^2(0,\infty).$

  If we impose Neumann--Kirchoff boundary conditions at the vertices, 
then $\mcd$ is the set of functions $f\in \mch$ whose restrictions to all edges are in $H^2$, which are continuous across vertices, 
and which satisfy, at each vertex $v$,
 $\sum_{e\in \mathcal{E}(v)}\frac{\partial f}{\partial r_e}(v)=0$. Here 
$\mathcal{E}(v)$ is the set of edges meeting 
at vertex $v$, $r_e$ is the coordinate on edge $e$, and the derivatives
are taken in directions away from the vertex.  Now let $V\in L^\infty(X;\Real)$
have compactly supported (distributional) derivative, and suppose $V\geq 0$ 
outside of a compact set.  Then $H$ is given by $-\partial_{r_e}^2+V(r_e)$ on 
each edge. 
See e.g.  \cite[Chapter 8]{Borth} and \cite[Chapter 8]{EK} for more on quantum graphs.
 In this example, $\mch_Y=\Complex^N$, where $N$ is the number
of edges in $\mathcal E_\infty$ and the operator $H_Y$
is determined in a manner similar to the step-like Schr\"odinger case
described above, by the values of $V$ ``near infinity.''

A possibility with $\mch_Y$ infinite-dimensional
 is to let $Y$ be a finite quantum graph. 
Then using notation as in the previous example, $\mathcal E=\mathcal E^K$ and 
$\mch_Y = \bigoplus_{e \in \mathcal E} L^2(0, \ell(e)).$  We can let $H_Y$ 
be the operator $-\partial_y^2$ on each edge, with domain $\mcd_Y$
 obtained by imposing
Neumann-Kirchoff boundary conditions as described above.
Then let $X=\Real_+ \times Y$, $\mch =L^2(X)$, 
$V\in L^\infty_c(X;\overline{\Real_+})$, and $H=-\partial_r^2+H_Y+V$, with 
$\mcd=\mcd_\infty\cap \{ (u_j)\in \ell^2(\Natural_0;H_0^1(\Real_+)\}$,
giving the Dirichlet boundary condition at $r=0$.


Note that if the operator $H$ with domain $\mcd\subset \mch$ satisfies the
general 
conditions listed in Section 
\ref{s:bbdefnot}, so does the operator $H+m^2$ for any $m\in \Real$,
by replacing $H_Y$ by $H_Y+m^2$ (and hence replacing
$\sigma_j^2$ by $\sigma_j^2+m^2$).  This implies that Theorems
 \ref{thm:wavedecay} and \ref{thm:waveexp2}
apply to  the Klein-Gordon equation as well as the wave equation.

\subsection{The resolvent}\label{ss:r}
 Now let $H$ be an operator as described above, and for $\Im \lambda>0$, $\Re \lambda \ne 0$, define
the resolvent
\[
R(\lambda)=(H-\lambda^2)^{-1}: \mch\rightarrow \mch.
\]

We recall and extend some results (proved in  \cite{gui89},
 \cite[sections 6.7-6.10]{tapsit}, \cite[Section 2]{ace},
and \cite{Pa95} for manifolds with cylindrical ends) on the behavior of $R(\lambda)$ and of its meromorphic continuation, focusing on the region near the real axis and on the consequences for the  spectral measure. 

We use the following model resolvent. For $\Im \lambda>0$, let $R_0(\lambda)$ be the resolvent for $-\partial_r^2 + H_Y$ with Dirichlet boundary condition at $r=0$, that is to say
\[
(-\partial_r^2 + H_Y - \lambda^2) u = f \Longleftrightarrow u = R_0(\lambda) f,
\]
for $f \in L^2(\mathbb R_+; \mch_Y)$ and $u \in \mcd_\infty$ such that $u\restrict_{r=0} = 0$, where $\mcd_\infty$ is defined in \eqref{e:dinfdef}. Explicitly, for $\Im \lambda>0$, $\sigma_j^2\in \spec(H_Y)$, let
\begin{equation}\label{e:taudef}
\tau_j(\lambda)= (\lambda^2-\sigma_j^2)^{1/2},
\end{equation}
where we take the square root to have positive imaginary part. 
Then for any $(f_j) \in \ell^2(\mathbb N_0; L^2(\mathbb R_+))$ (recalling sequences
terminate if $\dim \mch<\infty$),
\begin{equation}\label{e:r0def}
R_0(\lambda) \sum_{j\geq 0} f_j(r) \phi_j =  \sum_{j\geq 0}\frac i {2 \tau_j(\lambda) }  \int_0^\infty\left(e^{i\tau_j(\lambda)|r-r'|} -e^{i\tau_j(\lambda)(r+r')} \right) f_j(r')dr' \phi_j.
\end{equation}

From this formula we see that the spectrum of $-\partial_r^2+  H_Y$, with Dirichlet boundary condition at $r=0$, is absoultely continuous and given by $[\sigma_0^2,\infty)$. 
However, the multiplicity of the 
spectrum changes at the points $\sigma_j^2\in \spec(H_Y)$, because, for any $\lambda^2 \ge 0$, the number of summands in \eqref{e:r0def} which are not bounded $L^2(\mathbb R_+) \to L^2(\mathbb R_+)$ equals the number values of $j$ such that $\sigma_j^2 \le \lambda^2$.
We call the points $\{\sigma_j^2\}$ thresholds.  When using $\lambda^2$ as
a spectral parameter, we shall abuse terminology a bit and refer to the 
points $\{ \pm \sigma_j \}$ as thresholds as well. 

In the remainder of Section \ref{ss:r} we will derive, among other things, the following description of the spectrum of $H$. The essential spectrum of $H$ is given by the same continuous spectrum $[\sigma_0^2,\infty)$, with the same thresholds and multiplicities. The discrete spectrum of $H$ is countable (possibly empty) with infinity being the only possible accumulation point.

We begin with the upper half plane, where $H$ may have only discrete  spectrum.

\begin{lemma}\label{l:uhp}
The resolvent $R(\lambda)$ is meromorphic for $\Im \lambda >0$.
\end{lemma}

This is fairly clear from our assumptions on $H$. However, 
we give a detailed proof based on (\ref{e:vodevphys}), an identity due to Vodev
\cite{VodevMathNach2014}, which will also be useful in our analysis later. In \cite{VodevMathNach2014} the identity is stated only for Schr\"odinger operators  on $\Real^d$.  However, it in fact holds in far greater
generality for operators which are, in an appropriate sense, compactly supported perturbations of each other.  The identity  (\ref{e:vodevphys}) is a version
adapted to our setting.
 
\begin{proof}
Let $\chi_1 = \chi_1(r)\in C_c^{\infty}([0,\infty))$ be $1$ near $r=0$. Let $\lambda$ and $\lambda_0$ have positive imaginary parts and nonzero real parts. To relate $R(\lambda)$ and $R_0(\lambda)$ we start with
$$R(\lambda)(1-\chi_1)\colon L^2(\mathbb R_+; \mch_Y)\rightarrow \mch,$$
which we write as
$$R(\lambda)(1-\chi_1)= R(\lambda)(1-\chi_1)(-\partial_r^2 + H_Y -\lambda^2)R_0(\lambda).$$
But since $(-\partial_r^2 + H_Y)(1-\chi_1) = H(1-\chi_1)$, we may 
write
\begin{align}\label{eq:QQ01}
R(\lambda)(1-\chi_1)& = R(\lambda)\{(H-\lambda^2) (1-\chi_1)+ [\chi_1,\partial_r^2]\}R_0(\lambda)\nonumber \\ & 
= \{(1-\chi_1)-R(\lambda)[\partial_r^2,\chi_1] \}R_0(\lambda).
\end{align}
Likewise
\begin{equation}\label{eq:QQ02}
(1-\chi_1)R(\lambda_0)= R_0(\lambda_0)\{(1-\chi_1)+[\partial_r^2,\chi_1]R(\lambda_0)\}.
\end{equation}

On the other hand, by the resolvent identity we have
\begin{align} \label{eq:V1}
R(\lambda)-R(\lambda_0)& = (\lambda^2-\lambda_0^2) R(\lambda) R(\lambda_0)\nonumber \\
 & = (\lambda^2-\lambda_0^2) \left( R(\lambda ) \chi_1 (2-\chi_1)R(\lambda_0)+ R(\lambda)(1-\chi_1)^2R(\lambda_0)\right).
\end{align}

Inserting \eqref{eq:QQ01} and \eqref{eq:QQ02} into \eqref{eq:V1} gives
\begin{equation}\label{e:vodevphys}
\begin{split}
R(\lambda)-R(\lambda_0)= &(\lambda^2-\lambda_0^2)\Big(R(\lambda) \chi_1(2-\chi_1)R(\lambda_0)
\\
&+ \{ (1-\chi_1)-R(\lambda)[\partial_r^2,\chi_1]\} R_0(\lambda)R_0(\lambda_0)
\{ (1-\chi_1)+[\partial_r^2,\chi_1]R(\lambda_0)\}\Big).
\end{split}
\end{equation}
Bringing the $R(\lambda)$ terms to the left, the remaining terms to the right, and factoring, gives
\[
R(\lambda)(I + K(\lambda,\lambda_0)) = F(\lambda,\lambda_0),
\]
for any $\lambda$ and $\lambda_0$ with positive imaginary parts and nonzero real parts, where
\[
K(\lambda,\lambda_0) = (\lambda_0^2 - \lambda^2)\Big(\chi_1(2-\chi_1)R(\lambda_0) - [\partial_r^2,\chi_1]R_0(\lambda)  R_0(\lambda_0) \{ (1-\chi_1)+[\partial_r^2,\chi_1]R(\lambda_0)\}\Big),
\]
and
\[
F(\lambda,\lambda_0) = R(\lambda_0) + (\lambda^2-\lambda_0^2)(1-\chi_1) R_0(\lambda)  R_0(\lambda_0) \{ (1-\chi_1)+[\partial_r^2,\chi_1]R(\lambda_0)\}.
\]
For $\lambda_0$ fixed and $\lambda$ sufficiently close to $\lambda_0$, we can invert $I + K(\lambda,\lambda_0)$ using a Neumann series. Observe that $K(\lambda,\lambda_0)$  is compact, because of our assumption that  $\bbbone_{0}(H+i)^{-1}$ is compact; see e.g. \cite[Lemma~4.2]{dy-zw}. Moreover $\lambda \mapsto K(\lambda,\lambda_0)$ continues analytically to the upper half plane. Consequently, by the analytic Fredholm theorem (see e.g. \cite[Theorem C.8]{dy-zw}), 
\begin{equation}\label{e:rkf}
\lambda \mapsto R(\lambda) = F(\lambda,\lambda_0)(I+K(\lambda,\lambda_0))^{-1},
\end{equation}
continues meromorphically to the upper half plane, for any fixed $\lambda_0$ with positive imaginary part and nonzero real part.
\end{proof}

The statement of Lemma \ref{l:uhp} is equivalent to the statement that the essential spectrum of $H$ is contained in $[0,\infty)$. Its proof can be easily adapted to show that the essential spectrum of $H$ equals $[\sigma_0^2,\infty)$. More specifically, if $\sigma_0 >0$, then use the fact that $R_0(\lambda)$ continues analytically to $\mathbb C \setminus \{\lambda \in \mathbb R \mid |\lambda|\ge \sigma_0\}$ to show that the essential spectrum of $H$ is contained in $[\sigma_0^2,\infty)$. To show the reverse containment, reverse the roles of $R(\lambda)$ and $R_0(\lambda)$ in the proof. 

To study the essential spectrum in more detail, we use  the minimal complete Riemann surface to which the functions $\tau_j$ (defined by \eqref{e:taudef}) continue analytically for each $\sigma_j^2 \in \spec(H_Y)$. We denote this Riemann surface by $\zhat$.  We use the term \textit{physical space} to refer to a certain copy of the upper half plane $\{\Im \lambda>0\}$ in $\zhat$, defined as follows: if $\sigma_0 =0$, then the physical space is the unique copy of the upper half plane on which $\Im \tau_j(\lambda)>0$ for each $\sigma_j^2 \in \spec(H_Y)$, and if $\sigma_0>0$, then there are two such copies, and we let the physical space be one of them (it doesn't matter which one). Throughout this paper we will identify the upper half plane in $\mathbb C$ with the physical space in $\zhat$, and $\mathbb R$ in $\mathbb C$ with the boundary of the physical space in $\zhat$.

We define a
projection
$\pr: \zhat \rightarrow \Complex$ as follows:  When $\lambda$ lies in the 
physical space, we put
$\pr(\lambda)=\lambda$, and for general $\lambda \in \zhat$, 
extend $\pr(\lambda)$ by analytic continuation.  
Then $\pr$ is an countably-sheeted covering map which is regular over 
$\mathbb C \setminus  \{\pm \sigma_j \colon j\in\Natural_0,\;\sigma_j>0\}$, and which has points of ramification of order two over each 
distinct $\pm \sigma_j$ with $\sigma_j >0$. If $\sigma_0 = 0$, then $\pr = \tau_0$. Regardless of the value of $\sigma_0$, the physical space is the unique preimage of the upper half plane in $\mathbb C$ under $\pr$ on which $\Im \tau_j(\lambda)>0$ for each $\sigma_j^2 \in \spec(H_Y)$. The global structure of $\zhat$ is complicated, but we shall only need to work with 
$\lambda\in \zhat$ in a neighborhood of $\Real$. For more about $\zhat$, see e.g. \cite[Section~6.7]{tapsit}

From the explicit formula \eqref{e:r0def}, we see that, for any $\chi = \chi(r) \in C_c^{\infty}([0,\infty))$, the cut-off resolvent $\chi R_0(\lambda) \chi$ continues analytically to $\zhat$. The physical space is that preimage of the upper half plane in $\mathbb C$ under $\pr$ on which $R_0(\lambda)$ is bounded $\mch \to \mch$. We now prove meromorphic continuation of $\chi R(\lambda) \chi$ to the same Riemann surface.

\begin{lemma}\label{l:mercont}
For any $\chi = \chi(r) \in C_c^\infty([0,\infty))$ the cut-off resolvent $\chi R(\lambda) \chi$ continues meromorphically to $\zhat$. 
\end{lemma}

\begin{proof}
The proof is straightforward
 because the preliminary work has been done in the proof of Lemma~\ref{l:uhp}.

Without loss of generality we may assume that $\chi$ is $1$ near $r=0$. Fix $\chi_1 = \chi_1(r) \in C_c^\infty([0,1))$, also $1$ near $r=0$, so that $\chi\chi_1=\chi_1$. Then, multiplying \eqref{e:rkf} on the left and right by $\chi$, we obtain
\begin{equation}\label{e:rfkc}
\chi R(\lambda) \chi = \chi F (\lambda,\lambda_0) (I+K(\lambda,\lambda_0) )^{-1}\chi =  \chi F (\lambda,\lambda_0)\chi (I+K(\lambda,\lambda_0)\chi )^{-1},
\end{equation}
for $\lambda_0$ fixed with positive imaginary part and nonzero real part, and for $\lambda$ sufficiently close to $\lambda_0$, where for the second equality we used the Neumann series for $(I+K)^{-1}$ and $\chi K = K$. From the fact that $\chi R_0(\lambda) \chi$ continues analytically to $\zhat$, and the resolvent identity $(\lambda^2 - \lambda_0^2)R_0(\lambda)R_0(\lambda_0) = R_0(\lambda) - R_0(\lambda_0)$, we see that $\lambda \mapsto \chi F (\lambda,\lambda_0)\chi $ and $\lambda \mapsto K(\lambda,\lambda_0)\chi$ both continue analytically to $\zhat$. Hence, by the analytic Fredholm theorem, $\chi R(\lambda) \chi$ continues meromorphically to $\zhat$.
\end{proof}

We now refine Lemma \ref{l:mercont} when $\Im \lambda = 0$, that is to say on the boundary of the physical space. Our arguments here mostly follow 
\cite[Propositions 6.27 and 6.28]{tapsit}.  We denote by $\Pre_{\lambda_0}\colon \mch \to \mch$ the orthogonal projection onto the eigenspace of $H$ at $\lambda_0^2$, with the convention that if $\lambda_0^2$ is not an eigenvalue of $H$, then $\Pre_{\lambda_0}=0$. 

\begin{lemma}\label{l:laurentreal}
Let $\lambda_0 \in \mathbb R$ and $\epsilon >0$. 

\begin{itemize}
\item If $\lambda_0$ is not a threshold, then there exists $B(\lambda)\colon \langle r \rangle^{-1/2-\epsilon}\mch \to \langle r \rangle^{1/2+\epsilon}\mch$, continuous with respect to $\lambda$, such that 
\begin{equation}\label{e:laurentnnotthresh}
R(\lambda) = - \frac {\Pre_{\lambda_0}}{\lambda^2 - \lambda^2_0} + B(\lambda),
\end{equation}
when  $\Im \lambda \ge 0$ and $\lambda$ is sufficiently close to $\lambda_0$.

\item If $\lambda_0$ is a threshold, that is to say $\lambda_0 = \sigma_j$ for some $\sigma_j^2 \in \spec(H_Y)$, then there exist $A_0, \ B(\lambda)\colon \langle r \rangle^{-1/2-\epsilon}\mch \to \langle r \rangle^{1/2+\epsilon}\mch$, with $A_0$ independent of $\lambda$ and $B$ continuous with respect to $\lambda$,  such that 
\begin{equation}\label{e:laurentthresh}
R(\lambda) = - \frac {\Pre_{\lambda_0}}{\tau_j(\lambda)^2} + \frac{A_0} {\tau_j(\lambda)} + B(\lambda),
\end{equation}
when $\Im \lambda \ge 0$ and $\lambda$ is sufficiently close to $\lambda_0$.
\end{itemize}
\end{lemma}

Before giving the proof, we clarify what we mean by $Hf$ if $f$ lies in a weighted space, rather than $\mch$.  Let $g(r)=\langle r \rangle ^p$ 
or $g(r)=e^{pr}$ for some $p\in \Real$.   
Suppose $f\in g \mcd$, with $\bbbone_{\infty}f=\sum f_j \phi_j$.  Then
$(f_j),\; (-f_j''+\sigma_j^2f_j )\in \ell^2(\Natural_0;g L^2(\Real_+))$.
If $\chi=\chi(r) \in C_c^\infty(\Real_+)$ is $1$ in a neighborhood of $r=0$,
then $H(1-\chi)f\in g \mch$ is given by
$H(1-\chi)f=\sum  (-\partial_r^2+\sigma_j^2)((1-\chi)f_j) \phi_j.$
Moreover,
 $H f= H\chi f + H(1-\chi)f.$

\begin{proof}
We use a more elaborate version of \eqref{e:rfkc}. Fix $\chi_0, \ \chi_1 \in C_c^\infty([0,1))$, both $1$ near $r=0$, such that $\chi_0\chi_1=\chi_1$. Then,
abbreviating $K(\lambda,\lambda_0)$ to $K$ and using $\chi_0 K = K$, we have
\[
(I-K(1-\chi_0)) (I+K) = I + K\chi_0 \quad \Longrightarrow \quad  (I+K)^{-1} = (I + K\chi_0)^{-1} (I-K(1-\chi_0)),
\]
so that 
\begin{equation}\label{e:rfkchi0}
 R(\lambda)  = F (\lambda,\lambda_0) (I + K(\lambda,\lambda_0)\chi_0)^{-1} (I-K(\lambda,\lambda_0)(1-\chi_0)),
\end{equation}
for $\lambda$ and $\lambda_0$ in the upper half plane, away from any poles of $R$. By  the explicit formula \eqref{e:r0def}, we see that $\lambda\mapsto R_0(\lambda) \colon \langle r \rangle^{-1/2-\epsilon}\mch \to \langle r \rangle^{1/2+\epsilon}\mch$ extends continuously from the physical space to its boundary. Hence $\lambda\mapsto  (I-K(\lambda,\lambda_0)(1-\chi_0)) \colon \langle r \rangle^{-1/2-\epsilon}\mch \to \langle r \rangle^{-1/2-\epsilon}\mch$ and $\lambda\mapsto  F(\lambda,\lambda_0) \colon \langle r \rangle^{-1/2-\epsilon}\mch \to \langle r \rangle^{1/2+\epsilon}\mch$ also extend continuously. By the analytic Fredholm theorem, 
$\lambda\mapsto  (I+K(\lambda,\lambda_0)\chi_0)^{-1} \colon \langle r \rangle^{-1/2-\epsilon}\mch \to \langle r \rangle^{-1/2-\epsilon}\mch$ also extends continuously, except possibly on a discrete set where is has poles.

To describe these poles, we use the fact that if $\Im \lambda>0$, then $\|R(\lambda)\|^{-1}$ equals the distance from $\lambda^2$ to the spectrum of $H$, where $\| \cdot \|$ denotes the operator norm $\mch \to \mch$. In particular,
\begin{equation}\label{e:resuhp}
0 < \Im \lambda \le |\Re \lambda| \Longrightarrow \|R(\lambda)\| \le 1/(2 \Im \lambda |\Re \lambda|). 
\end{equation}
Hence, if $\lambda_0 \in \mathbb R$ is not a threshold, then there exist $A \colon \mch \to \mch$ and $B(\lambda) \colon \langle r \rangle^{-1/2-\epsilon}\mch \to \langle r \rangle^{1/2+\epsilon}\mch$ such that
\[
R(\lambda) =  \frac {A}{\lambda^2 - \lambda_0^2} + B(\lambda),
\]
when $\Im \lambda \ge 0$ and $\lambda$ is sufficiently close to $\lambda_0$.
Composing on the left with $H-\lambda^2$, and matching powers of $\lambda^2 - \lambda^2_0$ as $\lambda \to \lambda_0$, gives
\begin{equation}\label{e:hab}
(H-\lambda_0^2)A = 0, \qquad I = - A  + (H-\lambda_0^2)B(\lambda_0),
\end{equation}
which implies $A = - \Pre_{\lambda_0}$. Indeed, the first of \eqref{e:hab} implies that the image of $A$ is contained in the image of $\Pre_{\lambda_0}$, so that $\Pre_{\lambda_0}A=A$. Then apply $\Pre_{\lambda_0}$ to the second of \eqref{e:hab} and use $\Pre_{\lambda_0} (H-\lambda_0^2) = 0$ to  see that $\Pre_{\lambda_0} = - \Pre_{\lambda_0} A$.

If $\lambda_0 = \sigma_j$, then near $\lambda_0$ we must use the coordinate $\tau_j(\lambda)$ to describe the possible pole there, and it follows from \eqref{e:resuhp} that  there are operators $A \colon \mch \to \mch$ and $A_0, \ B(\lambda) \colon  \langle r \rangle^{-1/2-\epsilon}\mch \to \langle r \rangle^{1/2+\epsilon}\mch$ such that
\[
R(\lambda) = \frac {A}{\tau_j(\lambda)^2} + \frac{A_0} {\tau_j(\lambda)} + B(\lambda),
\]
when $\Im \lambda \ge 0$ and $\lambda$ is sufficiently close to $\lambda_0$. Composing on the left with $H-\lambda^2$, and matching powers of $\tau_j(\lambda)$ as $\lambda \to \sigma_j$, gives \eqref{e:hab} and hence $A = - \Pre_{\lambda_0}$.
\end{proof}

Let $\Pre$ denote the orthogonal projection onto the span of the eigenfunctions of $H$, if any.  Then $R(\lambda) \Pre$ continues meromorphically from the physical space to $\zhat$ (or even just to $\mathbb C$), and  for any
$\chi = \chi(r) \in  C_c^\infty([0,\infty))$, 
$\chi R(\lambda)(I-\Pre)\chi$ is continuous for $\lambda$
 on the boundary of the 
physical space, except, perhaps, at the thresholds $\{ \pm \sigma_j\mid \sigma_j^2\in 
\spec(H_Y)\}$. We will further 
describe the possible singularities at the thresholds in Section~\ref{s:threshold}.

\subsection{The generalized eigenfunctions}\label{ss:gef}

To describe $R(\lambda)$ in more detail, we define a family of generalized
eigenfunctions $\gef_j$ of $H$ depending on the parameter
$\lambda\in \zhat$, though we shall be most interested in them for $\lambda\in\Real$, that is, $\lambda $ on the boundary
of the physical space. Recall that
 $\{ \phi_0, \; \phi_1,\dots\} \subset \mcd_Y$ is a complete orthonormal set of  eigenfunctions
of $H_Y$, with $ H_Y\phi_j=\sigma_j^2\phi_j$.
Let $\chi = \chi(r) \in C_c^\infty ([0,\infty))$ be $1$ for $r\leq 1$, and set
$$\gef^\infty_j(\lambda)=
e^{-i\tau_j(\lambda)r}\phi_j\in C^\infty(\mathbb R_+; \mathcal H_Y).$$
Now, for $\lambda$
in the physical space and 
away from poles of the resolvent 
we define a function 
$\gef_j = \gef_j(\lambda)\in \langle r\rangle ^{1/2+\epsilon }e^{\Im \tau_j(\lambda)r} \mcd$:
\begin{equation}\label{eq:gef}
\gef_j(\lambda)= (1-\chi) \gef_j^\infty(\lambda)  
- R(\lambda) [\partial_r^2,\chi]\gef_j^\infty(\lambda),\; \text{for}\; \Im \lambda >0.
\end{equation}
Then (in the sense of the comment after the statement of Lemma~\ref{l:laurentreal})
\begin{equation}\label{eq:innullspace}
(H-\lambda^2)\gef_j(\lambda)=0.
\end{equation}
It is easy to see that when $\Im \lambda>0$ and $\lambda^2$ is not an eigenvalue of $H$ then $\gef_j(\lambda)$ is independent of the choice of $\chi \in C_c^\infty([0,\infty))$ which is $1$ for $r \le 1$: If $\gef_j$, $\tilde{\gef}_j$ are
defined as in (\ref{eq:gef}) with two different such functions $\chi,\; \tilde{\chi}$, then $\gef_j - \tilde{\gef}_j \in \mch$  is in the null space of $H-\lambda^2$, and hence is $0$ by necessity.  The
functions $\gef_j$ have a meromorphic
continuation to $\zhat$ (with the weighted space to which they belong 
depending not only on the index $j$ but also on the point
in the Riemann surface $\zhat$), which we continue to denote in the same
way.  Moreover, we shall show below that
 $\gef_j(\lambda)$ is a continuous function of $\lambda$ when 
 $\lambda$ is on the boundary of the physical space and 
$|\lambda|\geq \sigma_j$.  When $\lambda$ is in the boundary of the physical space
and $\sigma_j\geq |\lambda|$, $\gef_j(\lambda)\in \langle r \rangle^{1/2+\epsilon}\mch$ for any $\epsilon>0$.

From (\ref{eq:gef}), away from poles of the resolvent we have
\begin{equation}\label{eq:endexp}
\left[\bbbone_\infty \gef_j(\lambda)\right](r)=  e^{-i\tau_j(\lambda)r} \phi_j +\sum _{\sigma_m^2\in \spec(H_Y)} \tS_{mj}(\lambda)  e^{i\tau_m(\lambda)r} \phi_m,
\end{equation}
for some functions $\tS_{mj}(\lambda)$ which 
 determine the scattering matrix.   
 For each $\lambda$ away from poles of the resolvent the 
series in (\ref{eq:endexp}) converges absolutely, uniformly for $r$ varying in a compact set, as do its derivatives with respect to $r$ and its `derivatives' with respect to $H_Y$.  For 
$\lambda \in \Real$ (ie., on the boundary of the physical space) with $|\lambda|>\sigma_j$ and away from the thresholds and 
the  poles of the resolvent, one can 
equivalently define $\gef_j(\lambda)$ to be the element of $\langle r \rangle ^{1/2+\epsilon}\mch$ which satisfies $(H-\lambda^2)\gef_j=0$ and 
which has an expansion of the form (\ref{eq:endexp}) for some $\tS_{mj}$. Note that in 
the expansion  (\ref{eq:endexp}), the terms with $\sigma_m>|\lambda|$ are exponentially decaying;
these correspond to evanescent modes.  Those with $\sigma_m<|\lambda|$ correspond to outgoing 
propagating modes.  

Both the generalized eigenfunctions $\gef_j$ and the functions $S_{mj}$ 
are meromorphic functions on $\zhat$, as can be seen from
(\ref{eq:gef}).  In particular, near a threshold
corresponding to $\sigma_k$, both $\gef_j$ and $S_{mj}$ are locally meromorphic functions
of $\tau_k(\lambda)$.




We give a self-contained proof of the following lemma, which is
known (at least in specific cases) but perhaps not explicitly in the literature.

\begin{lemma}\label{l:gefcont} If $\lambda_0\in \Real$ and $|\lambda_0|\geq \sigma_j\geq 0$,
then  $\gef_j(\lambda)$ and 
$S_{mj}(\lambda)$ 
are continuous at $\lambda_0$.
\end{lemma}
We comment that there is no restriction on the size of $\sigma_m$ compared
to $|\lambda_0|$.
\begin{proof} We give the proof for $\lambda_0\geq\sigma_j$, as the proof
for $\lambda_0\leq -\sigma_j$ is similar.

 From \cite[(3.4)]{Pa95}, or \cite[Lemma 1.2]{ace},
\begin{equation}\label{eq:smsum}\sum_{0\leq \sigma_m \leq \lambda}\tau_m(\lambda)
S_{mj}(\lambda)\overline{S}_{mk}(\lambda) = \tau_j(\lambda)\delta_{jk},\; \text{if} \; 
0\leq  \sigma_j,\sigma_k\leq \lambda.
\end{equation}
In particular, this implies 
$$\sum_{0\leq \sigma_m \leq \lambda}\tau_m(\lambda)
|S_{mj}(\lambda)|^2 =\tau_j(\lambda),\; \text{if}\; 0\leq \sigma_j\leq\lambda.$$
Thus $S_{mj}(\lambda)(\tau_m(\lambda)/\tau_j(\lambda))^{1/2}$ is bounded for $\lambda\geq \sigma_m, \sigma_j$.
Since $S_{mj}$ 
 is meromorphic on $\zhat$,  $S_{mj}(\lambda)$
 is actually continuous in this region. Thus far we have proved $S_{mj}(\lambda)$
has no poles if $0\leq \max(\sigma_j,\sigma_m)\leq \lambda$.

Now suppose there is a $\lambda_0\geq \sigma_j\geq 0$ so that $\gef_j$ has
a pole at $\lambda_0$.  Let $\sigma_{k_0}$ be any minimizer of $\{|\sigma_k-\lambda_0|:\ \mid \sigma_k^2\in \spec(H_Y)\}$, and use the coordinate $\tau_{k_0}(\lambda)$ near $\lambda_0$.
Suppose that, with respect to this coordinate $\gef_j(\lambda)$, has a pole of order $p_0>0$
at $\lambda_0$.  In particular this implies 
\begin{equation}\label{eq:sing}
f(\lambda)\defeq \gef_j(\lambda)(\tau_{k_0}(\lambda))^{p_0}
\end{equation}
is analytic with respect to the coordinate 
$\tau_{k_0}(\lambda)$ near $\lambda_0$, and $f(\lambda_0)$ is 
nontrivial. Since we have already shown $S_{mj}(\lambda)$ is regular at $\lambda_0$
if $0\leq \max(\sigma_j,\sigma_m)\leq \lambda_0$, this ensures
by (\ref{eq:endexp}) that  $f(\lambda_0)\in \mch$.
Using (\ref{eq:innullspace}), we have that $(H-\lambda_0^2)f(\lambda_0)=0$, so that $f(\lambda_0)$ is an eigenfunction
of $H$ with eigenvalue $\lambda_0^2.$

  Let $\Pre_{\lambda_0}:\mch\rightarrow \mch$ 
denote orthogonal projection
onto the span of the eigenfunctions of $H$ with eigenvalue $\lambda_0^2$.
Note that if $\eta\in \mch$ is an
eigenfunction of $H$ with eigenvalue $E$ and $\sigma_l^2 \le E$, then
\begin{equation}\label{e:etaphi}
\langle \bbbone_\infty \eta , \phi_l \rangle_{\mch_Y}(r)=0.
\end{equation}  
Hence
$\Pre_{\lambda_0}  [\partial_r^2,\chi]\gef_j^\infty(\lambda )=0$ for any $\lambda$.  
This in turn means 
$$f(\lambda_0)=\lim_{\lambda \rightarrow \lambda_0:\; \Im \lambda>0}f(\lambda) =\lim_{\lambda\rightarrow \lambda_0,\; \Im \lambda>0}(\tau_{k_0}(\lambda))^{p_0}\Pre_{\lambda_0}R(\lambda) 
[\partial_r^2,\chi]\gef_j^\infty(\lambda )=0$$ 
since $R(\lambda)$ commutes with $\Pre_{\lambda_0}$ when $\lambda$ 
is in the physical space.  But this contradicts
that $f(\lambda_0)$  is nontrivial.  Hence $\gef_j(\lambda)$ does not 
have a pole at $\lambda_0$. 
The expansion (\ref{eq:endexp}) then shows 
that all the $S_{mj}(\lambda)$ are regular at $\lambda_0$.
Hence $\Phi_j$ and $S_{mj}$ are continuous at $\lambda_0$.
\end{proof}


 This
lemma implies that if
$\sigma_k\geq \sigma_j$ and $\delta_0>0$ is sufficiently small, then
$\gef_j(\lambda)$ is a smooth function of $\tau_k(\lambda)$ on 
$ (\pm \sigma_k-\delta_0, \pm \sigma_k+\delta_0)$.




\subsection{The spectral measure}
The spectral measure for $(I-\Pre)H$ can be written in terms of the generalized
eigenfunctions $\{\gef_j\}$.
  When $H$ is the Laplacian on a manifold with infinite cylindrical ends
(or in fact on a more general $b$-manifold), 
 Lemma \ref{l:specmeas} below follows from the results of \cite[Section 6.9]{tapsit}
and \cite[Section 2]{ace}.  (See particularly \cite[(2.2)]{ace} and the 
end of the proof of Lemma 2.5.  See also
\cite[Section 5]{Lyf75}.) 

  Below  we use the notation 
$$(g\otimes h)f= g\langle f, h\rangle_{\mch}.$$
We remark that if $ f \in \langle r \rangle^{-p}\mch$ and $h\in \langle r \rangle^p\mch$ then
we can make sense of $\langle f, h\rangle_{\mch}$ 
by writing $\langle f, h\rangle_{\mch} = \langle \langle r \rangle^p f, \langle r \rangle^{-p} h\rangle_{\mch} $.
  Recall that 
when $\lambda\in \Real$, 
$R(\lambda)=R(\lambda+i0)$, where $\lambda+i0$ is a (particular) 
point on the boundary of the physical space.

The next lemma will be used to give an expression for the 
part of the  spectral 
measure corresponding to the continuous spectrum of $H$.
In order to prove it, we use another version of Vodev's identity.  Let 
$\chi_1=\chi_1(r)\in C_c^\infty(\Real_+)$
be $1$ in a neighborhood of the origin,
and let $\chi =\chi(r)\in C_c^\infty(\Real_+)$,
with $\chi \chi_1=\chi_1$.  Then
for $\lambda$ in the physical space, starting
with  (\ref{e:vodevphys}),
  using $\lambda^2-\lambda_0^2=\pr^2(\lambda)-\pr^2(\lambda_0)$ and $(\lambda^2 - \lambda_0^2)R_0(\lambda)R_0(\lambda_0) = R_0(\lambda) - R_0(\lambda_0)$ and 
multiplying on both the left and right by $\chi$
gives
\begin{equation}\label{e:vodev}
\begin{split}
\chi R&(\lambda)\chi- \chi R(\lambda_0)\chi= (\pr^2(\lambda)-\pr^2(\lambda_0))\chi 
R(\lambda)\chi \chi_1(2-\chi_1)\chi R(\lambda_0)\chi \\
&+ (1-\chi_1-\chi R(\lambda)\chi[\partial_r^2,\chi_1]) \left( \chi R_0(\lambda)\chi-\chi R_0(\lambda_0)\chi \right) (1-\chi_1+[\partial_r^2, \chi_1]\chi R(\lambda_0)\chi).
\end{split}
\end{equation}
The identity holds on all of $\zhat$ by meromorphic continuation. 

\begin{lemma}\label{l:specmeas}  Let 
$\chi = \chi(r) \in C_c^\infty([0,\infty))$. 
 Then
 for $\lambda \in \Real$, 
$\lambda \not = \pm \sigma_k$, $k\in \Natural_0$, we have
\begin{equation}\label{eq:Hresolvediff}
\frac{1}{i}\chi[R(\lambda)-R(-\lambda)](I-\Pre)\chi= \frac{1}{2} 
\sum_{0\leq \sigma_j^2\leq \lambda^2} \frac{1}{\tau_j(\lambda)}\chi\gef_j(\lambda) \otimes \gef_j(\lambda)\chi.
\end{equation}
\end{lemma}

Note that by Lemma \ref{l:laurentreal}, 
$$\chi[R(\lambda)-R(-\lambda)](I-\Pre)\chi = 
\chi[R(\lambda)-R(-\lambda)]\chi.$$
We include the $(I-\Pre)$ here to emphasize that this combination does not have 
poles arising from the embedded eigenvalues of $H$.

\begin{proof} 
Without loss of generality we may assume that $\chi$ is real valued, and that $\chi(r) = 1$ for $r\leq 2$. Choose $\chi_1\in C_c^{\infty}([0,\infty))$, also real valued,
so that $\chi_1\chi=\chi_1$ and $\chi_1=1$ for $r\leq 1$.  We shall use 
(\ref{e:vodev}).  
 We identify points on the open upper half plane (with $\lambda$ as the
parameter) with the 
physical space of $\zhat$.  Thus 
$\lambda>0$ corresponds to approaching the spectral parameter $\lambda^2$ from the upper half
plane, and $\lambda<0$ corresponds to approaching the spectral parameter $\lambda^2$ from the lower 
half plane.  

Recall that in \eqref{e:r0def} we defined  $R_0(\lambda)$ to be the resolvent for $-\partial_r^2 + H_Y$ with Dirichlet boundary conditions at $r=0$. By explicit computation,
\begin{equation}\label{eq:Q0}
R_0(\lambda)-R_0(-\lambda) =\frac{i}{2}\sum _{0\leq \sigma_j \leq \lambda} \frac{1}{\tau_j(\lambda)} 
\gef_j^0(\lambda) \otimes \gef_j^0(\lambda)
\end{equation}
where 
\begin{equation}\label{eq:gef0}\gef_j^0(\lambda)=\gef_j^0(\lambda,r)  = (e^{-i\tau_j(\lambda) r}- e^{i\tau_j(\lambda) r})\phi_j.
\end{equation}
We remark here that if $|\lambda|>\sigma_j$, then $\gef_j^0(-\lambda)=-\gef_j^0(\lambda)$, while for $|\lambda|<\sigma_j$, $\gef^0_j(-\lambda)= \gef^0_j(\lambda)$, and $\gef_j^0(\sigma_j)=0$.  
From  (\ref{e:vodev}) and (\ref{eq:Q0}), if $\lambda^2$ is not an eigenvalue of $H$,
\begin{align} &
\chi R(\lambda) \chi-\chi R(-\lambda) \chi \nonumber \\ &
= (1-\chi_1-\chi R(\lambda)\chi [\partial_r^2,\chi_1])(\chi R_0(\lambda)\chi-\chi R_0(-\lambda)\chi)
(1-\chi _1+[\partial_r^2,\chi_1]\chi R(-\lambda)\chi)\nonumber \\
& = \frac{i}{2}(1-\chi_1-\chi R(\lambda)\chi [\partial_r^2,\chi_1])\left( \sum _{0\leq \sigma_j \leq \lambda} \frac{1}{\tau_j(\lambda) }
\chi \gef_j^0(\lambda) \otimes \gef_j^0(\lambda)    \chi\right) \nonumber
(1-\chi _1+[\partial_r^2,\chi_1]\chi R(-\lambda)\chi).
\end{align}
Now we claim that 
\begin{equation}\label{eq:alternate}
\gef_j(\lambda)= (1-\chi_1)\gef^0_j(\lambda)-R(\lambda)[\partial_r^2,\chi_1]\gef_j^0(\lambda).
\end{equation}
That (\ref{eq:alternate})
holds can be checked by using the comment following 
(\ref{eq:endexp}).  Alternately, one may use that for $\lambda$ in the
physical space, the functions 
defined in (\ref{eq:gef}) and by the right hand side of (\ref{eq:alternate})
both are in the null space of $H-\lambda^2$, and differ
by an element of $\mch$, and hence must agree 
if $\lambda^2$
not an eigenvalue of $H$.  
The identity then holds for all 
$\lambda$ by analytic continuation.

   This finishes the proof if $\lambda^2$ is not an eigenvalue of $H$.  If $\lambda^2$ is an eigenvalue,
the result follows from the fact that both sides of (\ref{eq:Hresolvediff}) are continuous functions of $\lambda$ away from the thresholds.
\end{proof}
We note that a
 related proof of an analogous result for the Schr\"odinger operator on $\Real$ 
can be found in, for example, \cite[Appendix to XI.6]{r-sv3}.

\subsection{Threshold behavior}\label{s:threshold} We now discuss in more detail the behavior
of the resolvent at thresholds. Even after projecting
away from the eigenfunctions, the resolvent near the threshold
corresponding to $\sigma_j$ may have a singularity which 
is like $1/\tau_j$, and it is this singularity we wish to better understand.
Lemma \ref{l:specmeas} can be combined with a result of \cite{tapsit} to 
obtain the following corollary.  We note that we are giving a somewhat different formulation, and a rather
different proof, of part of \cite[Proposition 6.28]{tapsit} in this more 
general setting.  
\begin{lemma}\label{c:threshold}  For $\lambda $ sufficiently 
near $\pm \sigma_j$, 
$$\chi\left(  R(\lambda) (I-\Pre) - \frac{i}{4\tau_j(\lambda)}
\sum_{l: \sigma_l=\sigma_j} \gef_l(\sigma_j)\otimes \gef_l(\sigma_j) \right) \chi$$
is bounded if $\chi = \chi(r) \in C_c^\infty([0,\infty))$.
Moreover, \begin{equation}\label{eq:pm}
\gef_j(\sigma_j)=\gef_j(-\sigma_j).
\end{equation}
\end{lemma}
\begin{proof} By the Laurent expansion \eqref{e:laurentthresh},
 there is an
operator $A_0$ so that 
$$\chi\left(  R(\lambda) (I-\Pre) -\frac{A_0}{\tau_j(\lambda)} \right)\chi$$
is bounded near $\lambda= \sigma_j$.  Likewise, there is a $B_0$ so that 
$$\chi\left(  R(\lambda) (I-\Pre) -\frac{B_0}{\tau_j(\lambda)} \right)\chi$$
is bounded near $\lambda= -\sigma_j$.  

If $\sigma_j=0$, then trivially $A_0=B_0$.  So suppose temporarily
$\sigma_j>0$.
If $\lambda \in \Real $ with $0<\lambda<\sigma_j$, then 
$\tau_j(-\lambda)=\tau_j(\lambda)$.  Thus we find 
$$\lim_{\lambda \uparrow  \sigma_j}\tau_j(\lambda)\chi(R(\lambda)-R(-\lambda))(I-\Pre)\chi=
\chi(A_0-B_0)\chi.$$  
However, since by Lemma \ref{l:gefcont} $\gef_k(\lambda)$ is continuous at $\lambda=\sigma_j$ if
$\sigma_k\leq \sigma_j$, we find from (\ref{eq:Hresolvediff}) that
$$\lim_{\lambda \uparrow  \sigma_j}\tau_j(\lambda)\chi(R(\lambda)-R(-\lambda))(I-\Pre)
\chi=0.$$  Thus $A_0=B_0$.

Now let $\sigma_j \geq 0$.
If $\lambda \in \Real$, $\lambda>\sigma_j$, then $\tau_j(-\lambda)=
-\tau_j(\lambda)$, so that 
$$\lim_{\lambda \downarrow  \sigma_j}\tau_j(\lambda)\chi(R(\lambda)-R(-\lambda))(I-\Pre)\chi=
\chi(A_0+B_0)\chi= 2\chi A_0\chi.$$  Comparing  (\ref{eq:Hresolvediff}) 
this means $$A_0= \frac{i}{4} \sum_{l:\sigma_l=\sigma_j} 
\gef_l(\sigma_j)\otimes \gef_l(\sigma_j) .$$


Now we turn to proving
(\ref{eq:pm}).  Let $\chi, \chi_1\in C_c^\infty([0,\infty))$ be $1$ for 
$r\leq 1$,
and let $\Pre_{\sigma_j}$ denote projection onto
the eigenfunctions of $H$ with eigenvalue $\sigma_j^2$, if any.
By \eqref{e:laurentthresh},
$$\chi \left( R(\lambda) -\frac{A_0}{\tau_j(\lambda)}+\frac{\Pre_{\sigma_j}}{(\tau_j(\lambda))^2} \right)\chi$$
is bounded for $\lambda$ near $\pm \sigma_j$. By \eqref{e:etaphi} we have  $\Pre_{\sigma_j}[\partial_r^2,\chi_1]\gef_j^0(\lambda)=0$, where $\gef_j^0(\lambda)$ is given by \eqref{eq:gef0}.
  Now use this and  (\ref{eq:alternate}) to find
\[
\chi \Phi_j(\pm \sigma_j) = - \lim_{\lambda \to \pm \sigma_j} \chi  R(\lambda)[\partial_r^2,\chi_1]\gef_j^0(\lambda) =  - \chi A_0 [\partial_r^2,\chi_1] \lim_{\lambda \to \pm \sigma_j} \frac{ \gef_j^0(\lambda)} {\tau_j(\lambda)}  = 2i \chi A_0 [\partial_r^2,\chi_1] r \phi_j.  
\]
\end{proof}


Given $\sigma_{j_0}^2\in \spec(H_Y)$, consider the set
\begin{equation}\label{eq:halfboundstates}
\mathcal{G}_{j_0}\defeq\{ \gef_j(\sigma_j)\mid \; \sigma_j=\sigma_{j_0}\}.\end{equation}
If $\mathcal{G}_{j_0}$
 contains at least one nonzero element, we  say $\pm \sigma_{j_0}$ 
is a  threshold resonance, and the nonzero elements of this set are
 {\em resonant states} associated with $\pm \sigma_j$.
It follows from (\ref{eq:smsum}) that they have 
an expansion on the ends as in (\ref{eq:endexp}).
    The set $\mathcal{G}_{j_0}$ contains a nonzero element if and only
if  $R(\lambda)(I-\Pre) $ has a pole 
at $\sigma_{j_0}$ 
on the boundary of the physical space.  Indeed, we can 
see from Lemma \ref{c:threshold} that if the set $\mathcal{G}_{j_0}$
contains only $0$, then $R(\lambda)(I-\Pre)$ is  continuous at
 $\lambda=\sigma_{j_0}$.  If $\sigma_{j_0}^2$ is a simple eigenvalue of
$H_Y$, then the other direction is immediate.  Otherwise, if
$\sigma_{j_0}^2$ is an eigenvalue of $H_Y$ of multiplicity 
at least two, 
 see \cite[Proposition 6.28]{tapsit} to see that
the singularity at $\sigma_{j_0}$ of $R(\lambda)(I-\Pre)$ is nontrivial.

  The threshold resonant states are analogous to the familiar
half-bound states of Euclidean scattering theory, see e.g. \cite{New02}.

 

\section{Two term wave expansions}\label{s:tta}


Let $H$ be an operator as in Section \ref{ss:blackbox}
and let $u(t)$ be the solution to the wave equation
\begin{equation}\label{eq:usolution}
(\partial^2_t+H)u =0, \;
u(0) = f_1,\; u_t(0)=f_2; 
\end{equation} 
that is, $u(t) =\cos(t\sqrt{H})f_1+ \frac{ \sin(t\sqrt{H})}{\sqrt{H}}f_2.$
Here $f_1, f_2\in \mch$, though
  later we shall impose more stringent 
conditions on $f_1,\; f_2$.

We begin by recalling the contribution of the eigenvalues to $u$. In Section 
\ref{ss:wavedecaythm} we state the main theorem, the two term asymptotics 
result.  In the remainder of Section \ref{s:tta} we give the proof of 
Theorem \ref{thm:wavedecay}.

\subsection{Projection onto the eigenfunctions}\label{ss:eigenfunctions}
We begin by recalling the contribution of the eigenvalues to the behavior of 
$u$.  This requires only that $H$ is self-adjoint.

Let $\{ \ev_\ell\} $ denote the eigenvalues of $H$, repeated with 
multiplicity, with corresponding orthonormal
eigenfunctions $\{ \eta_\ell\}$: $H \eta_\ell= \ev_\ell \eta_\ell$.   For a general black box operator
$H$ the set $\{ \ev_\ell\}$ could be empty, nonempty but finite, or infinite, and examples
are known of each.
However, the assumptions on $H$ which we make in Theorem \ref{thm:wavedecay}  imply that $H$ does not have 
infinitely many eigenvalues.  
On the other hand, \cite{ch-da} contains examples of 
Schr\"odinger operators on a manifold with infinite cylindrical ends which
have finitely many 
embedded eigenvalues but which still have the type of high-energy resolvent estimate which we need for Theorem \ref{thm:wavedecay}.  

The following lemma, however, does not require high energy estimates on the cut-off resolvent, and holds
whether the set of eigenvalues is finite or infinite.
\begin{lemma}\label{l:ueigenvalues} Let $u(t)$ be the solution of 
(\ref{eq:usolution}).   Then, with $u_e(t)=\Pre u(t)$,
\begin{multline}\label{eq:ue}
u_e(t) = \sum_{\substack{\ev_\ell\in \spec_{p}(H) \\
\ev_\ell\not =0}} \eta_\ell \left( \cos((\ev_\ell)^{1/2} t) \langle f_1,\eta_\ell\rangle_{\mch} + 
 \frac{\sin( (\ev_\ell)^{1/2} t)}{(\ev_\ell)^{1/2} } \langle f_2,\eta_\ell\rangle_\mch \right)
\\ + \sum_{\substack{\ev_\ell\in \spec_{p}(H)\\\ev_\ell =0}} \eta_\ell \left( \langle f_1,\eta_\ell\rangle_\mch + t \langle f_2,\eta_\ell\rangle_\mch \right).
\end{multline}
\end{lemma}
\begin{proof}
For the initial data $f_1$, $f_2$ we can write $f_j =\Pre f_j+ (I-\Pre )f_j$.  Then, since $\Pre$ commutes with $H$, 
$\Pre u(t)=u_e(t)$, where $u_e$ satisfies 
\begin{align*}
(\partial^2_t+H)u_e(t)& =0\\ u_e(0)&=\Pre f_1,\\ (\partial_tu_e)(0)&=\Pre f_2.
\end{align*}
Then a straightforward computation shows that the explicit expression in  (\ref{eq:ue}) solves this initial value problem.
\end{proof}

	\subsection{Statement of Theorem \ref{thm:wavedecay}}\label{ss:wavedecaythm}
 
 Let 

\begin{equation}
\begin{split}
 \Cs &\defeq  \Complex \setminus \left( \bigcup_{l>0} \bigcup_{\pm} \{ \pm \nu_l -is, \; s \geq 0 \} \right) = \Complex \setminus \left( \bigcup_{j: \sigma_j>0} \bigcup_{\pm} \{ \pm \sigma_j -is, \; s \geq 0 \} \right) .
\end{split}
 \end{equation}

\begin{figure}[h]
\labellist
\small 
\pinlabel $\nu_1$ [l] at 155 85
\pinlabel $\nu_2$ [l] at 182 85
\pinlabel $\cdots$ [l] at 205 85
\pinlabel $-\nu_1$ [l] at 70 85
\pinlabel $-\nu_2$ [l] at 42 85
\pinlabel $\cdots$ [l] at 20 85
\endlabellist
 \includegraphics[width=10cm]{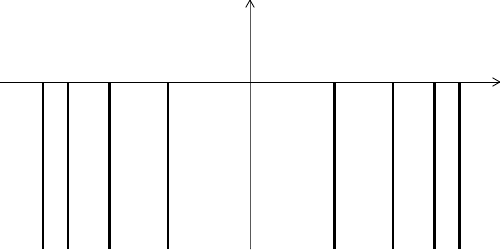} 
 \caption{The set $ \Cs $ is the complex plane with downward half-lines removed at the square roots of the nonzero eigenvalues of $H_Y$.}\label{f:cslit}
\end{figure}

The operator $\chi R(\lambda)\chi$ continues meromorphically
from $\{ \lambda \in \Complex: \Im \lambda>0\}$ to $\Cs$.
In fact, we can identify $\Cs$ with a subset of the Riemann 
surface $\zhat$.  
We use the same notation for the continuation of $R(\lambda) $ to 
$\Cs$, so that when $\lambda \in \Real$, $R(\lambda)= R(\lambda +i0)$.
We note that $\Cs $ does not require a cut at $0$, because if $\sigma_j=0$, then $\tau_j(\lambda)=\tau_0(\lambda)=\lambda$ is
already analytic on $\Complex$.  This does not mean, however, that $R(\lambda)$ must be
analytic at the origin, as it may still have a pole there.

In the following theorem, $u_e$ encodes the contribution of the eigenvalues of $H$ 
as in Lemma \ref{l:ueigenvalues} and $u_{thr}$ is the leading order contribution from the threshold 
resonances (if any).  The expansion for $u_e(t)$ is given in Lemma \ref{l:ueigenvalues}.  Recall that we have assumed that $H$ is lower semibounded.
Here we choose $M_0\in (0,\infty)$ so that $H+M_0>0$.

\begin{thm}\label{thm:wavedecay}  Let $H$ be a black box 
operator as defined in Section \ref{ss:bb},
and suppose that for some $N_1,\; N_2\in [0,\infty)$, $\lambda_0>1$, and any 
$\tilde{\chi} \in C_c^\infty([0,\infty)) $ 
there are $C_0$, $C_1$  
so that $\tilde{\chi} R(\lambda) \tilde{\chi}$ is analytic on the set
\begin{equation}\label{eq:resfree}
\{ \lambda \in \Cs\mid \; \Re \lambda>\lambda_0-1\; \text{and}\; \Im \lambda> -C_0 (\Re \lambda)^{-N_1} \}
\end{equation}
in $\zhat$, and that in this region
\begin{equation}\label{eq:resolveassumption}
\| \tilde{\chi} R(\lambda) \tilde{\chi}\| \leq C_1(1+|\lambda|)^{N_2}.
\end{equation}
 Fix $r_1>0$. For $k=1,\;2$, suppose $f_k\in (H+M_0)^{-m_k}\mch$ 
with $m_k\in \Natural$, $m_k>(N_2+4-k)/2$ and $m_k\geq (N_1+3-k)/2$ .  If the 
set $\{\nu_l\}$ is infinite, assume in addition 
 there are positive constants $N_3, C_2$ and $L_0$ so that 
\begin{equation}\label{e:weyl}
l^{1/N_3}/C_2 \le \nu_l , \text{ for } l \ge L_0
\end{equation} and
$m_k>(N_2+N_3+3-k)/2$ for $k=1,2$.
Suppose $\bbbone_{\infty}f_k$ vanishes for $r>r_1>0$.  
Let $u(t)$ be the solution of  (\ref{eq:usolution}).  Then
$$u(t)= u_e(t)+ u_{thr}(t)+u_r(t),$$ where
$u_e(t)=\Pre u(t)$ has an expansion as given in Lemma \ref{l:ueigenvalues}, and
\begin{multline}\label{eq:thr}
u_{thr}(t)= \frac{1}{4} \sum_{\sigma_j=0} \gef_j(0)\langle f_2,\gef_j(0)\rangle_\mch + 
\frac{1}{2\sqrt{t}}\sum_{\sigma_j>0} \sqrt{\frac{ \sigma_j}{2\pi}}\cos(\sigma_j t+\pi/4) \gef_j
(\sigma_j) \langle f_1,\gef_j(\sigma_j) \rangle_\mch  \\+
\frac{1}{2\sqrt{t}}\sum _{\sigma_j>0}  \frac{1}{\sqrt{2\pi \sigma_j}} \sin(\sigma_j t +\pi/4) \gef_j(\sigma_j) \langle f_2,\gef_j(\sigma_j)\rangle_\mch.
\end{multline}
Moreover, if $\chi\in C_c^\infty([0,\infty))$,
then 
\[
\| \chi(H+M_0)^{1/2}u_r(t)\|_{\mch}+ \| \chi (\partial_tu_r)(t) \|_{\mch} 
\leq Ct^{-1} \sum_{k=1}^2 \|  (H+M_0)^{m_k}(I-\Pre) f_k\|_{\mch},
\]if $t$ is sufficiently large.
\end{thm}
\begin{rmk} Note that it is possible that there are no $\sigma_j$ which
are zero, in which case the first sum in (\ref{eq:thr}) is $0$.
\end{rmk}
\begin{rmk}
We note that our assumptions on $H$ in this theorem ensure that
there are no threshold resonances  larger than $\lambda_0-1$,
or eigenvalues larger than $(\lambda_0-1)^2$.  Hence
 the sums in $u_e$ and $u_{thr}$, see
(\ref{eq:ue}) and (\ref{eq:thr}), are finite.
\end{rmk}
\begin{rmk}  The identity $R(-\overline{\lambda})=R(\lambda)^*$, which is a consequence of the self-adjointness of $H$, and
the consequent 
 symmetry of the resonances mean 
 that $\tilde{\chi} R(\lambda)\tilde{\chi}$ is analytic in the region 
$\{ \lambda \in \Cs\mid \;\Re \lambda<-(\lambda_0-1)\; 
\text{and}\;  \Im \lambda > -C_0(-\Re \lambda)^{-N_1}\}$, and satisfies
$\| \tchi R(\lambda)\tchi\| \leq C_1(1+|\lambda|)^{N_2}$ there.
\end{rmk}

\begin{rmk}\label{rmk:seemweak}
The assumption that there is a resonance-free region of the form (\ref{eq:resfree}) 
 follows from the seemingly weaker assumption that
the bound (\ref{eq:resolveassumption}) on the cut-off resolvent holds
 for $\lambda \in \Real$, $|\lambda|>\lambda_0-1$. 
By \cite[Theorem 5.6]{ch-da} this implies the existence of a resonance-free region of the form (\ref{eq:resfree}) with $N_1=N_2+1$, with a corresponding 
estimate on the cut-off resolvent there.  Note 
that by \cite[Theorem 3.1]{ch-da} and \cite[Sections 3.2, 3.3]{ch-da}
 this bound on the resolvent for  
 $-\Delta_X$ (or $-\Delta_X+V$, for a 
large class of $V\in C_c^\infty(X;\Real)$)
  holds for the examples of manifolds $X$ in Sections \ref{ex:non1} and \ref{ex:3fun1}.  
Moreover, \cite[Theorem 3.1]{ch-da} gives a more general method of constructing
  manifolds  and 
Schr\"odinger operators for which such an estimate holds.  The paper
\cite{ch-da3} includes classes  of domains  $\Omega \subset \Real^d$ for 
which the Dirichlet Laplacian satisfies the conditions of the theorem.
These include the planar waveguides described in Section \ref{ss:waveguides}.
Thus Theorem \ref{thm:t1intro} follows from Theorem \ref{thm:wavedecay}.
\end{rmk}

\begin{rmk} The assumption (\ref{e:weyl}) is a weak substitute
for a Weyl law for the eigenvalues of $H_Y$.  For example, if $(Y,g_Y)$ is 
a smooth compact manifold of dimension $d-1$ and $H_Y=-\Delta_Y$, then
the Weyl law implies
(\ref{e:weyl}) holds with $N_3=d-1$. But if the eigenvalues of $H_Y$ have high multiplicity, then \eqref{e:weyl} may also hold for a smaller value of $N_3$. For example, let
$\beta>0$ be a fixed real number, $\mu_0\in \Natural$, and let $(Y,g_Y)=
\sqcup_{\mu=1}^{\mu_0}(\Sphere^{d-1},\beta g_{\Sphere^{d-1}})$ where $\Sphere^{d-1}$ is the $d-1$-dimensional
unit sphere, and $g_{\Sphere^{d-1}}$ is the usual metric on it. Then $\nu_l = \sqrt{l(d+l-2)/\beta}$ and \eqref{e:weyl} holds with $N_3=1$, regardless of the value of the dimension $d$.
\end{rmk}


Note that without loss of generality we may assume that 
$\chi(r)=1$ for
$r<r_1$, where $\chi$
is as in the statement of the theorem.
  We do so in the remainder of this section.  In particular, this
implies that
$\chi f_k=f_k$, $k=1,2$.

\subsection{Reduction to Propositions \ref{p:Isbd} and \ref{p:Imbd}}
In this section we prove Theorem \ref{thm:wavedecay} modulo the proofs
of two propositions.
We have already found the contribution of the discrete spectrum to $u(t)$
in Lemma \ref{l:ueigenvalues}.  
We use the spectral theorem to write $(I-\Pre) u(t)$ as an integral.
This, in turn, we write as the sum of three integrals depending on the size
of the spectral parameter.  Each of these three will be evaluated
or  bounded using a 
different technique.

\begin{lemma}\label{l:contspec} Let $u(t)$ be the solution of (\ref{eq:usolution}).
Then
$$\left( \begin{array}{c} 
(I-\Pre)u(t) \\
\partial_t(I-\Pre) u(t)\end{array}\right) = \PV \frac{1}{2\pi i} \int_{-\infty}^\infty 
e^{it\lambda} A(\lambda) (R(\lambda)- R(-\lambda)) (I-\Pre) d\lambda\left( \begin{array}{c}f_1\\f_2 
\end{array} \right) $$
where 
\begin{equation}\label{eq:A}
A(\lambda) = \left( \begin{array}{cc} \lambda & -i \\
i\lambda^2  & \lambda
\end{array}\right)
\end{equation}
and $\PV$ is the principal value.
\end{lemma}
\begin{proof}
We have
$$u(t) =\cos(t\sqrt{H})f_1+ \frac{ \sin(t\sqrt{H})}{\sqrt{H}}f_2.$$
By the functional calculus and Stone's formula,
\begin{align}
\nonumber 
\cos (t\sqrt{H})(I-\Pre) & = \frac{1}{2\pi i} \int_0^\infty 
\cos(t\sqrt{\tau}) [ (H-\tau -i0)^{-1} -(H-\tau +i0)^{-1}] (I-\Pre) d\tau\\
 & = \frac{1}{2\pi i} \int_0^\infty [e^{it\lambda} + e^{-it\lambda}]
[R(\lambda)-R(-\lambda)](I-\Pre) \lambda d\lambda  \nonumber \\
 & =\frac{1}{2\pi i} \int_{-\infty}^\infty e^{it\lambda} 
[R(\lambda)-R(-\lambda)](I-\Pre) \lambda d\lambda.
\end{align}

As in the statement of Lemma \ref{l:specmeas}, the factor $I-\Pre$ can be omitted in the last 
two instances, but it is needed in the first two.  We continue to often include it even 
when it is unnecessary, both to
emphasize that there are no poles due to 
eigenvalues, and because some later manipulations will require $I-\Pre$.

Similarly 
\begin{align*}
\frac{\sin (t\sqrt{H})}{\sqrt{H}}(I-\Pre) & = \frac{-1}{2\pi }
 \int_0^\infty [e^{it\lambda} - e^{-it\lambda}]
[R(\lambda)-R(-\lambda)](I-\Pre)  d\lambda \\
& = -\PV \frac{1}{2\pi}\int_{-\infty}^\infty e^{it\lambda} [R(\lambda)-R(-\lambda)](I-\Pre)  d\lambda .
\end{align*}
Here we do need the principal value if $H$ has $0$ as a threshold resonance,
since in that case $(R(\lambda)-R(-\lambda))(I-\Pre)$ has a pole of order 
$1$ at $0$.
\end{proof}

We use the integral representation from Lemma \ref{l:contspec} to 
write $(I-\Pre) u(t)$ as the sum of three terms:
\begin{equation}\label{eq:Ismlsum}
 (I-\Pre) \left(\begin{array} {c}u(t) \\
\partial_t u(t)\end{array} \right)   = 
\left(  I_s(t) +I_m(t) +I_l(t) \right) \left( \begin{array}{c} f_1 \\f_2
\end{array} \right) 
\end{equation}
where
\begin{align*}
I_s(t) & =  \PV \frac{1}{2\pi i} \int_{|\lambda|<\lambda_0}
e^{it\lambda} A(\lambda) (R(\lambda)- R(-\lambda)) (I-\Pre) d\lambda\\
I_m(t) & = \frac{1}{2\pi i}  \int_{\lambda_0< |\lambda|< t^{\epsilon} }e^{it\lambda}  A(\lambda) (R(\lambda)- R(-\lambda)) (I-\Pre) d\lambda
\\
I_l(t) & = \frac{1}{2\pi i}\int_{ t^\epsilon< |\lambda| } e^{it\lambda}  A(\lambda) (R(\lambda)- R(-\lambda)) (I-\Pre) d\lambda   .
\end{align*}
Here $\epsilon>0$ is a constant to be determined later
and $\lambda_0 > 1$ is 
as in the statement of
Theorem \ref{thm:wavedecay}.  We shall later add an additional 
assumption on $\lambda_0$ for convenience.
The subscripts $s,\; m,\; l$ stand for small, medium, and
large, and refer to the size of $|\lambda|$.   We shall bound each of these in turn, beginning with the easiest.  Note that for $I_m$, $I_l$ we may omit
the $I-\Pre$ as $H$ has no eigenvalues which are greater than or equal to $\lambda_0^2$.

Recall $M_0\in (0,\infty)$ was chosen so that $H+M_0>0$. 
\begin{lemma}\label{l:Il}
If $\min(2m_1-1,2m_2)\geq 0$, then
\begin{multline}\label{eq:Ilbd}
\left\|\left( \begin{array}{c c}  (H+M_0)^{1/2} & 0 \\
0 & I \end{array} \right) 
  I_l(t)\left( \begin{array}{c c}(H+M_0)^{-m_1} & 0 \\
0 &(H+M_0)^{-m_2} \end{array}\right) \right\|_{\mch \oplus \mch \rightarrow \mch\oplus \mch} \\ =
 O(t^{-\epsilon(\min(2m_1-1, 2m_2))})
\end{multline}
\end{lemma}
\begin{proof}
 We note 
that for $m\in \Natural$, 
\begin{multline}\label{eq:Qmidentity}
\chi (R(\lambda)-R(-\lambda) )\chi= \chi (R(\lambda)-R(-\lambda))(H+M_0)^{-m}(H+M_0)^{m} \chi \\ =
\chi (R(\lambda)-R(-\lambda))(\lambda^2+M_0)^{-m}(H+M_0)^{m} \chi.
\end{multline}

By the spectral theorem, using that for $\lambda\gg 0$, 
$\lambda (R(\lambda) - R(-\lambda))d\lambda$ is, up to a constant multiple, the spectral measure for $H$,
we find for $m,\;j,\;k\in \Integers$
\begin{align*} 
& \left \| (H+M_0)^{j/2}\int_{|\lambda|>t^\epsilon} e^{it\lambda}
\lambda^{k+1} (R(\lambda )-R(-\lambda) )d\lambda (H+M_0)^{-m}\right\|_{\mch \rightarrow \mch} \\ & =
\left \| \int_{|\lambda|>t^\epsilon} e^{it\lambda}\lambda^k  (\lambda^2+M_0)^{-m+j/2}\lambda (R(\lambda) -R(-\lambda) )d\lambda \right\|_{\mch \rightarrow \mch} \\ & 
= 4\pi \sup_{|\lambda |>  t^\epsilon} |\lambda|^k(M_0+\lambda^2)^{-m+j/2} = O(t^{(-2m+k+j)\epsilon})
\end{align*}
if $k+j-2m\leq 0$.  Here we use that a nontrivial spectral projector
has norm $1$.

The lemma follows from this and the expression for $I_l(t)$.
\end{proof}


Evaluating and bounding the contributions of $I_s$ and $I_m$ requires more effort and each will be studied separately.  We state the results here.
\begin{prop}\label{p:Isbd}
Let $f_1,\; f_2, \chi$, $\lambda_0$ and $u_{thr}$ be as  in Theorem 
\ref{thm:wavedecay}.  In addition, suppose $\lambda_0\not = \sigma_j$
for any $j$.
Then there is a constant $C$
(depending on $\chi$) so that 
\begin{equation} \label{eq:Isbd}\left \| 
\chi\left( \begin{array}
{c c}
(H+M_0)^{1/2} & 0\\ 0 & I \end{array} \right)  \left( I_s(t) \left( \begin{array}{c} f_1 \\ f_2 \end{array} \right) 
- \left( \begin{array} {c} u_{thr}(t) \\ \partial_t u_{thr}(t) \end{array}\right) \right) \right\|_{\mch \oplus \mch} 
\leq C t^{-1} \left\| \left( \begin{array} {c} f_1 \\f_2\end{array}\right)\right\|_{\mch \oplus \mch}
\end{equation}
when $t$ is sufficiently large.
\end{prop} 
The term 
 $u_{thr}$ corresponds to possible resonances in $[0,\lambda_0)$
and if it is nontrivial decays at fastest at a rate proportional to $t^{-1/2}$.
We remark that the error 
$O(t^{-1})$ in the estimate of Proposition \ref{p:Isbd} is sharp
and is due to the 
discontinuous nature of the cut-off at $\lambda =\pm \lambda_0$ in the definition of $I_s(t)$.  The
error could be improved by instead using a smooth cut-off function 
in the $\lambda$ variable  to 
define  $I_s(t)$ and $I_m(t)$; we do something similar in Section \ref{s:exp2}. 
 However, since 
our methods for estimating the contribution of $I_m(t)$ result in an error 
of size $O(t^{-1})$ even with
this change,  we would not gain by taking this alternate approach here.

For $I_m(t)$, the corresponding result is
\begin{prop}\label{p:Imbd}  Assume (\ref{eq:resfree}), (\ref{eq:resolveassumption}),
and (\ref{e:weyl}).
Let $0<\epsilon<1/N_1$ and
let  $\lambda_0$ be as in the statement of Theorem \ref{thm:wavedecay}.
If  $H_Y$ is bounded so that the set 
$\{\nu_l\}$ is finite, by increasing $\lambda_0$ if 
necessary, choose $\lambda_0^2>\|H_Y\|$.
Choose $m_k\in \Natural$ so that $m_k>(N_2+4-k)/2$ for $k=1,2$, and, if the
set
$\{\nu_l\}$ is infinite, assume in addition $m_k>(N_2+N_3+3-k)/2$.
Then for $t$ sufficiently large
\begin{equation}
\left\| \chi \left( \begin{array} { c c} 
(H+M_0)^{1/2} & 0 \\
0 & I \end{array}\right)I_m(t)   \left( \begin{array} { c c} 
(H+M_0)^{-m_1} & 0 \\
0 &(H+M_0)^{-m_2}  \end{array}\right)  \chi \right\|_{\mch \oplus \mch 
\rightarrow \mch \oplus \mch }
\leq Ct^{-1}.
\end{equation}
\end{prop}
It is only in the proof of Proposition \ref{p:Imbd} that we use the assumptions
(\ref{eq:resfree}), (\ref{eq:resolveassumption}) and (\ref{e:weyl}).
We prove Proposition \ref{p:Isbd} in Section \ref{ss:Is} and Proposition 
\ref{p:Imbd} in Section \ref{ss:Im}.

Assuming these two propositions, we may now prove Theorem \ref{thm:wavedecay}.

\vspace{2mm}
\noindent {\em Proof of Theorem \ref{thm:wavedecay}.}
 Writing 
$u(t)=\Pre u(t)+(I-\Pre)u(t)$,   Lemma \ref{l:ueigenvalues} gives an explicit
expression for $u_e(t)=\Pre u(t)$. 

Recall that $\bbbone_{\infty}f_1,\; \bbbone_{\infty}f_2$ vanish 
for $r$ sufficiently 
large, and  without loss of generality we have chosen $\chi$ so that $\chi f_k= f_k$; in 
particular, $(H+M_0)^{m_k} \chi f_k = (H+M_0)^{m_k}  f_k
= \chi(H+M_0)^{m_k}  f_k$.
From  (\ref{eq:Ismlsum}) we see that to understand
$(I-\Pre) u(t)$ it suffices to understand 
 the contributions of $I_s(t)$, $I_m(t)$, and $I_l(t)$.  Here we choose
$\epsilon=1/(N_1+1)$.  If $H_Y$ is bounded, we choose $\lambda_0$ both 
satisfying
the conditions of the theorem and so that $\lambda_0^2>\|H_Y\|$.
If $H_Y$ is unbounded, choose $\lambda_0$ satisfying the conditions
of the theorem, and so that $\lambda_0\not = \sigma_j$ for any $j$.

Then, with $m_k\geq (N_1+3-k)/2$, $k=1,2$, the bound of $Ct^{-\epsilon\min(2m_1-1,2m_2)}$ on $I_l(t)$ from
(\ref{eq:Ilbd})  is less than 
or equal to $Ct^{-1}$.  The results
of Propositions \ref{p:Isbd} and  \ref{p:Imbd} complete the proof.
\qed

\subsection{The contribution of $I_s(t)$}\label{ss:Is}

The goal of this section is to prove Proposition \ref{p:Isbd}.  We 
remark that to prove this proposition, we do not need to assume
bounds on the resolvent of $H$ at high energy.  Moreover, 
note that the bound is given in terms of $\|f_1\|_\mch$, $\|f_2\|_{\mch}$,
so that we do not need the Sobolev-type norms $\| (H+M_0)^{m_k}f_k\|_{\mch}$ in 
this section.

In order to prove the proposition, we shall  write the integral
defining $I_s$
as the sum of three types of terms: an integral over a small neighborhood 
of $0$, an integral over a small neighborhood of $\nu_l$ or $-\nu_l$
  for each $\nu_l \in (0, \lambda_0)$, and an integral of a function with support disjoint from all
$\pm \sigma_j$ with $0\leq \sigma_j<\lambda_0$.  Recall that the integrand is 
continuous near $\lambda_0$, since we have assumed that no embedded resonances exceed 
$\lambda_0-1$.


Let $\psi \in C_c^{\infty}(\Real)$ be a function which is $1$ in a neighborhood
of $0$, and set
$$I_{0}(t)  =  \PV \frac{1}{2\pi i} \int \psi(\lambda)
e^{it\lambda} A(\lambda) (R(\lambda)- R(-\lambda)) (I-\Pre) d\lambda
.$$
Note that if  $\sigma_0>0$, then $I_0(t)=0$ if
the support of $\psi$ is chosen sufficiently small. 

We emphasize that the following lemma does not require high energy
resolvent estimates for $H$, and is valid for any $f_1,\; f_2\in \mch$ which
have $\bbbone_{\infty}f_1,\;\bbbone_{\infty}f_2$  both supported for
$r$ in 
a fixed compact subset of $[0,\infty)$.
\begin{lemma}\label{l:I0}
Let $H$ be any operator satisfying the hypotheses of Section \ref{ss:blackbox},
and let $f_1,\; f_2\in \mch$ have $\bbbone_{\infty}f_1,\;\bbbone_{\infty}f_2$  both supported in $r\leq r_1$.
Then with the support of $\psi$ chosen sufficiently small, for any $q
\in \Natural_0,$ $k\in \Natural$ there
is a constant $C$ depending on $q,\;k$  and the support of $f_1$, $f_2$ so that 
\begin{multline*}
\left\|  \chi(H+M_0)^{q/2} I_0(t)  \left( \begin{array} {c} f_1\\ f_2\end{array} \right) 
-\frac{1}{4} \left( \begin{array} {c} \chi \sum_{\sigma_j=0}M_0^{q/2} \Phi_j(0)\langle f_2,\Phi_j(0)\rangle
\\ 0 
\end{array} \right) \right \|_{\mch \oplus \mch}
\\ \leq C t^{-k}(\|  f_1\|_\mch+\|  f_2\|_\mch )
\end{multline*}
when $t$ is sufficiently large.
\end{lemma}
\begin{proof}
Recall that $\chi (R(\lambda)-R(-\lambda))(I-\Pre) \chi$ has a singularity at worst like $1/\lambda$ at $\lambda=0$.  Thus,
 if $p\in \Natural$, then 
$\lambda^p (\lambda^2+M_0)^{q/2}\psi(\lambda) \chi(R(\lambda)-R(-\lambda))(1-\Pre) \chi$ is a smooth
function of $\lambda\in \Real$ if the support of $\psi$ is chosen sufficiently 
small that it contains no $\pm\sigma_j$ with $\sigma_j\not =0$.
Hence for $p\in \Natural$, by integrating by parts $k$ times, we find
for $t>0$
\begin{multline}\label{eq:pneg}
\left \|\int_{-\infty}^\infty  e^{it \lambda} \lambda^p (\lambda^2+M_0)^{q/2}\psi(\lambda) \chi (R(\lambda)-R(-\lambda) )(I-\Pre)\chi d\lambda \right\|_{\mch \rightarrow \mch} \\ \leq t^{-k}\int_{-\infty}^\infty
\left\| \frac{d^k }{d\lambda^k}\left(\lambda^p (\lambda^2+M_0)^{q/2}\psi(\lambda) \chi (R(\lambda)-R(-\lambda) )(I-\Pre)\chi\right) \right\|_{\mch\rightarrow \mch}
d\lambda =O(t^{-k})
, \; p\in \Natural
\end{multline}
for any $k\in \Natural$. 

Using (\ref{eq:pneg}) and considering the expression (\ref{eq:A}) for $A$, this
means we need only consider more carefully the entry corresponding to the 
upper right-hand corner of $A$.  We also will use, see 
Lemma \ref{l:specmeas},
$$(R(\lambda)-R(-\lambda) )(I-\Pre)= \frac{i}{2\lambda} \sum_{\sigma_j=0} 
\Phi_j(0)\otimes \Phi_j(0) +B(\lambda)$$
where $\chi B(\lambda)\chi $ is analytic in a neighborhood of $\lambda=0$.
Hence, using another integration by parts argument,
\begin{multline}\label{eq:reduction1}
\PV \int_{-\infty}^\infty  e^{it\lambda}(\lambda^2+M_0)^{q/2} \psi(\lambda) 
\chi (R(\lambda)-R(-\lambda) )(I-\Pre)\chi d\lambda
=\\\frac{iM_0^{q/2}}{2} \sum_{\sigma_j=0} 
\chi \gef_j(0)\otimes \chi \gef_j(0) \PV \int_{-\infty}^\infty  e^{it\lambda}
 \frac{\psi(\lambda) }{\lambda}
 d\lambda +O(t^{-k})
\end{multline}
when the support of $\psi$ is sufficiently small.

Now we use 
$$\PV \int_{-\infty}^\infty e^{it\lambda}\frac{1}{\lambda}d\lambda= i\pi, \; t>0$$
and the fact that $\psi$ is $1$ in a small neighborhood of the origin to 
find that 
\begin{multline*}
\PV \int_{-\infty}^\infty  e^{it\lambda} (\lambda^2+M_0)^{q/2}\psi(\lambda) 
\chi (R(\lambda)-R(-\lambda) )(I-\Pre)\chi d\lambda\\= \frac{-\pi M_0^{q/2}}{2} \sum_{\sigma_j=0} 
\chi \Phi_j(0)\otimes \chi \Phi_j(0) +O(t^{-k})
\end{multline*}
for $t$ sufficiently large.
\end{proof}


The next lemma follows directly from the more general Lemma 
\ref{l:allinonestatphase}.
\begin{lemma}\label{l:simplestatphase}
Let $\mcx$ be a Banach 
space, $\sigma_j>0,$ and set $B_0=\{ z\in \Complex\mid |z|<\min(\sigma_j,1)/2\}$.
If $F\in C_c^\infty(B_0;\mcx)$ then there is a  $C>0$
 so that
$$\left\| \int _0^\infty e^{- i\lambda t}\frac{ F(\tau_j(\lambda))}{\tau_j(\lambda)}
d\lambda - (\sigma_j t)^{-1/2} e^{-i\pi/4} \sqrt{2\pi}F(0)e^{-i\sigma_j t}\right\|_{\mcx} \leq C t^{-1},\; t>0. $$
Moreover, 
$$\left\| \int _0^\infty e^{ i\lambda t}\frac{ F(\tau_j(\lambda))}{\tau_j(\lambda)}
d\lambda \right\|_{\mcx} \leq 
 C t^{-1},\; t>0.$$
\end{lemma}

\vspace{2mm}
\noindent{\em Proof of Proposition \ref{p:Isbd}}.  
Let $\psi\in C_c^\infty(\Real;[0,1])$ be equal to $1$ in a small 
neighborhood of the origin. 
Set $L=\max \{l\mid\nu_l<\lambda_0\}$, and, for $l=1,...,L$, set
$\psi_l(\lambda)= \psi(|\lambda^2-\nu_l^2|)$. Note that $\psi_l$ is smooth since $\psi$ is $1$ in
a neighborhood of the origin. 
Choose the support of $\psi$ sufficiently small
that 
$$0< l, \; l'\leq L, \; l\not = l' \Rightarrow 
\supp \psi_l \cap \supp \psi_{l'}=\emptyset \;\text{and} \supp \psi 
\cap \supp \psi_l=\emptyset.$$
Then set
$\psi_s(\lambda)= \psi(\lambda)+ \sum_{l=1}^L \psi_l(\lambda).$
By shrinking the support of $\psi$ if necessary, we can assume that 
$\psi_s$ is $0$ in a neighborhood of $\pm \lambda_0$.
Note that $(R(\lambda)-R(-\lambda))(I-\Pre)(1-\psi_s(\lambda))$ is 
smooth on $[-\lambda_0, \lambda_0]$ by our choice of
$\psi_s$ and since $\lambda_0^2$
is not an eigenvalue of $H_Y$.

To simplify notation, for $q\in \Natural_0$, set 
$$A_q(\lambda)=(\lambda^2+M_0)^{q/2}A(\lambda).$$
Hence, integrating by parts,
\begin{align}\label{eq:negl}
& \left \| \int_{-\lambda_0}^{\lambda_0} 
e^{i t \lambda}A_q(\lambda) (1-\psi_s(\lambda))\chi ( R(\lambda)-R(-\lambda))(I-\Pre) \chi d\lambda \right\|_{\mch\oplus \mch \rightarrow \mch\oplus \mch} \nonumber \\&
= \left\|  \frac{-i}{t}[e^{i t  \lambda_0}A_q(\lambda_0)+e^{-it\lambda_0} A_q(-\lambda_0) ]\chi ( R(\lambda_0)-
R(-\lambda_0))(I-\Pre) \chi  \right. \nonumber \\ &  \hspace{5mm}
\left.  +\frac{i}{t} 
\int_ {-\lambda_0}^{\lambda_0} e^{i t \lambda}\frac{d}{d\lambda}
\left( A_q(\lambda)(1-\psi_s(\lambda))\chi ( R(\lambda)-R(-\lambda))(I-\Pre) \chi
\right)  d\lambda \right\|_{\mch\oplus \mch \rightarrow \mch\oplus \mch} \nonumber
\\ &  =O(t^{-1})
\end{align}
since $$\left \| \frac{d}{d\lambda}
\left(A_q(\lambda) (1-\psi_s(\lambda))\chi ( R(\lambda)-R(-\lambda)) (I-\Pre)\chi
\right) \right\|_{\mch\oplus \mch \rightarrow \mch\oplus \mch} \in L^1([-\lambda_0,\lambda_0]).$$

We will now focus on neighborhoods of $\pm \nu_l$. Consider 
$\int_{-\lambda_0}^{\lambda_0} e^{it \lambda} A_q(\lambda)
\psi_l(\lambda ) (R(\lambda)-R(-\lambda) )(I-\Pre)
d\lambda$ with $0< l \leq L$.
Let $j=j(l)\in \Natural$ be such that $\nu_l=\sigma_j$.  By 
our choice of the support properties of $\psi$, 
$\tau_j(\lambda) \psi_l(\lambda ) R(\lambda)(I-\Pre)$ is a smooth 
function of $\tau_j(\lambda)$ for $\lambda \in \Real$, but 
$\tau_j(\lambda) \psi_l(\lambda ) R(-\lambda)(I-\Pre)$ is not.  The reason 
for this second is that 
for $\lambda \in \Real$, $\tau_j(-\lambda)/\tau_j(\lambda)= 1$ if 
$0<|\lambda|<\sigma_j$,
and $\tau_j(-\lambda)/\tau_j(\lambda)=-1$ for $|\lambda|>\sigma_j$.  Hence 
we do a change of variable:
\begin{multline}
\int_{-\lambda_0}^{\lambda_0} e^{it \lambda} A_q(\lambda)
\psi_l(\lambda) (R(\lambda)-R(-\lambda) )(I-\Pre)
d\lambda \\
= \int_{-\lambda_0}^0 (e^{i t\lambda}A_q(\lambda)-e^{-it \lambda} A_q(-\lambda))
\psi_l(\lambda) R(\lambda)(I-\Pre) d\lambda  \\
+ \int^{\lambda_0}_0 (e^{i t\lambda}A_q(\lambda)-e^{-it \lambda} A_q(-\lambda))
\psi_l(\lambda)R(\lambda) (I-\Pre) d\lambda.
\end{multline}

By shrinking the support of $\psi$ if 
necessary, we may apply Lemma \ref{l:simplestatphase}
to the second integral on the right-hand side,
 with $F(\tau)= \chi \tau A_q((\tau^2+\sigma_j^2)^{1/2})\psi(|\tau|^2)R(\lambda(\tau))(I-\Pre)\chi$, where 
$\lambda(\tau)$ is the locally well-defined inverse of $\zhat\ni \lambda\mapsto \tau_j(\lambda)\in \Complex$.  Thus, using the meromorphic continuation
of $\chi R \chi$ to $\zhat$ and Lemma \ref{c:threshold},
 $F$ is a smooth function
supported in a complex neighborhood of the origin.

From Lemma \ref{l:simplestatphase}, for $t>0$
\begin{multline}\label{eq:lint1}
\int^{\lambda_0}_0 e^{- i t\lambda}A_q(-\lambda)
\psi_l(\lambda)\chi R(\lambda)(I-\Pre)\chi d\lambda= \int^{\lambda_0}_0 e^{- i t\lambda}
A_q(- \lambda)
\psi(|\tau_j(\lambda)|^2)\chi R(\lambda)(I-\Pre) \chi d\lambda\\
= \sqrt{2\pi}e^{- i(\sigma_jt+\pi/4) }A_q(- \sigma_j)[\chi R(\lambda)(I-\Pre)\chi \tau_j(\lambda)]\restrict_{\lambda=\sigma_j}(\sigma_j t)^{-1/2}+B_{l,1}(t)
\end{multline}
where $\| B_{l,1}(t) \|_{\mch\oplus \mch\rightarrow \mch\oplus \mch}= 
O(t^{-1})$.
By a result parallel to Lemma \ref{l:simplestatphase}, for $t>0$,
\begin{multline}
\int_{-\lambda_0}^0 e^{- i t\lambda}A_q( -\lambda)
\psi_l(\lambda)\chi R(\lambda)(I-\Pre ) \chi d\lambda =\int_{-\lambda_0}^0 e^{- i t\lambda}
A_q(- \lambda)
\psi(|\tau_j(\lambda)|^2)\chi R(\lambda) (I-\Pre )\chi d\lambda\\
= -\sqrt{2\pi}e^{ i(\sigma_j t+\pi/4) }A_q( \sigma_j)
[\chi R(\lambda)(I-\Pre)\chi \tau_j(\lambda)]\restrict_{\lambda=-\sigma_j}(\sigma_j t)^{-1/2}+B_{l,2}(t)
\end{multline}
where $\|  B_{l,2}(t) \|= O(t^{-1})$.
From Lemma \ref{l:simplestatphase} (and the analogous
result for the integral over $\lambda<0$)
\begin{equation}
\left\| \chi \int_{-\lambda_0}^{\lambda_0} e^{ i t\lambda}A_q( \lambda)
\psi_l(\lambda )R(\lambda)(I-\Pre )\chi  d\lambda\right\|_{\mch\oplus \mch\rightarrow \mch\oplus \mch}=O(t^{-1}),\; t\rightarrow \infty.
\end{equation}


It follows from 
Lemma \ref{c:threshold}
that 
\begin{equation}\label{eq:resatthresh}
(\tau_j (\lambda)\chi R(\lambda)(I-\Pre) \chi)\restrict_{\lambda =\pm \sigma_j}=\frac{i}{4} \sum_{j': \sigma_{j'}=\sigma_j= \nu_l} \chi \Phi_{j'}(\sigma_j)\otimes \chi
\Phi_{j'}(\sigma_j).
\end{equation}
Hence from (\ref{eq:lint1}-\ref{eq:resatthresh}), we have
\begin{multline}\label{eq:lcont}
\int_{-\lambda_0}^{\lambda_0} e^{it \lambda} A_q(\lambda) \psi_l(\lambda) \chi (R(\lambda)-R(-\lambda))(I-\Pre)\chi d\lambda \\
= (\sigma_j t)^{-1/2} \frac{i}{2} \sqrt{\pi/2} \sum_{j: \sigma_j=\nu_l} \left[
e^{i(\sigma_jt+\pi/4)}A_q(\sigma_j)- e^{-i(\sigma_jt+\pi/4)}A_q(-\sigma_j)\right] \chi \Phi_j(\sigma_j)\otimes \chi \Phi_j(\sigma_j)
+B_{l,\chi}
\end{multline}
where $\| B_{l,\chi}\|=O(t^{-1}).$

Using (\ref{eq:negl}), (\ref{eq:lcont}) and Lemma \ref{l:I0} proves Proposition
\ref{p:Isbd}.  \qed

\subsection{The contribution of $I_m(t) $}\label{ss:Im}
The main result of this section is Proposition \ref{p:Imbd}, which provides the needed bound 
on $I_m(t)$.  This is the portion of the proof of Theorem \ref{thm:wavedecay}
for which we use the resonance-free region and the high energy 
resolvent estimate, and we assume these hold for all results in this section.
We have some freedom in our choice of $\lambda_0$ -- we can always choose
a larger value. 
 The choice of $\lambda_0$
in Proposition \ref{p:Imbd} is made to simplify the proof a bit.

In order to prove the bound,
we shall perform a contour deformation
into $\Cs$.  
We recall that 
\begin{equation} \label{eq:ImQs}
I  _m(t)  = \frac{1}{2\pi i } \int_{\lambda_0<|\lambda| < t^{\epsilon}}e^{it\lambda}A(\lambda) 
[ R(\lambda)-R(-\lambda)] d\lambda.
\end{equation}
There is no need to compose on the right with $(I-\Pre)$ here,
since  $\lambda_0^2$ exceeds the largest eigenvalue of $H$.

In order to simplify notation, 
set
\begin{equation}\label{eq:G}
G(\lambda)= \left( \begin{array} { c c} 
(\lambda^2+M_0)^{1/2} & 0 \\
0 & 1 \end{array}\right) A(\lambda)  \left( \begin{array} { c c} 
(\lambda^2+M_0)^{-m_1} & 0 \\
0 &(\lambda^2+M_0)^{-m_2}  \end{array}\right) 
\end{equation}
and
$$
R_\chi(\lambda)= \chi R(\lambda)\chi.$$
Then  
\begin{multline}\label{eq:withG}
\left\| \chi \left( \begin{array} { c c} 
(H+M_0)^{1/2} & 0 \\
0 & I \end{array}\right)I_m(t)   \left( \begin{array} { c c} 
(H+M_0)^{-m_1} & 0 \\
0 &(H+M_0)^{-m_2}  \end{array}\right)  \chi \right\|_{\mch \oplus \mch 
\rightarrow \mch \oplus \mch }\\
= 
\left\|  \frac{1}{2\pi i } \int_{\lambda_0<|\lambda| <
 t^{\epsilon}}e^{it\lambda}G(\lambda) 
[ R_\chi(\lambda)-R_\chi(-\lambda)]d\lambda \right\|_{\mch \oplus \mch 
\rightarrow \mch \oplus \mch }.
\end{multline}

We treat the contributions from the terms $R_\chi(\lambda)$ and $R_\chi(-\lambda)$ separately; the second
is substantially more difficult than the first.

To bound the term in (\ref{eq:ImQs}) with $e^{it\lambda}R_\chi(\lambda)$ we shall
use the following lemma.
\begin{lemma}\label{l:easycontour}
Let $\lambda_0>0$ be as in the statement of Theorem \ref{thm:wavedecay}.
Then, if $\epsilon,\;t>0$ then there is a constant $C$ so that 
\begin{equation}\label{eq:easycontourest}
\left\| \int_{\lambda_0<|\lambda|<t^\epsilon} 
e^{it\lambda} G(\lambda) R_\chi(\lambda) d\lambda \right\| _{\mchs \rightarrow \mchs}
\leq Ct^{-1} 
\end{equation}
if $N_2+\max(2-2m_1, 1-2m_2)\leq 0$ and $t$ is sufficiently large.
\end{lemma} 
We postpone the  proof of this lemma to Section \ref{ss:segmentintegral}.

The proof of Lemma \ref{l:easycontour} uses a contour deformation
argument, deforming the contour into the upper half-plane where $e^{it\lambda}$
decays in $t$.
To bound an integral of the form
\begin{equation}\label{eq:twotobound}\int^{-\lambda_0}_{-t^\epsilon} G(\lambda)   e^{it\lambda}
R_\chi(-\lambda) d\lambda\; \text{or}\; \int_{\lambda_0}^{t^\epsilon} G(\lambda)   e^{it\lambda}
R_\chi(-\lambda) d\lambda
\end{equation}
in a similar way, we run into the problem that $R_\chi(-\lambda)$ must
be evaluated at a point with $\Im(-\lambda)<0$.  Since the continuation
of $R_\chi$ is to $\Cs$, we see that this is complicated.
 Each {\em distinct } value of 
$\sigma_j>0$ gives ramification points at $\pm \nu_l$ in $\zhat$; this 
corresponds to the omitted rays in the lower 
half plane in $\Cs$, and we must stay away from these omitted rays.  Instead,
a contour deformation argument gives us the following proposition, which 
we prove in  Section \ref{ss:segmentintegral}.  The values of $a$ and $b$ are
chosen so that we can avoid the omitted rays in $\Cs$.

\begin{prop}\label{p:segmentintegral}
Let $0<\epsilon<1/N_1$ 
and set
$$p=N_2+\max(2-2m_1,1-2m_2)$$
where $N_1,\; N_2$ are as in the statement of Theorem \ref{thm:wavedecay}.
  Then there are constants $T,\; C>0$ so that if
 $t>T$ then
\begin{equation}\label{eq:segment1}
\left\| \int_{a\leq |\lambda|\leq b} 
e^{i\lambda t}G(\lambda ) R_\chi(-\lambda)  d\lambda\right\|_{\mchs \rightarrow \mchs} 
\leq  C t^{-1}\left(a^{p} + b^{p}+ \int_{a}^bs^{p}
ds\right)
\end{equation}
whenever $a$ and $b$ satisfy $\max(\lambda_0,\nu_l)< a <b<
\min(\nu_{l+1}, t^\epsilon)$ for some $l\in \Natural_0$.
Moreover, 
if $H_Y$ is bounded so that $\{\nu_l\}$ is finite,
then (\ref{eq:segment1}) holds whenever
$\max(\lambda_0,\|H_Y\|)
< a<b< t^\epsilon$.
 \end{prop}

Note that, by assumption, $\lambda_0>1$ and $p<-1$.

For the proofs of both Propositions \ref{p:segmentintegral} and
\ref{p:hardmedium}, we focus on the integrals over negative values of 
$\lambda$, as these are notationally slightly easier to handle after a 
change of variable as in (\ref{eq:top1}).  The integrals
over positive values of $\lambda$ can be handled in almost the same way.
Alternatively, one can use that for $\lambda \in \Real$, $R_\chi(\lambda)= R_\chi(-\lambda)^*$
for real-valued $\chi$.
We postpone the proof of Proposition \ref{p:segmentintegral}
to Section \ref{ss:segmentintegral}, and instead turn to the consequences of the proposition.


\begin{prop}\label{p:hardmedium} 
Let $\lambda_0$ 
 be as in the statement of 
of Proposition \ref{p:Imbd}, and assume $m_k>(N_2+4-k)/2$ for $k=1,2$.
If the set $\{\nu_l\}$ is infinite, assume in addition that 
$m_k>(N_2+N_3+3-k)/2$ for $k=1,2$.
Then, if $0<\epsilon<1/N_1$ and  $t$ is sufficiently
large, there is a constant $C$ so that 
\begin{equation}\label{eq:hm1}
\left\| \int_{\lambda_0<|\lambda|<t^\epsilon}
 e^{it\lambda} G(\lambda) 
R_\chi(-\lambda) d\lambda \right\|_{\mchs\rightarrow \mchs} 
\leq Ct^{-1}. \end{equation}
\end{prop}
\begin{proof}  If $H_Y$ is bounded, so that the set $\{\nu_l\}$ is 
finite, then this proposition follows almost directly from Proposition 
\ref{p:segmentintegral}, using the choice of $\lambda_0^2>\|H_Y\|$.
  We write
$$\int_{\lambda_0<|\lambda|<t^\epsilon}e^{it\lambda} G(\lambda)R_\chi(-\lambda)d\lambda
= \lim_{\delta \downarrow 0}\int_{\lambda_0+\delta <|\lambda|<t^\epsilon-\delta}e^{it\lambda }G(\lambda)R_\chi(-\lambda)d\lambda,$$
  apply the estimate of Proposition \ref{p:segmentintegral}, and use the
fact that $\int_1^\infty s ^{p}ds$ converges since $p<-1$.

Now suppose the set $\{\nu_l\}$ is infinite.  
We give the proof for the integral over
$(-t^\epsilon,-\lambda_0)$, as the 
proof for the integral over $(\lambda_0,t^\epsilon)$ is essentially identical.
Choose $l_0\in \Natural$
so that $\nu_{l_0}\geq \lambda_0$ but $\nu_{l_0-1}\leq \lambda_0$, and 
let $L(t)\in \Natural$ be such that $\nu_{L(t)}\leq t^\epsilon$, but 
$\nu_{L(t)+1}\geq t^\epsilon$.
Then we write
\begin{multline}\label{eq:dividedintegral}
  \int^{-\lambda_0}_{-t^\epsilon}   e^{it\lambda} G(\lambda) 
R_\chi(-\lambda) d\lambda
= \lim_{\delta \downarrow 0} \left( 
\sum_{l=l_0}^{L(t)-1} \int_{-\nu_{l+1}+\delta}^{-\nu_l-\delta} e^{it\lambda} G(\lambda) 
R_\chi(-\lambda) d\lambda \right. \\ \left. + \int^{-\nu_{L(t)}-\delta}_{-t^\epsilon+\delta} 
 e^{it\lambda} G(\lambda) 
R_\chi(-\lambda) d\lambda + 
\int^{-\lambda_0-\delta}_{-\nu_{l_0}+\delta}  e^{it\lambda} G(\lambda) 
R_\chi(-\lambda) d\lambda\right).
\end{multline}
Using (\ref{eq:dividedintegral}) and
Proposition \ref{p:segmentintegral} we find
that
\begin{align}\label{eq:ubdsum}
& \left \| \int^{-\lambda_0}_{-t^\epsilon} e^{it\lambda} G(\lambda) 
R_\chi(-\lambda)  d\lambda\right \|
\leq Ct^{-1} \left( \sum_{l=l_0}^{L(t)-1}  
\left(\nu_l^{p}+\nu_{l+1}^{p}\right)
+ \int_{\lambda_0}^{t^\epsilon}s^{p}ds 
+  \lambda_0^p + t^{\epsilon p}\right).
\end{align}
Using the lower bound \eqref{e:weyl}: $\nu_l \geq l^{1/N_3}/C_2$, 
\begin{equation}
 \sum_{l=l_0}^{L(t)-1}  
\left(\nu_l^{p}+\nu_{l+1}^{p}\right) \leq 2 \sum_{l=l_0}^{L(t)}  
 (l^{1/N_3}/C_2)^{p}.
\end{equation}
This sum is bounded independently of $t$ because $p/N_3<-1$, or 
$m_k>(N_2+N_3+3-k)/2$ for $k=1,2$. The integral
in (\ref{eq:ubdsum}) is bounded independently of $t$ because $p<-1$.
\end{proof}

Proposition \ref{p:Imbd} follows directly from (\ref{eq:withG}),
Proposition \ref{p:hardmedium}, and Lemma \ref{l:easycontour}.

\subsection{Proofs of Proposition \ref{p:segmentintegral} and Lemma \ref{l:easycontour}}
\label{ss:segmentintegral}
It remains to prove Proposition \ref{p:segmentintegral} and Lemma \ref{l:easycontour}.  
The proofs are similar, involving contour deformations off the real axis to take advantage
of the exponential decay of $e^{\pm it\lambda}$ in the appropriate half-plane. 
That this and the resulting estimates are possible are due to the assumptions 
of a resonance-free region in which we have an estimate on
the cut-off resolvent, our assumptions (\ref{eq:resfree})
and (\ref{eq:resolveassumption}). 
As the proof of Proposition \ref{p:segmentintegral} is more complicated, we focus on it.
In particular, we 
prove (\ref{eq:segment1})  for the integral over
$[-b,-a]$ carefully, as the proof for the integral over $[a,b]$
is completely analogous.
 A first step in our proof of Proposition \ref{p:segmentintegral} is
 
\begin{lemma}\label{l:Ijs}
Let  $\epsilon <1/N_1$.    Suppose 
 $\max(\lambda_0, \nu_l)< a<b <
\min(\nu_{l+1}, t^\epsilon)$ for some $l \in \Natural_0$.  Alternatively, if $H_Y$ is
bounded, we allow the possibility that $\max(\lambda_0,\|H_Y\|)< a <b<t^\epsilon$. For each $t>1$, let  $\mathfrak{R}_t$ be the closed rectangle in $\Cs$ with vertices
$$a,\; b,\; b-i(\log t)/t,\; \text{ and} \; a-i (\log t)/t.$$  
Let $\gamma_{\downarrow}$, $\gamma_{\rightarrow}$, and $\gamma_{\uparrow}$ be the 
left, bottom, and right sides of  $\mathfrak{R}_t$, oriented counterclockwise.

Then there is a $T>1$, independent of $a,\;b$, and $l$, such that  for $t>T$
$$\int_{-b}^{-a} 
e^{i\lambda t} G(\lambda) R_\chi(-\lambda) d\lambda
= I_{\downarrow}+I_{\rightarrow}+I_{\uparrow},$$
where
\begin{equation}\label{eq:directionalintegrals}
I_{\bullet}  =  \int_{\gamma_{\bullet}} e^{-i\lambda t} G(-\lambda) R_\chi(\lambda) d\lambda , 
\end{equation}
with $\bullet$ denoting one of $\downarrow$, $\rightarrow$, or $\uparrow$.
\end{lemma}

\begin{figure}[h]
\labellist
\pinlabel $\gamma_{\downarrow}$ [l] at 17 19
\pinlabel $\gamma_{\uparrow}$ [l] at 98 19
\pinlabel $\gamma_{\rightarrow}$ [l] at 54 -4
\tiny
\pinlabel $a$ [l] at 24 38
\pinlabel $b$ [l] at 93 38
\pinlabel $-i\log t/t$ [l] at 9 2
\endlabellist
 \includegraphics[width=10cm]{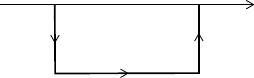}
 \caption{The contour used in Lemma \ref{l:Ijs}.}
\label{f:gamma}
\end{figure}
\begin{proof}  By a change of variable,
\begin{equation}\label{eq:top1}\int_{-b}^{-a} 
e^{i\lambda t} G(\lambda) R_\chi(-\lambda) d\lambda= \int_{a}^{b} 
e^{-i\lambda t} G(-\lambda) R_\chi(\lambda) d\lambda .
\end{equation}
Note that
$ G(-\lambda) $ is analytic in $\lambda$ in a neighborhood of ${\mathfrak R}_{t}$, for any $t>1$.   Moreover, there is a $T>1$, independent of $a,\;b$ and 
$l$ satisfying 
conditions of the lemma,
such  that $R_\chi(\lambda) $ is analytic in a neighborhood of 
${\mathfrak R}_{t}$ when $t>T$ . It is here that
we use $\epsilon<1/N_1$ and the fact that $R_\chi(\lambda)$ has an analytic continuation
to $\{\lambda \in \Cs\mid\ \Re \lambda>\lambda_0 \; \text{and}\; \Im \lambda >-C_0 (\Re \lambda)^{-N_1}\}.$
Hence by Cauchy's theorem
$$\int_{-\partial{\mathfrak R}_{t}} e^{-i\lambda t }G(-\lambda) R_\chi(\lambda)d\lambda=0.$$   
 The integral over the top side of the rectangle is the integral in (\ref{eq:top1}), hence the lemma follows directly.
\end{proof}

The next lemma bounds the integrals over the vertical sides of the rectangle.
\begin{lemma} \label{l:sides}
Set $p=N_2+\max(2-2m_1,1-2m_2)$. With the notation and assumptions as in Lemma
 \ref{l:Ijs},  there is a constant 
$C>0$ independent of $a, \; b$ and $l$  so that for $t>T$
$$\| I_{\downarrow}\|_{\mch\oplus \mch\rightarrow\mch\oplus \mch} \leq Ct^{-1}a^{p},$$
and
$$\|  I_{\uparrow}\|_{\mch\oplus \mch\rightarrow\mch\oplus \mch} 
\leq Ct^{-1}b^{p}.$$
Here $N_2$ is as in the statement of Theorem \ref{thm:wavedecay}.
\end{lemma}
\begin{proof}
The proofs of the two inequalities are essentially the same, so we prove 
only the first one.
We use that $|e^{-i\lambda t}|= e^{t\Im \lambda}$. 
Hence
\begin{align*}
\|  I_{\downarrow}\|_{\mchs\rightarrow \mchs} &=
 \left\| \int_{\gamma_{\downarrow}}
e^{-i\lambda t} G(-\lambda) R_\chi(\lambda) d\lambda \right\|_{\mchs\rightarrow \mchs} \\ 
& \leq \int_{-(\log t)/t}^0 e^{t s}
 \|G(-(a+is)) R_\chi(a+is )\|_{\mchs\rightarrow \mchs} ds.
\end{align*}
Recall from the definition of $G$ in (\ref{eq:G}) that for $\lambda \gg 1$,
$$G(-\lambda) \sim \left( \begin{array} {cc} -\lambda^{2-2m_1} & 
-i\lambda^{1-2m_2}\\
i\lambda^{2-2m_1} & -\lambda^{1-2m_2}\end{array}\right).$$
 Moreover, for $\lambda$ lying on the 
image of $\gamma_{\downarrow}$,  $|\lambda|$ is quite close to $a$, so
that
on $\gamma_{\downarrow}$,
 $\|G(-\lambda)\|\leq Ca^{\max(2-2m_1,1-2m_2)}$ and 
$\| R_\chi(\lambda)\|_{\mch\rightarrow \mch} \leq Ca^{N_2}$ by our assumptions on 
$R_\chi$ in the statement of the theorem.  The constants here are 
independent of  $a$ and $l$.
Thus
\begin{align*}\| I_{\downarrow}\|_{\mchs\rightarrow \mchs} &
 \leq C \int_{-(\log t)/t}^0 e^{t s} a^{N_2+\max(2-2m_1,1-2m_2)}ds 
\leq C a^{p}\;t^{-1}.
\end{align*}
\end{proof}

Next we bound the integral over the bottom side of the rectangle.
\begin{lemma}\label{l:bottom} With the notation and assumptions of Lemma \ref{l:Ijs},
there is a constant 
$C>0$ independent of $a$, $b$, and $l$  so that for $t>T$ 
$$\| I_{\rightarrow }\|_{\mchs\rightarrow \mchs} \leq C t^{-1}\int_{a}^b
 s^{N_2+ \max(2-2m_1,1-2m_2)}ds$$
with $N_2$ as in the statement of Theorem \ref{thm:wavedecay}.
\end{lemma}
\begin{proof}Arguing as in the proof of the previous lemma,
\begin{align*}
\| I_{\rightarrow}\|_{\mchs\rightarrow \mchs} & = \left\| \int_{a}^{b}
e^{-i(s-i(\log t)/t)t} G(-s+i(\log t)/t) R_\chi(s-i(\log t)/t) ds \right\|_{\mchs\rightarrow \mchs} \\
& \leq  C \int_{a}^{b }e^{-\log t} s^{N_2+\max(2-2m_1,1-2m_2)} ds
\end{align*}
which proves the lemma.
\end{proof}

\vspace{2mm}
\noindent { \em Proof of Proposition \ref{p:segmentintegral}.}\; 
The proof of the estimate on the the integral over $[-b,-a]$ in
(\ref{eq:segment1})  follows by combining the results of Lemmas \ref{l:Ijs},
\ref{l:sides}, and \ref{l:bottom}.  The proof of the
estimate for the integral over $[a,b]$  follows
in a completely analogous way, using the consequences of the self-adjointness
of $H$ for the (continued) resolvent $R(\lambda)$.
\qed
\vspace{2mm}

\noindent{\em Proof of Lemma \ref{l:easycontour}.}\; We can use Cauchy's theorem to write the integral
\begin{equation}
\int_{\lambda_0<\lambda<t^\epsilon}  e^{it\lambda}  G(\lambda)  \chi
R(\lambda)\chi d\lambda
\end{equation}
as the sum of integrals over the three line segments
$\lambda_0+i[0,3t^{-1}\log t]$,
$[\lambda_0,t^{\epsilon}]+i3t^{-1}\log t$, and 
$t^{\epsilon}+i[0,3t^{-1}\log t]$, where we reverse the orientation 
on the last interval.  We are deforming into the upper half plane, the physical
region, where
$R(\lambda)$ is a bounded operator on $\mch$ when $\lambda^2$ is not an 
eigenvalue of $H$--hence the assumption $\lambda_0>1$.  Note that 
$|e^{it \lambda}|= e^{-t\Im \lambda}$.  Now we can bound the integrals over
the vertical segments as in Lemma \ref{l:sides} and the integral over the 
top as in Lemma \ref{l:bottom}.  To bound
the integral over the sides, it suffices to have
$N_2+\max(2-2m_1,1-2m_2)\leq 0$. To bound the integral over the top,
where we can use $\|R(\lambda)\|_{\mch \rightarrow \mch}\leq 1/|\Im \lambda^2|$,
it would suffice to take $m_1=m_2=1$.

The bound for the portion of the integral over $-t^{\epsilon} <\lambda<-\lambda_0$ is proved in a similar way.
\qed

\section{A wave expansion under a hypothesis on the distinct eigenvalues
of $H_Y$}\label{s:exp2}

Under an assumption on the distinct eigenvalues of $H_Y$,
we can find an asymptotic expansion of $u(t)$ to order $t^{- k_0}$
for  any $k_0\in \Natural$.
This expansion involves an infinite sum, see (\ref{eq:u_{thr,k}}).  If multiplied
by the cut-off function $\chi$, the 
 sum over $l$ converges absolutely, see (\ref{eq:best}).    
The main result of this section is Theorem \ref{thm:waveexp2}.

In order to state the theorem, we introduce the 
notion of a distance on $\zhat$.  
For two points $\lambda, \; \lambda'\in \zhat$
we define 
$d_{\zhat}(\lambda, \lambda')=\sup_j|\tau_j(\lambda)-\tau_j(\lambda')|$. That
this is well-defined is shown in \cite[Lemma 5.1]{ch-da}.  In the 
statement of Theorem \ref{thm:waveexp2}
 below, by $\lambda' \in \Real$ we mean that
$\lambda'$ lies on the boundary of the physical space.
We also recall that $\nu_l^2$ denote the {\em distinct} positive 
eigenvalues of $H_Y$, with $0<\nu_1<\nu_2...$

\begin{thm} \label{thm:waveexp2} 
Let $H$ be a black box operator as in Section \ref{ss:bb},
 and suppose that for some $N_1,\; N_2\in [0,\infty)$, $\lambda_0>1$, and any 
$\tchi \in C_c^\infty([0,\infty)) $ with $\tchi(r)=1$ for $r\leq 1$ there are $C_0$, $C_1$  
so that $\tchi R(\lambda) \tchi$ is analytic on the set
\begin{equation}\label{eq:resfree2}
\{ \lambda \in \zhat: \; d_{\zhat}(\lambda,\lambda')<C_0(1+\lambda')^{-N_1}\; 
\text{for some $\lambda'\in \Real$, $\lambda'>\lambda_0-1$} \}
\end{equation}
 and that in this region
\begin{equation}\label{eq:resolveassumption2}
\| \tchi R(\lambda) \tchi\| \leq C_1(1+|\lambda|)^{N_2}.
\end{equation}
In addition, suppose that there are $\cY>0$, $\NY\geq 0$ so that
\begin{equation}\label{eq:spacing}
\nu_{l+1}-\nu_l>\cY \nu_l^{-\NY}\; \text{ when }\nu_l>1.
\end{equation}  Let $k_0\in \Natural$
be given, and $\chi \in C_c^\infty([0,\infty))$ be one for $r\leq 1$.
  Let $u(t)$ be the solution of  (\ref{eq:usolution}),
with $f_1, \; f_2\in (H+M_0)^{-m}\mch$ for any $m\in \Natural_0$ and 
$\bbbone_{\infty}f_1, \;\bbbone_{\infty}f_2 $ supported in $r\leq r_1<\infty$.
Then  there are  $b_{l,k,\pm}\in \langle r\rangle^{1/2+2k+\epsilon}\mch$,
 depending on $f_1$, $f_2$ so that
if we set 
\begin{equation}\label{eq:u_{thr,k}}
u_{thr,k_0}(t)=
\frac{1}{4} \sum_{\sigma_j=0} \Phi_j(0)\langle f_2,\Phi_j(0)\rangle_\mch +
 \sum_{k=0}^{k_0-1}t^{-1/2-k} 
\sum_{l>0} (e^{it \nu_l}b_{l,k,+}+ e^{-it \nu_l}b_{l,k,-})  
\end{equation}
then
there are  $m\in \Natural $, $C>0$  so that 
$$\|\chi(  u(t) - u_{e}(t)-u_{thr,k_0}(t))\|_{\mch}  \leq C  t^{-k_0}
\left( \|(H+M_0)^m f_1\|_\mch + \| (H+M_0)^m f_2\|_\mch \right) $$
if $t$ is sufficiently large.
Moreover, for $k=0,\;1, ...,k_0-1$
\begin{equation}\label{eq:bd}
\sum_{l>0} \left( \|\chi b_{l,k,+}\|_{\mch} +  \|
\chi b_{l,k,-}\|_{\mch} \right)
\leq  C  
\left( \|(H+M_0)^m f_1\|_\mch + \| (H+M_0)^m f_2\|_\mch \right).
\end{equation}
The value of 
$m$ needed depends polynomially on $k_0$, and also depends on $N_1$, $N_2$, and 
$\NY$. The $b_{l,k,\pm}$
 are determined by the value of $\nu_l$, the initial data $f_1$, $f_2$, and the 
derivatives with respect to $\tau_j$ 
 of order at most $2k$ of elements of the set 
$\{ \gef_{j'}\}_{0\leq \sigma_{j'}\leq \nu_l}$ evaluated at $\pm \nu_l$,
where $\sigma_j=\nu_l$.  Recall $u_e$ is given in (\ref{eq:ue}).
 \end{thm}
As in Remark \ref{rmk:seemweak}, the assumption that 
(\ref{eq:resolveassumption2}) holds in a region of the form (\ref{eq:resfree2})
follows from 
the  seemingly weaker assumption that the bound (\ref{eq:resolveassumption2}) 
on the cut-off resolvent holds for all sufficiently large $\lambda \in \Real$.
This follows from  \cite[Theorem 5.6]{ch-da}.  Thus, the primary 
difference in the hypotheses of Theorems \ref{thm:wavedecay} and Theorem \ref{thm:waveexp2} is the assumption that (\ref{eq:spacing}) holds.  If $H_Y$ is 
bounded, then (\ref{eq:spacing}) always holds.  

 The paper \cite{ch-da} includes examples of manifolds $X$ and large 
classes of potentials $V\in C_c^\infty(X;\Real)$ so that the 
hypotheses of this theorem hold for $H=-\Delta +V$ on $X$; see 
\cite[Theorems 3.1 and 5.6, and Section 3.2]{ch-da}.   These manifolds include the manifolds in Section \ref{ex:non1} and any of the 
manifolds in Section \ref{ex:3fun1} which
satisfy (\ref{eq:spacing}).  
 The question of spacing of 
distinct eigenvalues of the Laplacian on a compact manifold is complicated, even for 
manifolds of dimension $1$ which are not connected; see the
discussion after (\ref{eq:distinct}).
For general manifolds in higher
dimensions this is, as far as we know,
an open problem.  However, examples of compact manifolds satisfying the
condition (\ref{eq:spacing}) include spheres and flat tori.

The paper \cite{ch-da3} includes classes of open subsets of $\Real^d$
with cylindrical ends
for which the resolvent of the Dirichlet Laplacian satisfies the 
needed estimate,
and these include the examples of Section \ref{ss:waveguides}.
The  Dirichlet eigenvalues on the cross
section satisfy  (\ref{eq:spacing}) for all of the 
examples of Figure \ref{f:isometricends}
but only sometimes for the examples of Figure \ref{f:nonisometricends}.

The sums over
$l>0$ in (\ref{eq:u_{thr,k}}) and (\ref{eq:bd}) are sums over all 
values of $l\in \Natural$ so that $\nu_l^2\in \spec(H_Y)\setminus\{0\}$.  Hence the 
sums in $l$ are finite sums if $H_Y$ is bounded, and otherwise are sums
over $l\in \Natural$.  We use this convention throughout this section.

 Our Theorems \ref{thm:wavedecay} and \ref{thm:waveexp2} require a polynomial bound on the cut-off resolvent at high energies in order to handle 
 the large energy contribution to solutions of the wave equation.  If instead we consider only the solution localized in a  finite energy range, neither the bound 
 on the cut-off resolvent nor the assumption made
in  Theorem \ref{thm:waveexp2} on the distinct eigenvalues of $H_Y$ is necessary. We prove the following
 Proposition  naturally in the course of the proof of Theorem 
\ref{thm:waveexp2}.  
 \begin{prop}\label{p:speccutoffexp} Let $H$ be any operator satisfying all the conditions on the black box operator outlined in
Section \ref{ss:bb}.  Let
 $\psi_{sp}\in C_c^\infty(\Real;\Real)$, $r_1>0$, $f_1,f_2\in \mch$ satisfy 
$\bbbone_\infty f_1,\; \bbbone_{\infty}f_2$ vanish for $r>r_1$.
Let $\psiso\in C_c^\infty(\Real;\Real)$ satisfy 
$\psiso\psi_{sp}=\psi_{sp}$.
   Then if $k_0\in \Natural$, for each  $\nu_l$ with $\nu_l\in \supp \psiso$ there are 
 $b_{l,k,\pm}= b_{l,k,\pm}(f_1,f_2,\psi_{sp})\in r^{2k+1/2+\epsilon}\mch$, $k=0,...,k_0-1$, so that 
 \begin{multline}\label{eq:energycutoff}
 \chi \psi_{sp}(H)u(t) = \chi \psi_{sp}(H) u_e(t)
+ 
\frac{1}{4} \psi_{sp}(0) \sum_{\sigma_j=0}\chi \Phi_j(0)\langle f_2,\Phi_j(0)\rangle_\mch  \\+  \sum_{l>0}
 \psiso(\nu_l^2) \sum_{k=0}^{k_0-1}\chi (b_{l,k,+}e^{it\nu_l}+ b_{l,k,-}e^{-it\nu_l}) t^{-1/2-k} + \chi u_{r,k_0,\psi_{sp}}(t)
 \end{multline}
 with $\| \chi u_{r,k_0,\psi_{sp}}(t) \|_{\mch} \leq  C t^{-k_0}$ for sufficiently
large $t$.  Here $u_e(t)$ is as given in (\ref{eq:ue}).
 \end{prop}
 Note that the assumption that $\psiso$ has compact support means that the sum in (\ref{eq:energycutoff}) is finite.  Related results
for spectrally cut-off solutions of the wave equation (though with quite
different geometry) can be found in \cite{GHS13,VaWu13}.  Again, we note that we need neither the assumption of high-energy bounds on the (cut-off) resolvent
nor an assumption on the spacing of the 
distinct eigenvalues $\nu_l^2$ in the hypotheses of Proposition \ref{p:speccutoffexp}.




\subsection{Bounds on the derivatives of the cut-off resolvent}
Our proof of Theorem \ref{thm:waveexp2} will require bounds of the 
derivatives of the cut-off resolvent along the real axis.  It is here
that we will use our assumption on the spacing of the distinct eigenvalues
of $H_Y$.  Our first lemma, however, does not need these hypotheses as 
we bound the derivatives away from the thresholds.

\begin{lemma}\label{l:usecauchyestimates1}  Suppose the hypotheses of 
 Theorem \ref{thm:waveexp2} hold.    Let
$N \geq N_1$ be fixed and  let
 $\beta\geq \max(1,2/C_0)$.  
If $\lp \in \Real, \; |\lp|>\lambda_0$ and 
$\inf_{l,\pm }|\lp\pm \nu_l|>|\lp|^{-N}/\beta$, then there is a $C>0$ so
that 
\begin{equation}\label{eq:CE1}\left \| \frac{\partial ^k}{\partial \lambda^k}\chi R(\lambda) \chi\restrict_{\lambda=\lp}\right\|_{\mch\rightarrow \mch} \leq C k!
(1+|\lp|)^{N_2+kN}\beta^k, \; k\in \Natural.
\end{equation}
\end{lemma}
\begin{proof}
There is a ball in $\Cs $ centered at $\lp$ of radius 
$|\lp|^{-N}/\beta$ on which $\chi R(\lambda)\chi$ is analytic, with norm
bounded by $C|\lambda|^{N_2}$.  Hence the estimate (\ref{eq:CE1}) follows
immediately from the Cauchy estimates.  
\end{proof}


Away from the thresholds we 
can use $\lambda$ as a coordinate, as we did 
 in Lemma \ref{l:usecauchyestimates1}.  Near a threshold we need to introduce
a different local coordinate.  In particular, near the threshold $\sigma_j$
(and $-\sigma_j$) we shall use $\tau_j$ as a local coordinate.

For the setting of the next lemma, we think of $\{\nu_l\}$ as denoting 
not just the square roots of the distinct eigenvalues of $H_Y$, but 
also corresponding to a point on the boundary of 
the physical space in $\zhat$.  Given $\nu_l$, choose
$\epsilon= \epsilon(l)>0$ so that $|\nu^2_{l}-\nu^2_{l\pm 1}|>\epsilon^2$ and let
$j=j(l)\in \Natural$ be such that $\nu_l=\sigma_j$.  Then 
we may, in a natural way, identify $B(0;\epsilon)=\{z\in \Complex: |z|<\epsilon\}$ with a (particular) neighborhood  $U_{\nu_l}(\epsilon)$ of $\nu_l\in \zhat$ by using 
\begin{equation} U_{\nu_l}(\epsilon)\ni \lambda \mapsto \tau_j(\lambda)\in B(0;\epsilon);
\end{equation}
$U_{\nu_l}(\epsilon)$ 
is defined to be 
the connected component of  $\tau_j^{-1}(B(0;\epsilon))$ containing
$\nu_l$, a point
on the boundary of the physical space. 
If $\epsilon$ is small enough, as we have 
chosen here, then $U_\epsilon$ is a double cover of a neighborhood of $\nu_l$
in $\Cs$.   A completely analogous identification can be done near $-\nu_l$,
also using $\tau_j$, where $j=j(l)$.

The assumption on the spacing of the distinct eigenvalues of $H_Y$
allows us to bound the
derivatives of $\chi R\chi$ in a neighborhood of each
threshold.
\begin{lemma}\label{l:usecauchyestimates2} Suppose the hypotheses of 
 Theorem \ref{thm:waveexp2} hold, and continue to use the 
notation $j=j(l)$, $U_{\nu_l}(\epsilon)$ introduced above.
Set $N_M=\max((N_Y-1)/2,N_1)$.
There are  $\alpha>0$, $C\in \Real$
so that if  $\nu_l=\sigma_j\geq \lambda_0+1$, 
 then
\begin{equation}\label{eq:CE2}
\left \| 
\left( \frac{\partial^k}{\partial \tau_j^k}
\chi R(\lambda) \chi \right) \restrict_{\lambda=\lp}\right\|_{\mch\rightarrow \mch} \leq C k! |\nu_l|^{N_2+kN_M}\alpha^{-k}, \; k\in \Natural
\end{equation}
for all $\lambda' \in U_{\pm \nu_l}(\alpha \nu_l^{-N_M})\subset \zhat.$
\end{lemma}
\begin{proof}
For simplicity, we give the proof only for $U_{\nu_l}(\alpha\nu_l^{-N_M})$.

The
assumptions on the spacing of the distinct eigenvalues of $H_Y$ ensure
that there is a $\beta >0$ so that 
$|\nu_l^2- \nu^2_{l\pm 1}|>  \nu_l^{1-N_Y}/\beta $
for all $l>1$.
Moreover, increasing $\beta>0 $ if necessary, 
our definition of $N_M$ and  the hypotheses of Theorem \ref{thm:waveexp2}
ensure   that using the coordinate $\tau_{j(l)}$, $\chi R(\lambda) \chi$ is analytic on 
$U_{\nu_l}(1/(\beta \nu_l^{N_M }))$,
again for all $l$ with $\nu_l> \lambda_0+1$.  Here we use the hypothesis (\ref{eq:resfree2}).
Moreover,
 $\| \chi R(\lambda ) \chi\|\leq C(1+\nu_l)^{N_2}$ in this set,
with constant $C$ independent of $l$.
Identify $U_{\nu_l}(1/(\beta \nu_l^{N_M }))$ with $B(0; 1/(\beta \nu_l^{N_M }))$. 
Each point $z$ in $B(0; 1/(2\beta \nu_l^{N_M }))$ has the property that the ball
with center $z$ and radius $1/(2\beta \nu_l^{N_M })$ lies in 
$B(0; 1/(\beta \nu_l^{N_M }))$.
Hence, we may prove the lemma by taking $\alpha=1/(2 \beta )$
and  by applying the Cauchy estimates on such a ball, recalling that the coordinate
is $\tau_j$.
\end{proof}

\subsection{The proof of Theorem \ref{thm:waveexp2}}
We turn more directly to the proof of Theorem  \ref{thm:waveexp2}.  As 
in the proof of Theorem \ref{thm:wavedecay}, we shall write $(I-\Pre)u(t) $
as the sum of several integrals.  In order to define these, 
let $\psi\in C_c^\infty(\Real; [0,1])$ 
have its support in a small neighborhood of the origin, and be one in a smaller 
neighborhood of the origin.  For convenience later, choose $\psi$ to be
even.
Set 
$$\tN=\max(\NY-1, 2N_1)$$
and 
\begin{equation}\label{eq:psil}
\psi_l(\lambda)= \psi\left(\nu_l^{\tN} (\lambda^2-\nu_l^2)\right).
\end{equation}
To prove  Proposition  \ref{p:speccutoffexp} (rather
than Theorem \ref{thm:waveexp2}), the 
choice of  of $\tN$ does not matter much--we can take $\tN=0$.
  We choose the support of $\psi$ to be small enough that
$$\supp \psi \cap \supp \psi_l=\emptyset,\; \text{for}\; \nu_l^2\in \spec(H_Y)\setminus\{0\}.\;$$
To prove Theorem \ref{thm:waveexp2}, by shrinking the support of $\psi$ 
if necessary, we choose $\psi$ to satisfy
\begin{equation}\label{eq:disjsupp}
\supp \psi_l \cap \supp \psi_{l'}=\emptyset,
\;
\text{if $l\not = l'$}. 
\end{equation}
The assumption on the spacing of the distinct eigenvalues of $H_Y$ 
and our choice of $\tN\geq \NY-1$ ensure that (\ref{eq:disjsupp}) is possible.
To prove Proposition \ref{p:speccutoffexp} instead we replace (\ref{eq:disjsupp})
by
\begin{equation}\label{eq:p4.2version}
\psi_{sp}(\lambda^2)\psi_l(\lambda)\psi_{l'}(\lambda)\equiv 0\; \text{if $l\not = l'$}. \; 
\end{equation}

Similarly to (\ref{eq:Ismlsum}), using the integral representation of
Lemma \ref{l:contspec} we can write
\begin{equation}\label{eq:seconddecomp}
(I-\Pre)\left( \begin{array}{c}u(t)\\ u_t(t) \end{array}\right)
 = \left( I_{0}(t)+I_{thr}(t)+I_{r}(t)\right) 
\left( \begin{array}{c}f_1\\ f_2 \end{array}\right)
\end{equation}
where
\begin{align*}
I_{0}(t) & =  \PV \frac{1}{2\pi i} \int 
e^{it\lambda}\psi(\lambda) A(\lambda) (R(\lambda)- R(-\lambda)) (I-\Pre) d\lambda\\
I_{thr}(t) & = \frac{1}{2\pi i}  \int e^{it\lambda}\sum_{l>0}
\psi_l ( \lambda )  A(\lambda) (R(\lambda)- R(-\lambda)) (I-\Pre) d\lambda
\\
I_{r}(t) & = \frac{1}{2\pi i}\int e^{it\lambda} 
\left( 1 -\psi(\lambda)-\sum_{l>0} \psi_l (\lambda )
\right)
  A(\lambda) (R(\lambda)- R(-\lambda)) (I-\Pre) d\lambda   .
\end{align*}
Here $A(\lambda)$ is as in (\ref{eq:A}).

We recall that we have already studied $I_0(t)$ in Lemma \ref{l:I0}.
 Lemma \ref{l:Ir} shows that $I_r(t)$ does not contribute to the
asymptotic expansion of $(I-\Pre)u(t)$.  In Lemma \ref{l:lcont} we evaluate
the contribution from any nonzero threshold.  Finally we combine these
 to prove the theorem.

\begin{lemma}\label{l:Ir}
Under the hypotheses of Theorem 
\ref{thm:waveexp2}, if the support of $\psi$ is chosen sufficiently small, then for any $k\in \Natural$, there is an 
$m\in \Natural$ depending polynomially on 
$k$ so that for $t>0$
$$\left\| \chi I_r(t) \left( \begin{array} {c} f_1\\f_2\end{array}\right) 
\right\|_{\mch\oplus \mch} \leq C t^{-k}( \| (H+M_0)^mf_1\|_\mch+ 
\| (H+M_0)^mf_2\|_\mch).$$
\end{lemma}
\begin{proof} 
Set 
\begin{equation}\label{eq:psitot}
\psi_{tot}(\lambda)=  \psi(\lambda)+\sum_{l>0} \psi_l(\lambda) .
\end{equation}
Note that by our assumptions on $\psi$, $1-\psi_{tot}$  vanishes in a neighborhood of $\lambda=0$ 
and in a neighborhood of $\lambda=\pm \sigma_j$ for each $\sigma_j$.  Hence
$$( 1 -\psi_{tot}(\lambda))
 (\chi R(\lambda)(I-\Pre)\chi-\chi R(-\lambda)(I-\Pre)\chi)$$
is a smooth function of $\lambda$. 
 Using Lemma \ref{l:usecauchyestimates1}, by choosing $m\in \Natural$ 
sufficiently large we can ensure that
\begin{equation}\label{eq:k'deriv}
\left\| \frac{d^{k'}}{d\lambda^{k'}}\left( (\lambda^2+M_0)^{-m} ( 1 -\psi_{tot}(\lambda))
 \chi( R(\lambda)- R(-\lambda))(I-\Pre) \chi\right)\right\|_{\mch\rightarrow \mch}
\leq C (1+|\lambda|)^{-2}\end{equation}
for all $k'\in \Natural_0$,  $k'\leq k$.
The choice of exponent $-2$ on the right-hand side is somewhat arbitrary,
but is made to ensure that the function is integrable.  We could replace 
$-2$ by $-p$, some other $p>1$, and such a change may change the value of $m$ 
which is needed on the left hand side.
Now we use (\ref{eq:Qmidentity}) and 
integrate  by parts $k$ times to prove the lemma.
\end{proof}

By way of comparison, we include  following lemma.
\begin{lemma} \label{l:Irspcut}
Under the hypotheses of Proposition \ref{p:speccutoffexp}, 
if the support of $\psi$ is chosen sufficiently small, then for any $k\in \Natural$, there is a $C>0$ so that
$$\left \| \chi \psi_{sp}(H) I_r(t) \left( \begin{array}{c}
f_1\\ f_2 \end{array}\right) \right \|_{\mch \oplus \mch} \leq C t^{-k}( \| f_1\|_{\mch}
+\|f_2\|_{\mch})\; \text{for $t>0$}.$$
\end{lemma}
\begin{proof}
We use $\psi_{tot} $ from (\ref{eq:psitot}).  Using the compact
support of $\psi_{sp}$ there is a $C>0$ so that
$$
\left\| \frac{d^{k'}}{d\lambda^{k'}}\left( \psi_{sp}(\lambda^2) ( 1 -\psi_{tot}(\lambda))
 \chi( R(\lambda)- R(-\lambda))(I-\Pre) \chi\right)\right\|_{\mch \mapsto\mch} 
\leq C (1+|\lambda|)^{-2}
$$
for all $k'\in \Natural_0$,  $k'\leq k$.
Then integrating by parts $k$ times proves the lemma.
\end{proof}

\begin{lemma}\label{l:lcont}  Let $H$, $f_1$, and $f_2$ satisfy 
the hypotheses of Theorem \ref{thm:waveexp2}.
Let $\psi_l$ be as defined in (\ref{eq:psil}) 
for $\nu_l^2\in \spec(H_Y)\setminus\{0\}$ and let $k_0\in\Natural$. Then with the support of
the function $\psi$ in (\ref{eq:psil}) chosen sufficiently small,  there are 
$b_{l,k,\pm},\;b^{(')}_{l,k,\pm}\in \langle r \rangle^{1/2+2k+\epsilon}\mch$, $k=0,1,...,k_0-1$  so that
\begin{multline} \label{eq:psilexp} \int_{-\infty}^\infty e^{i t \lambda}
\psi_l(\lambda) A(\lambda)
\chi(R(\lambda)-R(-\lambda))(I-\Pre)\left( \begin{array}{c} f_1\\ f_2
\end{array}\right) d\lambda \\
=\sum_{k=0}^{k_0-1}t^{-k-1/2} \left( \begin{array}{c} \chi b_{l,k,+}e^{it\nu_l}+
 \chi b_{l,k,-}e^{-it\nu_l} \\\chi b_{l,k,+}^{(')}e^{it\nu_l}+
 \chi b_{l,k,-}^{(')}e^{-it\nu_l}
\end{array}\right) + \chi B_{l,k_0}(t).
\end{multline}
There is an $m\in \Natural $ depending polynomially
 on $k_0$ as well as on $N_1,N_2,\NY$ and 
a constant $C$ independent of $l$ so that for $t>0$
\begin{equation}\label{eq:remeststph}\| \chi B_{l,k_0}(t)\|_{\mch \oplus \mch}\leq C l^{-2}t^{-k_0} (\|(H+M_0)^m f_1\|_\mch+\|(H+M_0)^m f_2\|_\mch)
\end{equation}
and
\begin{equation}\label{eq:best}
\| \chi b_{l,k,\pm} \|_\mch + \|\chi  b_{l,k,\pm}^{(')} \|_\mch\leq C  l^{-2}(\|(H+M_0)^m f_1\|_\mch+\|(H+M_0)^m f_2\|_\mch)
,\; k=0,...,k_0-1.
\end{equation}
The $b_{l,k,\pm}$, $b^{(')}_{l,k,\pm}$
 are determined by the initial data $f_1$, $f_2$, the value of $\nu_l$,  and the 
derivatives with respect to $\tau_j$ 
 of order at most $2k$ of elements of the set 
$\{ \gef_{j'}\}_{0\leq \sigma_{j'}\leq \nu_l}$ evaluated at $\pm \nu_l$,
where $\sigma_j=\nu_l$.
\end{lemma}
Before proving the lemma, we note that as in (\ref{eq:k'deriv}) the choice of exponent
$-2$ for $l$ in (\ref{eq:remeststph}) and (\ref{eq:best}) is 
 somewhat
arbitrary.  We choose it to ensure the sum over $l$  converges.
As in  (\ref{eq:k'deriv}),  $l^{-2}$ may be replaced by $l^{-p}$, $p>1$,
if $m$ is chosen sufficiently large, depending on $p$.
\begin{proof}
Let $j\in \Natural$ be such that $\nu_l=\sigma_j$.  We apply Lemma \ref{l:allinonestatphase}.  

By our choice of $\psi_l$, the support of $\psi_l$ contains no thresholds other than $\pm \nu_l$.
Then with $\sigma_j=\nu_l$,
 $\tau_j (\lambda) \chi R(\lambda)(I-\Pre) \chi$ is a 
smooth 
function of $\tau_j(\lambda)$ on the
support of $\psi_l(\lambda)$, $\lambda \in \Real$, but  
$\chi \tau_j(\lambda) R(-\lambda)(I-\Pre)\chi$ is not in general. 
Hence we shall rewrite the integral to avoid the use of $R(-\lambda)$.
 At the 
same time, for $m\in \Natural_0$ we write $f_1=(H+M_0)^{-m}(H+M_0)^mf_1$,
and similarly for $f_2$.   Thus we 
rewrite
\begin{multline}\label{eq:rewrite1}
 \int_{-\infty}^\infty e^{i t \lambda}
\psi_l(\lambda) A(\lambda)
\chi(R(\lambda)-R(-\lambda))(I-\Pre)(\lambda^2+M_0)^{-m}(H+M_0)^m\left( \begin{array}{c} f_1\\ f_2
\end{array}\right) d\lambda\\
= \int_{-\infty}^\infty (e^{i t \lambda} A(\lambda)-e^{-it\lambda} A(-\lambda))\psi_l(\lambda)
\chi R(\lambda)(I-\Pre)(\lambda^2+M_0)^{-m}(H+M_0)^m\left( \begin{array}{c} f_1\\ f_2
\end{array}\right) d\lambda.
\end{multline}
This integral has an asymptotic expansion, with 
contributions arising from the neighborhoods of $\lambda=\sigma_j=\nu_l$
and $\lambda = -\sigma_j=-\nu_l$; each will contribute both 
to the $b's$ and to $B_{l,k_0}$.

Now we recall from the proof of Lemma \ref{l:usecauchyestimates2} that
there are neighborhoods of 
 $ \nu_l=\sigma_j$ and $- \nu_l=-\sigma_j$ in $\zhat$
on which we may use $\tau_j$ as 
a coordinate and on which $\chi \tau_j(\lambda) R(\lambda)(I-\Pre)\chi$ is a smooth function of $\tau_j$. Under the hypotheses of Theorem \ref{thm:waveexp2}, the
radii of the balls about $\pm \nu_l$ can be taken proportional to 
$\min ( \sigma_j^{(1-\NY)/2}, \sigma_j^{-N_1})=\min(\nu_l^{(1-\NY)/2},\nu_l^{-N_1})$ in the $\tau_j$ coordinate.
Hence by our choice
of $\tN$ we can choose 
our original function $\psi$, with 
$\psi_l(\lambda)=\psi(\nu_l^{\tN}(\lambda^2-\nu_l^2))$, so
that we can extend 
$\psi_l(\lambda) A(\pm \lambda)\chi\tau_j(\lambda) R(\lambda)(I-\Pre)\chi$ to be a smooth, compactly supported function 
of $\tau_j$ in this complex ball.  In fact, with 
$\psi$ chosen to be even, $\psi( \nu_l^{\tN}|\lambda^2-\nu_l^2|)A(\pm \lambda)\chi\tau_j(\lambda) R(\lambda)(I-\Pre)\chi$ 
provides such an extension (where we understand
$\lambda$ to be the locally well-defined function of $\tau_j$).  With 
the support of $\psi$ chosen sufficiently small, this will hold for all 
$l\in \Natural$.

 Thus we may apply 
Lemma \ref{l:allinonestatphase} in order to find an expansion
for the portion of the integral in (\ref{eq:rewrite1})  over $(0,\infty)$.  From the second 
part of Lemma \ref{l:allinonestatphase} we obtain immediately
that 
\begin{multline}\label{eq:oneonly}
\int_{0}^\infty (e^{i t \lambda} A(\lambda) -e^{-it\lambda} A(-\lambda))\psi_l(\lambda)
\chi R(\lambda)(I-\Pre)(\lambda^2+M_0)^{-m}(H+M_0)^m\left( \begin{array}{c} f_1\\ f_2
\end{array}\right) d\lambda
\\
= -\int_{0}^\infty e^{-it\lambda} A(-\lambda)
\psi_l(\lambda)
\chi R(\lambda)(I-\Pre)(\lambda^2+M_0)^{-m}(H+M_0)^m\left( \begin{array}{c} f_1\\ f_2
\end{array}\right) d\lambda +O(t^{-k_0}).
\end{multline}
Now by applying the first part of 
 Lemma 
\ref{l:allinonestatphase}  and (\ref{eq:statphaseexpl}), we 
have an expansion of (\ref{eq:oneonly}) of the form in (\ref{eq:psilexp})
with 
exponential $e^{-it\nu_l}$, and the coefficients in the expansion; that
is, the  $b_{l,k.-}$ and $b_{l,k,-}^{(')}$; 
are determined by  $\sigma_j=\nu_l$ and derivatives with respect to $\tau_j$ of 
\begin{equation}
\label{eq:detderivs}
\psi_l(\lambda) A(- \lambda) \tau_j(\lambda)
\chi R(\lambda)(I-\Pre)(\lambda^2+M_0)^{-m}\left( \begin{array}{c} (H+M_0)^m f_1\\(H+M_0)^m f_2
\end{array}\right)
\end{equation}
of order at most $2k$ evaluated at $\lambda =\sigma_j=\nu_l$. 


Next we turn to the question of uniformity in $l$, as in (\ref{eq:remeststph}) and (\ref{eq:best}).  This 
is immediate if $H_Y$ is bounded, so we consider only the case of unbounded 
$H_Y$.  Note that in this case our assumption (\ref{eq:spacing}) ensures that 
$\nu_l\geq (C_Yl)^{1/(\NY+1)} +O(1)$.  
 By   Lemma \ref{l:usecauchyestimates2} 
the derivatives 
of fixed order of $\chi R (\lambda)(I-\Pre) \chi$ with respect to $\tau_j$  near $\pm \sigma_j$
grow at worst polynomially in $\sigma_j=\nu_l$, as do the derivatives of 
$\psi_l(\lambda )$ by definition.  Thus 
by taking $m$ sufficiently large, 
depending polynomially on $k_0$, we can guarantee that the analog
of (\ref{eq:best}) holds.  Similarly, using the remainder estimate of
Lemma \ref{l:allinonestatphase}, we find that the analog of
(\ref{eq:remeststph}) holds.

 The integral in (\ref{eq:rewrite1}) over $(-\infty,0)$ 
can be handled in an analogous manner, and gives us the $b_{l,k,+}$ and 
$b^{(')}_{l,k,+}$.

Thus far we have proved that the coefficients $b_{l,k,-}$ (and $b_{l,k,-}^{(')}$)
are determined by 
the value of 
$\nu_l=\sigma_j$ and the derivatives with respect to $\tau_j$ of 
(\ref{eq:detderivs}) evaluated at $\lambda=\nu_l=\sigma_j$.  We
can say a bit more.  Here we concentrate on describing the origin of the 
$b_{l,k,-}$ and do not worry about bounding them uniformly in $l$, as we
have already done so.  Rather than using (\ref{eq:oneonly})
we return to the original expression
\begin{align*}& \int_{-\infty}^\infty e^{i t \lambda}
\psi_l(\lambda) A(\lambda)
\chi(R(\lambda)-R(-\lambda))(I-\Pre)\left( \begin{array}{c} f_1\\ f_2
\end{array}\right) d\lambda \\ & =
\int_{-\infty}^0 e^{i t \lambda}
\psi_l(\lambda) A(\lambda)
\chi(R(\lambda)-R(-\lambda))(I-\Pre)(\lambda^2+M_0)^{-m}(H+M_0)^m\left( \begin{array}{c} f_1\\ f_2
\end{array}\right) d\lambda \\& \hspace{2mm}+
\int_{0}^\infty e^{i t \lambda}
\psi_l(\lambda) A(\lambda)
\chi(R(\lambda)-R(-\lambda))(I-\Pre)(\lambda^2+M_0)^{-m}(H+M_0)^m\left( \begin{array}{c} f_1\\ f_2
\end{array}\right) d\lambda.
\end{align*}
We concentrate on the integral over 
$(-\infty,0)$, which gives us the $b_{l,k,-}$ and $b_{l,k,-}^{(')}$ (the
integral over $(0,\infty)$ leads to the $b_{l,k,+}$).
Using Lemma \ref{l:specmeas},
\begin{align} & \nonumber
 \int^{0}_{-\infty} e^{it\lambda} A(\lambda)
\psi_l(\lambda)
\chi [ R(\lambda)-R(-\lambda)](I-\Pre)(\lambda^2+M_0)^{-m}(H+M_0)^m\left( \begin{array}{c} f_1\\ f_2
\end{array}\right) d\lambda \\ & \nonumber
= \frac{-i}{2}\int_{0}^{\sigma_j} e^{-it\lambda} A(-\lambda)
\psi_l(\lambda)
\chi \sum_{0\leq \sigma_{j'}< \sigma_j}\frac{\gef_{j'}(\lambda)\otimes \gef_{j'}(\lambda)}{\tau_{j'}(\lambda)}(\lambda^2+M_0)^{-m}(H+M_0)^m\left( \begin{array}{c} f_1\\ f_2
\end{array}\right) d\lambda \\ &   \hspace{3mm}
-\frac{i}{2}\int_{\sigma_j} ^\infty e^{-it\lambda} A(-\lambda)
\psi_l(\lambda)
\chi \sum_{0\leq \sigma_{j'}\leq \sigma_j}
\frac{\gef_{j'}(\lambda)\otimes \gef_{j'}(\lambda)}{\tau_{j'}(\lambda)}(\lambda^2+M_0)^{-m}(H+M_0)^m\left( \begin{array}{c} f_1\\ f_2
\end{array}\right) d\lambda .
\end{align}
We can apply Lemma \ref{l:bdrstatphase} to each of these two integrals. 
We know from our previous discussion that the sum will have an asymptotic
expansion in powers of $t^{-k-1/2}$, but we learn some more specific
information about the coefficients this way.  This
shows us that the $b_{l,k,-}$ and $b^{(')}_{l,k,-}$ are actually determined by 
the value $\nu_l=\sigma_j$, the initial conditions
$f_1$ and $f_2$, and elements of the set
 $$\left\{
(\partial_{\tau_j}^{k'} \gef_{j'} )\restrict_{\lambda=\sigma_j}\mid 
0\leq \sigma_{j'}\leq \sigma_j,\; 0\leq k'\leq 2k\right\}.$$
This also provides a natural way to see that
$b_{l,k-}\in \langle r \rangle^{1/2+2k+\epsilon}\mch$, since
$\partial_{\tau_j}^{k'}\gef_{j'}{\restrict_{ \lambda=\sigma_j}}
\in \langle r \rangle^{1/2+k'+\epsilon}\mch$ if $\sigma_{j'}\leq \sigma_j$.
\end{proof}

The next lemma is almost parallel to Lemma \ref{l:lcont}.  It differs in that it 
assumes only the hypotheses of Proposition \ref{p:speccutoffexp}, and only achieves uniformity in 
$l$ because of the multiplication by the 
compactly supported function  $\psi_{sp}(\lambda^2)$.  We omit the proof because 
it is so similar.  
\begin{lemma}\label{l:lcontprop} Let $H$, $f_1$, and $f_2$ satisfy 
the hypotheses of Proposition \ref{p:speccutoffexp}.  
Let $\psi_l$ be as defined in (\ref{eq:psil})
 and let $k_0\in\Natural$. Then with the support of
the function $\psi$ in (\ref{eq:psil}) chosen sufficiently small, for 
 $l\in \Natural$ there are 
$b_{l,k,\pm},\;b^{(')}_{l,k,\pm}\in \langle r \rangle^{1/2+2k+\epsilon}\mch$, $k=0,1,...,k_0-1$  so that
\begin{multline}  \int_{-\infty}^\infty e^{i t \lambda}
\psi_l(\lambda)\psi_{sp}(\lambda^2) A(\lambda)
\chi(R(\lambda)-R(-\lambda))(I-\Pre)\left( \begin{array}{c} f_1\\ f_2
\end{array}\right) d\lambda \\
=\sum_{k=0}^{k_0-1}t^{-k-1/2} \left( \begin{array}{c} \chi b_{l,k,+}e^{it\nu_l}+
 \chi b_{l,k,-}e^{-it\nu_l} \\\chi b_{l,k,+}^{(')}e^{it\nu_l}+
 \chi b_{l,k,-}^{(')}e^{-it\nu_l}
\end{array}\right) + \chi B_{l,k_0}(t)
\end{multline}
with $$\|\chi B_{l,k_0}(t) \|_{\mch} \leq C_{l,k}t^{-k_0} (\|f_1\|_{\mch}+\|f_2\|_{\mch}).$$
The $b_{l,k,\pm}$, $b^{(')}_{l,k,\pm}$
 are determined by the initial data $f_1$, $f_2$, the value of $\nu_l$,  and the 
derivatives with respect to $\tau_j$ 
 of order at most $2k$ of elements of the set 
$\{ \gef_{j'}(\lambda),\; \psi_{sp}(\lambda^2)\gef_{j'}(\lambda)\}_{0\leq \sigma_{j'}\leq \nu_l}$ evaluated at $\pm \nu_l$,
where $\sigma_j=\nu_l$.
\end{lemma}
Of course, the $b_{l,k,\pm}$, $b^{(')}_{l,k,\pm}$ in Lemma 
\ref{l:lcontprop} are $0$ if $\pm \nu_l$ are 
 not in the support of $\psi_{sp}(\lambda^2)$.

\vspace{2mm}
\noindent
{\em Proof of Theorem \ref{thm:waveexp2}.}
We write $u(t)=\Pre u(t)+ (I-\Pre)u(t)$.  The expansion of $u_e(t)=\Pre u(t)$ is
given in Lemma \ref{l:ueigenvalues}.
From equation (\ref{eq:seconddecomp}) we see that $(I-\Pre)u(t)$ is
given by a sum of contributions from $I_0$, $I_{thr}$, and $I_r$.
By Lemma \ref{l:Ir}, the contribution of $I_r$ is of order $t^{-k}$ for any
$k$.  
Note that using Lemma \ref{l:lcont} and summing over $l\in \Natural$ 
evaluates the contribution of $I_{thr}$;  the estimates
(\ref{eq:remeststph}) and (\ref{eq:best}) ensure the convergence
of the sums over $l$ to bound  the remainder 
and the absolute convergence of the sum over $l$
 in (\ref{eq:u_{thr,k}}), respectively.  Lemma \ref{l:I0} gives the 
contribution of $I_0$.  Summing the contributions of the terms 
from $I_0$ and $I_{thr}$ gives $u_{thr,k_0}$. 
\qed

\vspace{2mm}
\noindent
{\em Proof of Proposition \ref{p:speccutoffexp}.}
We use (\ref{eq:seconddecomp}), multiplying on the left hand side by
$\chi \psi_{sp}(H)=\chi \psi_{sp}(H)\psiso(H)$. By Lemma \ref{l:Irspcut}, 
$\| \chi \psi_{sp}(H) I_r(t) \chi\|_{\mch\oplus \mch \rightarrow 
\mch\oplus \mch} \leq Ct^{-k}$ for any $k$.
 Then by Lemma \ref{l:I0} and
 Lemma \ref{l:lcontprop}, summing over the finite number of $l$ with 
$\nu_l^2\in \supp \psi_{sp}$, we see that the sum of the contributions 
of $\chi \psi_{sp}(H)I_0(t)$ and $\chi \psi_{sp}(H)I_{thr}(t)$ gives an 
expansion as claimed.
\qed

\appendix
\section{Asymptotic expansions of some integrals}
In  this section we prove two lemmas  which 
are used in evaluating the contribution of the thresholds to the 
asymptotics of the solutions of the wave equation.  The 
proofs of these lemmas use a change of variables and stationary
phase to find asymptotic expansions of two types of integrals.

\begin{lemma}\label{l:allinonestatphase}  Fix 
$c_0\in (0,1)$ and set $B_0= \{ z\in \Complex\mid |z|<c_0/2\}$.
Let $\mcx$ be a Banach
space.  For each  $k_0\in \Natural$
there is a $C>0$ so that if $\sigma_j>c_0$ and $F\in C_c^\infty(B_0;\mcx)$
then there are $b_{k}=b_{k}(F,\sigma_j)\in \mcx$ so that for $t>0$
\begin{multline}\label{eq:statphaseexpl}\left\| \int _0^\infty e^{- i\lambda t}\frac{ F(\tau_j(\lambda))}{\tau_j(\lambda)}
d\lambda - (\sigma_j t)^{-1/2} e^{- i\sigma_j t} \sum_{k=0}^{k_0-1} b_{k }(t\sigma_j)^{-k}\right\|_{\mcx}
\\
\leq  C \left[(\sigma_j t)^{-k_0} \sum_{k  \leq 2k_0+1}  \sup_\tau \left\| \sigma_j^k  F^{(k)(\tau) } \right\|_{\mcx}+ 
t^{-k_0} (1+\sigma_j)^{k_0} \sum_{k\leq 2k_0+1}\sup _\tau \|  F^{(k)}(\tau)\|_{\mcx} \right].
\end{multline}
Here  $b_0 =e^{- i\pi/4}F(0)\sqrt{2\pi}$, and  the coefficients
$b_{k}$ are determined by the derivatives with respect to $\tau$ 
of $F(\sigma_j \tau) /\sqrt{\tau^2+1}$ of order at most $2k$,
evaluated at $\tau=0$. 
Moreover, under the same assumptions, for $t>0$
\begin{multline}
\label{eq:statphaseexp2}\left\| \int _0^\infty e^{ i\lambda t}\frac{ F(\tau_j(\lambda))}{\tau_j(\lambda)}
d\lambda \right\|_\mcx
\\
\leq  C \left[(\sigma_j t)^{-k_0} \sum_{k  \leq 2k_0+1}  \sup_{\tau}\left\| \sigma_j^k  F^{(k)} (\tau)\right\|_{\mcx}+ t^{-k_0} (1+\sigma_j)^{k_0} \sum_{k\leq 2k_0+1}\sup _\tau
\left\|   F^{(k)}(\tau)\right\|_{\mcx} \right].
\end{multline}
\end{lemma}
\begin{proof}  
In order to simplify notation, we give the proof for $\mcx=\Complex$, with
the notation $|\alpha |=\| \alpha\|_{\Complex}$.  The proof for a general 
Banach space $\mcx$ is essentially identical, though 
notationally more complicated.

We may write
$F(\tau_j(\lambda))= F_e(\tau_j(\lambda)) + F_o(\tau_j(\lambda))$, where
\begin{align*}
F_e(\tau_j(\lambda)) &= \frac{1}{2}\big(F(\tau_j(\lambda)) + F(-\tau_j(\lambda))\big)\\
F_o(\tau_j(\lambda)) &= \frac{1}{2}\big(F(\tau_j(\lambda)) - F(-\tau_j(\lambda))\big).
\end{align*}
Now $F_o(\tau_j(\lambda))/\tau_j(\lambda)$ is in fact a smooth function of $\tau_j^2=\lambda^2-\sigma_j^2$, and hence a smooth function of $\lambda$.  Then, 
integrating by parts, 
\begin{align*}
\left |\int_0^\infty e^{\pm i\lambda t} \frac{F_o(\tau_j(\lambda))}{\tau_j(\lambda)}  d\lambda \right | &  \leq t^{-k_0} \left \| \frac{d^{k_0}}{d\lambda^{k_0}}  \frac{F_o(\tau_j(\lambda))}{\tau_j(\lambda)} \right\|_{L^1}\\ & 
\leq Ct^{-k_0} (1+\sigma_j)^{k_0} \sum_{k\leq 2k_0+1}\sup_\tau | D_\tau^k F(\tau)| .
\end{align*}

To evaluate the integral $\int_0^\infty e^{\pm i t \lambda}\frac{ F_e(\tau_j(\lambda))}{\tau_j(\lambda)}d\lambda$, we make a change of  variables. For $\lambda\in [\sigma_j,\infty)$, 
$\tau_j(\lambda)\in [0,\infty)$ and we use the variable $\tau'=\tau_j$; for $\lambda\in [0,\sigma_j], $ $ \tau_j(\lambda)\in i[0,\infty)$ and we use the 
variable $\tau'=-i\tau _j$.  Hence 
\begin{multline}\label{eq:Fochange1}
\int_0 ^\infty e^{\pm i t \lambda}\frac{ F_e(\tau_j(\lambda))}{\tau_j(\lambda)}d\lambda \\= \int_0^\infty e^{\pm it\sqrt{(\tau')^2 +\sigma_j^2}} 
\frac{F_e(\tau')}{\sqrt{(\tau')^2+\sigma_j^2}}d\tau'
- i\int_0^{\sigma_j} e^{\pm it\sqrt{\sigma_j^2-(\tau')^2}} \frac{F_e(i\tau')}{\sqrt{\sigma_j^2-(\tau')^2}}d\tau'.
\end{multline}
For the first integral on the right-hand side of (\ref{eq:Fochange1}) we perform a change of variable in order to be able to track dependence on $\sigma_j$.
Using 
$\tau'=\sigma_j \tau$, we have
\begin{align}
\int_0^\infty e^{\pm it\sqrt{(\tau')^2 +\sigma_j^2}} \frac{F_e(\tau') }{\sqrt{(\tau')^2+\sigma_j^2}}d\tau'
& =  \int_0^\infty e^{\pm it\sigma_j \sqrt{\tau^2 +1}}\frac{ F_e(\sigma_j \tau )}{\sqrt{\tau^2+ 1}}d\tau \nonumber \\
& = \frac{1 }{2}\int_{-\infty}^\infty e^{\pm it\sigma_j \sqrt{\tau^2 +1}}\frac{ F_e(\sigma_j \tau )}{\sqrt{\tau^2+ 1}}d\tau
\end{align}
where for the second equality we have used that the integrand is even in $\tau$.
For this integral, we may use the method of stationary phase.  
Note that the only stationary point is at $\tau=0$.
By \cite[Theorem 7.7.5]{ho1}, we have that there are constants $\tilde{b}_{k\pm}$, depending on $F_e$ and $\sigma_j$, so that 
\begin{multline}
\left| \int_0^\infty e^{\pm it\sigma_j \sqrt{\tau^2 +1}}
\frac{ F_e(\sigma_j \tau )}{\sqrt{\tau^2+ 1}}d\tau
-  \left( \sigma_j t\right)^{-1/2} e^{\pm i\sigma_j t} \sum_{k=0}^{k_0-1} \tilde{b}_{k\pm }(\sigma_j t)^{-k}\right| \\
\leq C  (\sigma_j t)^{-k_0} \sum_{|\alpha| \leq 2k_0}  \sup_\tau\left| D^\alpha_\tau \left(\frac{ F_e(\sigma_j \tau) }{\sqrt{\tau^2+ 1}}\right)\right |.
\end{multline}
Moreover, the $\tilde{b}_{k\pm }$ are determined by derivatives with respect 
to $\tau$ of  $F_e(\sigma_j \tau) /{\sqrt{\tau^2+ 1}}$ of order less than or equal to $2k$, evaluated at $\tau =0$. 
The coefficient $\tilde{b}_{0\pm }=F(0)\sqrt{\pi/ 2}e^{\pm i\pi/4}$.  By allowing the constant to depend on $k_0$, we may bound 
$$\sum_{|\alpha| \leq 2k_0}  \sup_\tau \left| D^\alpha_\tau \left( \frac{F_e(\sigma_j \tau)}{\sqrt{\tau^2+ 1}}\right)\right | \leq 
C_{k_0}\sum_{k  \leq 2k_0}  \sup_\tau \left| \sigma_j^k  F^{(k)}(\tau) \right|.$$
A similar computation gives a similar expansion for the second integral on the right-hand side of (\ref{eq:Fochange1}) since the support of $F_e(i\tau')$ is small enough that 
$1/\sqrt{\sigma_j^2-(\tau')^2}$ is smooth on the support of $F_e$.
We note that $k=0$ coefficient for the expansion of the second term on the 
right-hand side of (\ref{eq:Fochange1})
(including the factor of $-i$ in front) is $-i \sqrt{\pi/2} F(0) e^{\mp i \pi /4}$.    This finishes the 
proof of (\ref{eq:statphaseexpl}), and shows that the integral on the left in (\ref{eq:statphaseexp2}) has a similar expansion.

To complete the proof of (\ref{eq:statphaseexp2}) it suffices to show that the coefficients in the expansion
are $0$.  We shall give two proofs of this.  For the 
first,  we return to (\ref{eq:Fochange1}), but with the ``$+$'' sign, writing  
\begin{multline}\label{eq:show0}
\int_0 ^\infty e^{ i t \lambda}\frac{ F_e(\tau_j(\lambda))}{\tau_j(\lambda)}d\lambda \\= \int_0^\infty e^{ it\sqrt{(\tau')^2 +\sigma_j^2}} 
\frac{F_e(\tau')}{\sqrt{(\tau')^2+\sigma_j^2}}d\tau'
- i\int_0^{\sigma_j} e^{ it\sqrt{\sigma_j^2-(\tau')^2}} \frac{F_e(i\tau')}{\sqrt{\sigma_j^2-(\tau')^2}}d\tau'\\
= \frac{1}{2} \left( \int_{-\infty}^\infty e^{ it\sqrt{(\tau')^2 +\sigma_j^2}} 
\frac{F_e(\tau')}{\sqrt{(\tau')^2+\sigma_j^2}}d\tau'
- i\int_{-\sigma_j}^{\sigma_j} e^{ it\sqrt{\sigma_j^2-(\tau')^2}} \frac{F_e(i\tau')}{\sqrt{\sigma_j^2-(\tau')^2}}d\tau' \right).
\end{multline}
As previously for the second equality we have used that the integrands are even in $\tau'$.  Setting 
$g(\tau')=((\tau')^2+\sigma_j^2))^{1/2}- (\sigma_j+(\tau')^2/2))$  we find from using the explicit
expression for the stationary phase coefficients (see, for example, \cite[Theorem 7.7.5]{ho1}) and 
summing the contributions from the two integrals  that
the coefficients in the asymptotic expansion of (\ref{eq:show0}) are linear combinations of
\begin{equation}\label{eq:munu}
e^{i\pi/4} D^{2\nu}_{\tau'}\left( \frac{g^\mu(\tau') F_e(\tau')}{\sqrt{(\tau')^2+\sigma_j^2}}\right)_{\restrict_{\tau'=0}} -
i e^{-i\pi/4} (-1)^\nu D^{2\nu}_{\tau'}
\left( \frac{g^\mu(i\tau') F_e(i\tau')}{\sqrt{(i\tau')^2+\sigma_j^2}}\right)_{\restrict{\tau'=0}}
\end{equation}
for $\nu,\mu\in \Natural_{0}.$
Since if $h$ is a smooth function in a complex neighborhood of the origin, then
$D^{2\nu}_{\tau'}h(i\tau')\restrict_{\tau'=0}= (i)^{2\nu}D^{2\nu}_{\tau'}h(\tau')\restrict_{\tau'=0}$, we see that 
the quantity in (\ref{eq:munu}) is $0$, and the 
sum of the terms coming from the stationary phase expansions in (\ref{eq:show0}) is $0$.

\begin{figure}[h] \label{f:contours}
\labellist
\small 
\pinlabel $\gamma_1$ [l] at 55 18
\pinlabel $\gamma_2$ [l] at 235 18
\endlabellist
 \includegraphics[width=12cm]{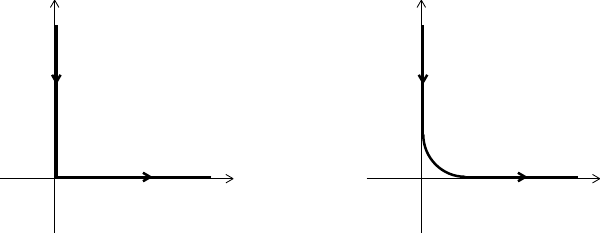} 
 \caption{The contours of integration $\gamma_1$ and $\gamma_2$.}\label{f:taucorner}
\end{figure}

Now we outline an alternate, perhaps  more intuitive, proof that the 
sum of the stationary phase coefficients in arising from the right-hand side of  (\ref{eq:show0}) is $0$.
The middle expression in (\ref{eq:show0}) may 
be written
\begin{equation}\label{eq:gammaint}
\int_{\gamma_1} e^{i t \sqrt{z^2+\sigma_j^2}}\frac{F_e(z)}{\sqrt{z^2+\sigma_j^2}}dz
\end{equation}
where $\gamma _1$ is as in Figure \ref{f:contours}: the path that goes down the positive
imaginary axis to the origin, and then to infinity along the positive real 
axis.  We understand the square root to be analytic in the 
 closed first quadrant away from $i\sigma_j$ and to be  positive on
the positive real axis; this ensures $\Im \sqrt{z^2+\sigma_j^2}> 0$
in the open first quadrant.
 {\em If} $F_e$ were analytic in a neighborhood of the origin, then by Cauchy's 
Theorem  we could write
$$\int_{\gamma_1} e^{i t \sqrt{z^2+\sigma_j^2}}\frac{F_e(z)}{\sqrt{z^2+\sigma_j^2}}dz
= \int_{\gamma_2} e^{i t \sqrt{z^2+\sigma_j^2}}\frac{F_e(z)}{\sqrt{z^2+\sigma_j^2}}dz
$$
where $\gamma_2$ is 
smooth,  differs from $\gamma_1$ only in a suitably small neighborhood 
of the origin, and is contained in the closure of the first quadrant; see
Figure \ref{f:contours}.  
Since for $t>0$, $|e^{i\sqrt{z^2+\sigma_j^2} t}|\leq 1$ on $\gamma_2$
and (with suitably parameterized $\gamma_2$) the phase has no stationary points on $\gamma_2$, by repeated integration
by parts we can see that the integral in (\ref{eq:show0}) is $O(t^{-k})$, any
$k\in \Natural$, as $t\rightarrow \infty$.

If $F_e$ is only smooth, not complex analytic, in a neighborhood of the origin,
write $F_e=\tilde{\psi}T_{2k} + (F_e-\tilde{\psi}T_{2k})$ where 
$T_{2k}$ is the $2k$th Taylor polynomial of $F_e$ at $0$ and $\tilde{\psi}\in C_c^{\infty}(\Complex)$ is $1$ in a neighborhood of the origin.  Then
the argument outlined above may be applied to the 
integral with $\tilde{\psi}T_{2k}$ if $\gamma_2$ differs 
from $\gamma_1$ only on the set where $\tilde{\psi}$ is $1$.  Since $ (F_e-\tilde{\psi}T_{2k})$
vanishes to order $2k+1$ at the origin, we may integrate by parts $k$ times
to see that 
$$ \int_{\gamma_1} e^{i t \sqrt{z^2+\sigma_j^2}}\frac{F_e(z)-\tilde{\psi}(z)T_{2k}(z) }{\sqrt{z^2+\sigma_j^2}}dz
= O(t^{-k}).$$
\end{proof}

The second argument for (\ref{eq:statphaseexp2}) also gives an 
intuitive reason for the difference between (\ref{eq:statphaseexpl}) and 
 (\ref{eq:statphaseexp2}).   In place of (\ref{eq:gammaint}) we have instead 
for (\ref{eq:statphaseexpl}) the integral 
$\int _{\gamma_1} e^{-it\sqrt{z^2+\sigma_j^2}} F_e(z)(z^2+\sigma_j^2)^{-1/2}dz$.
If $F_e$ is analytic near the origin, we can, as in the argument above,
use a contour deformation argument to deform the contour of integration to
$\gamma_2$.  But for $z$ in the open first quadrant,
 $e^{-it\sqrt{z^2+\sigma_j^2}}$
is exponentially increasing as $t\rightarrow \infty$.  If we instead
 deform $\gamma_1$ to avoid the origin and the first quadrant,
the deformed path must have portions in each of quadrants $2$, $3$, and $4$.
But $e^{-it\sqrt{z^2+\sigma_j^2}}$ is exponentially increasing as 
$t\rightarrow \infty$ if $z$ is in the open third quadrant.

We state another lemma, with results similar to those of Lemma \ref{l:allinonestatphase}.  Note that this differs from Lemma \ref{l:allinonestatphase}
in the domain of integration, the assumptions on where $F$ and $G$ are 
smooth, and the less explicit bound on the error.  We remark that the powers $t^{-k/2}$,
rather than $t^{-k}$ of Lemma \ref{l:allinonestatphase} are a consequence of
the fact that after changing variable $\tau=\sqrt{\lambda^2-\sigma_j^2}$,  $\tau=0$ is
both a stationary point of the phase and an endpoint of integration.
\begin{lemma}\label{l:bdrstatphase}Let $\mcx$ be a Banach space
and let $\sigma_j>0$.  Let $F\in C_c^\infty([0,\sigma_j/2);\mcx)$ and 
$G\in  C_c^\infty(i[0,\sigma_j/2);\mcx)$.  Then given $k_0\in \Natural$
there are 
$\alpha_{k,\pm},\; \beta_{k,\pm}\in \mcx$, $k\in 0,1,...,2k_0-2$, 
$C=C(F,G,\sigma_j,k_0)>0$
 so that
\begin{equation}\label{eq:spb1}
\left \| \int_{\sigma_j}^\infty e^{\pm i \lambda t} \frac{F(\tau_j(\lambda))}{\tau_j(\lambda)}d\lambda- t^{-1/2}e^{\pm i t \sigma_j} \sum_{k=0}^{2k_0-2} \alpha_{k,\pm}
t^{-k/2}\right\|_{\mcx} \leq C t^{-k_0}, \; t>0
\end{equation}
and 
\begin{equation}\label{eq:spb2}
\left \| \int_0^{\sigma_j} e^{\pm i \lambda t} \frac{G(\tau_j(\lambda))}{\tau_j(\lambda)}d\lambda- t^{-1/2}e^{\pm i t \sigma_j} \sum_{k=0}^{2k_0-2} \beta_{k,\pm}
t^{-k/2}\right\|_{\mcx} \leq C t^{-k_0}, \; t>0.
\end{equation}
Here the $\alpha_{k,\pm}$ (respectively $\beta_{k,\pm}$) are determined by 
$\sigma_j$ and the derivatives of $F$ (respectively $G$) of order at most 
$k$, evaluated at $0$.
\end{lemma}
\begin{proof} We prove only (\ref{eq:spb1}), as the proof of (\ref{eq:spb2}) 
is almost identical.  
By introducing  $\tau=\tau_j(\lambda)$ as the variable of integration, 
\begin{equation}
\int_{\sigma_j}^\infty e^{\pm i \lambda t} \frac{F(\tau_j(\lambda))}{\tau_j(\lambda)}d\lambda = \int_0^\infty e^{\pm i t\sqrt{\tau^2+\sigma_j^2} }\frac{F(\tau)}{\sqrt{\tau^2+\sigma_j^2}}d\tau.
\end{equation}
Then an application of \cite[Section 2.9]{Erd} proves 
(\ref{eq:spb1}), with coefficients $\alpha_{k,\pm}$ determined by 
$\sigma_j$ and derivatives of $F$, evaluated at $0$, of order at most $k$.
\end{proof}

\vspace{2mm}
\noindent
{\bf Acknowledgments.}   The authors are grateful to the Simons Foundation for its support through the Collaboration Grants for Mathematicians program. KD was also supported by the National Science Foundation through  Grant DMS-1708511. 
It is a pleasure to thank Jeremy Marzuola and Maciej Zworski for helpful conversations.

\bibliographystyle{alpha}
\bibliography{bibtexfiledatchev}

\end{document}